\documentclass[reqno]{amsart}
\usepackage[utf8]{inputenc}
\usepackage[T1]{fontenc}
\usepackage[english]{babel}
\usepackage{lmodern}

\usepackage[backend=biber,hyperref=true,citestyle=numeric,sortcase=true,sorting=anyt,sortfirstinits=true,backref=false,maxnames=20,isbn=false,doi=false]{biblatex}
\usepackage{amssymb}
\usepackage{mathtools}
\usepackage{enumitem}
\usepackage[ampersand]{easylist}
\usepackage{url}
\usepackage[plainpages=false,pdfpagelabels,colorlinks=true,citecolor=blue,hypertexnames=false,pdftex=true,unicode]{hyperref}

\usepackage{ifthen}
\def\fromstep#1#2{\ifthenelse{\stepno > #1 \OR \stepno = #1}{#2}{}}
\def\fromtostep#1#2#3{\ifthenelse{\( \stepno > #1 \OR \stepno = #1 \) \AND \( \stepno < #2 \OR \stepno = #2 \)}{#3}{}}
\def\onlyatstep#1#2{\ifthenelse{\stepno = #1}{#2}{}}
\def\drawLraising#1#2{%
  \draw[->](#1.7,#2) arc (0:45:.7);
  \draw[->](#1,#2)--+(1,0);
}
\def\drawDraising#1#2{%
  \draw[<-](#1.7,#2) arc (0:-45:.7);
  \draw[->](#1,#2)--+(1,-1);
}
\def\exHsuffix#1{\def\stepno{#1}
\begin{mypic}
\foreach \x in {-3,-2,-1,0}
  \foreach \y in {0,1,2,3}
    \fill(\x,\y) circle[radius=2pt];
\draw[->](-4,0)--(1,0);
\draw[->](0,0)--(0,4);
\fromstep9{\draw[thick,->](0,0)--(-1,0);}
\fromstep8{\draw[thick,->](-1,0)--(-1,1);}
\fromstep7{\draw[thick,->](-1,1)--(-2,1);}
\fromstep6{\draw[thick,->](-2,1)--(-2,2);}
\fromstep5{\draw[thick,->](-2,2)--(-2,3);}
\fromstep4{\draw[thick,->](-2,3)--(-3,3);}
\fromstep3{\draw[thick,->](-3,3)--(-2,2);}
\fromstep2{\draw[thick,->](-2,2)--(-1,1);}
\fromstep1{\draw[thick,->](-1,1)--(0,0);}
\end{mypic}
}
\def\exMsuffix#1{\def\stepno{#1}
\begin{mypic}
\foreach \x in {0,1,2,3,4,5,6,7,8,9}
  \foreach \y in {0,1,2,3}
    \fill(\x,\y) circle[radius=2pt];
\draw[->](0,0)--(10,0);
\draw[->](0,0)--(0,4);
\fromstep9{\draw[thick,->](0,0)--(1,0);}
\fromstep8{\draw[thick,->](1,0)--(2,1);}
\fromstep7{\draw[thick,->](2,1)--(3,1);}
\fromstep6{\draw[thick,->](3,1)--(4,2);}
\fromstep5{\draw[thick,->](4,2)--(5,3);}
\fromstep4{\draw[thick,->](5,3)--(6,3);}
\fromstep3{\draw[thick,->](6,3)--(7,2);}
\fromstep2{\draw[thick,->](7,2)--(8,1);}
\fromstep1{\draw[thick,->](8,1)--(9,0);}
\end{mypic}
}
\def\exQsuffix#1{\def\stepno{#1}
\begin{mypic}
\foreach \x in {0,1,2}
  \foreach \y in {0,1,2,3}
    \fill(\x,\y) circle[radius=2pt];
\draw[->](0,0)--(3,0);
\draw[->](0,0)--(0,4);
\fromstep9{\draw[thick,->](0,0)--(0,1);}
\fromstep8{\draw[thick,->](0,1)--(0,2);}
\fromstep7{\draw[thick,->](0,2)--(1,1);}
\fromstep6{\draw[thick,->](1,1)--(1,2);}
\fromstep5{\draw[thick,->](1,2)--(1,3);}
\fromstep4{\draw[thick,->](1,3)--(2,2);}
\fromstep3{\draw[thick,->](2,2)--(1,2);}
\fromstep2{\draw[thick,->](1,2)--(0,2);}
\fromstep1{\draw[thick,->](0,2)--(1,1);}
\end{mypic}
}
\def\exQprefix#1{\def\stepno{#1}
\begin{mypic}
\foreach \x in {0,1,2}
  \foreach \y in {0,1,2,3}
    \fill(\x,\y) circle[radius=2pt];
\draw[->](0,0)--(3,0);
\draw[->](0,0)--(0,4);
\fromstep1{\draw[thick,->](0,0)--(0,1);}
\fromstep2{\draw[thick,->](0,1)--(0,2);}
\fromstep3{\draw[thick,->](0,2)--(1,1);}
\fromstep4{\draw[thick,->](1,1)--(1,2);}
\fromstep5{\draw[thick,->](1,2)--(1,3);}
\fromstep6{\draw[thick,->](1,3)--(2,2);}
\fromstep7{\draw[thick,->](2,2)--(1,2);}
\fromstep8{\draw[thick,->](1,2)--(0,2);}
\fromstep9{\draw[thick,->](0,2)--(1,1);}
\end{mypic}
}
\def\exMprefix#1{\def\stepno{#1}
\begin{mypic}
\foreach \x in {0,1,2,3,4,5,6,7,8,9}
  \foreach \y in {0,1,2,3}
    \fill(\x,\y) circle[radius=2pt];
\draw[->](0,0)--(10,0);
\draw[->](0,0)--(0,4);
\fromstep1{\draw[thick,->](0,0)--(1,0);}
\onlyatstep2{\draw[thick,->](1,0)--(2,0);}
\fromstep3{\draw[thick,->](1,0)--(2,1);}
\onlyatstep3{\drawLraising{1}{0}}
\fromtostep{3}{7}{\draw[thick,->](2,1)--(3,0);}
\fromstep8{\draw[thick,->](2,1)--(3,1);}
\fromtostep{4}{7}{\draw[thick,->](3,0)--(4,0);}
\onlyatstep8{\draw[thick,->](3,1)--(4,1);}
\fromstep9{\draw[thick,->](3,1)--(4,2);}
\onlyatstep5{\draw[thick,->](4,0)--(5,0);}
\fromtostep{6}{7}{\draw[thick,->](4,0)--(5,1);}
\onlyatstep8{\draw[thick,->](4,1)--(5,2);}
\fromstep9{\draw[thick,->](4,2)--(5,3);}
\onlyatstep6{\drawLraising{4}{0}}
\onlyatstep6{\draw[thick,->](5,1)--(6,0);}
\onlyatstep7{\draw[thick,->](5,1)--(6,1);}
\onlyatstep8{\draw[thick,->](5,2)--(6,2);}
\fromstep9{\draw[thick,->](5,3)--(6,3);}
\onlyatstep7{\drawDraising{5}{1}}
\onlyatstep7{\draw[thick,->](6,1)--(7,0);}
\onlyatstep8{\draw[thick,->](6,2)--(7,1);}
\fromstep9{\draw[thick,->](6,3)--(7,2);}
\onlyatstep8{\drawDraising{2}{1}}
\onlyatstep8{\draw[thick,->](7,1)--(8,0);}
\fromstep9{\draw[thick,->](7,2)--(8,1);}
\onlyatstep9{\drawLraising{3}{1}}
\fromstep9{\draw[thick,->](8,1)--(9,0);}
\end{mypic}
}

\makeatletter
\def\boxlabel#1{\phantomsection(#1)\def\@currentlabel{(\unexpanded{#1})}\label{box:#1}}
\makeatother
\def\boxref#1{\ref{box:#1}}

\usepackage{multirow}
\renewcommand{\arraystretch}{1.2}

\def\myclap#1{\rule{5mm}{0mm}\clap{#1}\rule{5mm}{0mm}}

\newtheorem{theorem}{Theorem}
\newtheorem{proposition}[theorem]{Proposition}
\newtheorem{lemma}[theorem]{Lemma}
\theoremstyle{definition}
\newtheorem{example}[theorem]{Example}

\usepackage{tikz}
\usetikzlibrary{shapes,backgrounds,positioning}

\newenvironment{mypic}{\begin{tikzpicture}[scale=.35, baseline=(current bounding box.center)]}{\end{tikzpicture}}

\newcommand{\transition}[5]{\draw(#1.5,-.7) node{\tiny#2}; \draw(#1,-1.5) node{\tiny#3}; \draw(#1,-2) node{\tiny#4}; \draw(#1.5,-2.7) node{#5};}

\def\dir#1#2#3{\begin{mypic}
  \color{orange}
  \ifx b#3\color{black}\else\fi
  \ifx n#3\color{white}\else\fi
  \draw(-.1,-.1)--(-.1,1.1)--(1.1,1.1)--(1.1,-.1)--(-.1,-.1);
  \color{orange}
  \ifx s#2\color{black}\draw(.5,.5) node{#1};\else
    \ifx b#2\color{black}\else\fi
    \ifx l#2\color{blue}\else\fi
    \ifx r#2\color{red}\else\fi
    \ifx L#2\color{black}\else\fi
    \ifx R#2\color{black}\else\fi
    \ifx L#2\else\ifx R#2\else
      \ifx0#1\draw[thick,->](.5,0)--(.5,1);\else\fi
      \ifx1#1\draw[thick,->](0,0)--(1,1);\else\fi
      \ifx2#1\draw[thick,->](0,.5)--(1,.5);\else\fi
      \ifx3#1\draw[thick,->](0,1)--(1,0);\else\fi
      \ifx4#1\draw[thick,->](.5,1)--(.5,0);\else\fi
      \ifx5#1\draw[thick,->](1,1)--(0,0);\else\fi
      \ifx6#1\draw[thick,->](1,.5)--(0,.5);\else\fi
      \ifx7#1\draw[thick,->](1,0)--(0,1);\else\fi
      \ifx u#1\draw[thick,->](0,0)--(1,1);\fill(.2,.8) circle[radius=5pt];\else\fi
      \ifx l#1\draw[thick,->](0,.5)--(1,.5);\fill(.5,.9) circle[radius=5pt];\else\fi
      \ifx d#1\draw[thick,->](0,1)--(1,0);\fill(.8,.8) circle[radius=5pt];\else\fi
    \fi\fi
    \ifx L#2
      \ifx1#1\draw[-,double=black,double distance=1pt](0,0)--(.89,.89);\draw[thick,->](.999,.999)--(1,1);\else\fi
      \ifx2#1\draw[-,double=black,double distance=1pt](0,.5)--(.85,.5);\draw[thick,->](.99,.5)--(1,.5);\else\fi
      \ifx3#1\draw[-,double=black,double distance=1pt](0,1)--(.89,.11);\draw[thick,->](.999,.001)--(1,0);\else\fi
      \ifx u#1\draw[-,double=black,double distance=1pt](0,0)--(.89,.89);\draw[thick,->](.999,.999)--(1,1);\fill(.2,.8) circle[radius=5pt];\else\fi
      \ifx l#1\draw[-,double=black,double distance=1pt](0,.5)--(.85,.5);\draw[thick,->](.99,.5)--(1,.5);\fill(.5,.9) circle[radius=5pt];\else\fi
      \ifx d#1\draw[-,double=black,double distance=1pt](0,1)--(.89,.11);\draw[thick,->](.999,.001)--(1,0);\fill(.8,.8) circle[radius=5pt];\else\fi
    \else\fi
    \ifx R#2
      \ifx1#1\draw[-,double=white,double distance=1pt](0,0)--(.89,.89);\draw[thick,->](.999,.999)--(1,1);\else\fi
      \ifx2#1\draw[-,double=white,double distance=1pt](0,.5)--(.85,.5);\draw[thick,->](.99,.5)--(1,.5);\else\fi
      \ifx3#1\draw[-,double=white,double distance=1pt](0,1)--(.89,.11);\draw[thick,->](.999,.001)--(1,0);\else\fi
      \ifx u#1\draw[-,double=white,double distance=1pt](0,0)--(.89,.89);\draw[thick,->](.999,.999)--(1,1);\fill(.2,.8) circle[radius=5pt];\else\fi
      \ifx l#1\draw[-,double=white,double distance=1pt](0,.5)--(.85,.5);\draw[thick,->](.99,.5)--(1,.5);\fill(.5,.9) circle[radius=5pt];\else\fi
      \ifx d#1\draw[-,double=white,double distance=1pt](0,1)--(.89,.11);\draw[thick,->](.999,.001)--(1,0);\fill(.8,.8) circle[radius=5pt];\else\fi
    \else\fi
    \color{black}
  \fi
\end{mypic}}

\def\ldir#1{\raisebox{.25em}{\dir#1}}
\def\sldir#1{\resizebox{!}{1mm}{\dir#1}}

\def\pdir#1{\raisebox{.25em}{\begin{mypic}
  \color{black}
  \draw(-.1,-.1)--(-.1,1.1)--(1.1,1.1)--(1.1,-.1)--(-.1,-.1);
  \ifx p#1\draw[thick,->](.8,.2)--(.8,1);\else\fi
  \ifx m#1\draw[thick,->](.8,.2)--(.2,.4);\else\fi
  \ifx 0#1\draw[thick,->](.8,.2)--(0,.2);\else\fi
  \ifx -#1\draw[thick,->](.8,.2)--(1,0);\else\fi
\end{mypic}}}

\def\ptandem#1#2{\begin{mypic}
\ifx e#2\draw(-.5,.5) node{err.};\else
  \draw(-2.1,-1.1)--(-2.1,2.1)--(1.1,2.1)--(1.1,-1.1)--(-2.1,-1.1);
\fi
\ifx2#1
  \ifx0#2\draw[thick,->](0,0)--(0,2);\draw(-.1,1)--(.1,1);\else\fi
  \ifx1#2\draw[thick,->](0,0)--(-1,1);\else\fi
  \ifx2#2\draw[thick,->](0,0)--(-2,0);\draw(-1,-.1)--(-1,.1);\else\fi
  \ifx-#2\draw[thick,->](0,0)--(1,-1);\else\fi
\fi
\ifx3#1
  \ifx0#2\draw[thick,->](0,0)--(0,2);\draw(-.1,.666)--(.1,.666);\draw(-.1,1.333)--(.1,1.333);\else\fi
  \ifx1#2\draw[thick,->](0,0)--(-.666,1.333);\draw(-.399,.633)--(-.267,.699);\else\fi
  \ifx2#2\draw[thick,->](0,0)--(-1.333,.666);\draw(-.633,.399)--(-.699,.267);\else\fi
  \ifx3#2\draw[thick,->](0,0)--(-2,0);\draw(-.666,-.1)--(-.666,.1);\draw(-1.333,-.1)--(-1.333,.1);\else\fi
  \ifx-#2\draw[thick,->](0,0)--(.666,-.666);\else\fi
\fi
\ifx g#1
  \ifx0#2\draw[thick,->](0,0)--(0,2);\draw(-1,1) node{\tiny$(0,p)$};\else\fi
  \ifx1#2\draw[thick,->](0,0)--(-1.5,.5);\draw(-.5,1) node{\tiny$({-}m{-}1,p{-}m{-}1)$};\else\fi
  \ifx-#2\draw[thick,->](0,0)--(.25,-.25);\else\fi
  \ifx s#2\draw(-.25,.6) node{$\bar a$};\else\fi
\fi
\end{mypic}}

\def\trsIIIstep in:#1/#2/#3, h:#4->#5, v:#6->#7, out:#8/#9 ({\begin{mypic}%
  \draw(-1,-4.7)--(2,-4.7)--(2,4.2)--(-1,4.2)--(-1,-4.7);
  \draw(.5,3.5) node{\dir#1};
  \draw(.5,2) node{#2};
  \draw(.5,.5) node{\dir#3};
  \draw(-.2,-.7) node{\scriptsize#4};
  \draw(1.2,-.7) node{\scriptsize#5};
  \draw(-.2,-1.5) node{\scriptsize#6};
  \draw(1.2,-1.5) node{\scriptsize#7};
  \draw(.5,-2.5) node{#8};
  \draw(.5,-4) node{\dir#9};
\trsIIIstepcontinued}

\def\trsIIIstepcontinued#1).{\draw(.5,-5.5) node{\scriptsize(#1)};\end{mypic}}

\def\justcolor#1#2#3{%
  \color{orange}
  \def\mc{}
  \ifx s#2\color{black}\else\fi
  \ifx b#2\color{black}\else\fi
  \ifx l#2\color{blue}\else\fi
  \ifx r#2\color{red}\else\fi
  \ifx L#2\color{black}\def\mc{\boldsymbol}\else\fi
  \ifx R#2\color{black}\def\mc{\mathbb}\else\fi
}

\def\trsVIstep in:#1/#2, h:#3->#4, v:#5->#6, out:#7, label:#8.{\begin{mypic}%
  \draw(-1.2,-3.4)--(2.2,-3.4)--(2.2,2.5)--(-1.2,2.5)--(-1.2,-3.4);
  \justcolor#2%
  \draw(.5,2) node{#1};
  \draw(.5,.5) node{\dir#2};
  \color{black}
  \draw(-.2,-.7) node{\scriptsize#3};
  \draw(1.2,-.7) node{\scriptsize#4};
  \draw(-.2,-1.5) node{\scriptsize#5};
  \draw(1.2,-1.5) node{\scriptsize#6};
  \draw(.5,-2.7) node{\dir#7};
  \draw(.5,-4.1) node{\scriptsize#8};
\end{mypic}}

\def\trsptandemcont cond:#1, label:#2.{%
  \draw(1,-8) node{\scriptsize#1};
  \draw(1,-9.2) node{#2};
}
\def\trsptandem p:#1, in:#2, h:#3->#4, v:#5->#6, out:#7/#8, cond:#9.{\begin{mypic}%
  \draw(-2,-7.1)--(-2,1.1)--(4,1.1)--(4,-7.1)--(-2,-7.1);
  \draw(1,.5) node{#2};
  \draw(-.5,-.7) node{\scriptsize#3};
  \draw(2.5,-.7) node{\scriptsize#4};
  \draw(-.5,-1.5) node{\scriptsize#5};
  \draw(2.5,-1.5) node{\scriptsize#6};
  \draw(1,-4) node{\ptandem#1#7};
  \draw(1,-6.5) node{#8};
  \trsptandemcont cond:#9.
  \end{mypic}}

\newcounter{mpw}

\usepackage{bbold}
\def\bN{\mathbb N}
\def\bZ{\mathbb Z}
\usepackage{amsbsy}

\def\sLR#1{#1^{\text{\tiny LR}}}
\def\sRL#1{#1^{\text{\tiny RL}}}
\def\mysmash#1#2{\mathchoice{#1^{#2}}{#1^{\smash{#2}}}{#1^{\smash{#2}}}{#1^{\smash{#2}}}}
\def\myconstructions#1#2#3#4{\mysmash{#2_{\text{\tiny#1},#4}}{(#3)}}
\def\LR{\myconstructions{LR}}
\def\RL{\myconstructions{RL}}
\def\myDconstructions#1#2#3#4{\mysmash{\Delta_{#3#1#4}}{(#2)}}
\def\DLR{\myDconstructions{\rightarrow}}
\def\DRL{\myDconstructions{\leftarrow}}

\def\prefH{H'}
\def\wt{\operatorname{wt}}

\def\hyph{\discretionary{-}{}{}}

\newcommand{\raisabledec}{\bullet}
\newcommand{\raisablelink}[3]{\overset{#2,#3}{(#1)}}
\def\rlklow#1#2#3{$\raisablelink{#1}{#2}{\ifx -#3-1\else#3\fi}$}
\def\rlk#1{\rlklow#1}

\newenvironment{ouralgo}%
  {\begin{quotation}\hrulefill\begin{enumerate}[ref=(\arabic*)]}%
  {\end{enumerate}\vskip-2mm\hrulefill\end{quotation}}
\def\midalgo{\end{enumerate}\vskip-2mm\hrulefill\begin{enumerate}[ref=(\arabic*)]}
\def\Input#1{\item[Input:]#1}
\def\Output#1{\item[Output:]#1}

\title
[Bijections Between Łukasiewicz Walks and Generalized Tandem Walks]
{Bijections Between Łukasiewicz Walks\\ and Generalized Tandem Walks}
\author{Frédéric Chyzak}
\address{INRIA, France}
\email{frederic.chyzak@inria.fr}
\author{Karen Yeats}
\address{Combinatorics and Optimization, University of Waterloo, Canada}
\email{kayeats@uwaterloo.ca}

\begin{document}
\begin{abstract}
In this article, we study the enumeration by length
of several walk models on the square lattice.
We obtain bijections
between walks in the upper half-plane returning to the $x$-axis
and walks in the quarter plane.
A recent work by Bostan, Chyzak, and Mahboubi has given
a bijection for models using small north, west, and south-east steps.
We adapt and generalize it to a bijection
between half-plane walks using those three steps in two colours
and a quarter-plane model
over the symmetrized step set consisting of
north, north-west, west, south, south-east, and east.
We then generalize our bijections to certain models with large steps:
for given~$p\geq1$, a bijection is given
between the half-plane and quarter-plane models
obtained by keeping the small south-east step
and replacing the two steps north and west of length~1
by the $p+1$ steps of length~$p$
in directions between north and west.
This model is close to, but distinct from, the model of generalized tandem walks
studied by Bousquet-Mélou, Fusy, and Raschel.
\end{abstract}
\maketitle
\centerline{\small\today}
\bigskip\bigskip\bigskip

\section{Introduction}

\emph{Tandem walks} are walks on the square lattice~$\bZ^2$
that start at the origin and are confined to the quarter plane,
with all steps taken in the set
\[ S_1 = \{(-1,0),(0,1),(1,-1)\} , \]
that is, the set consisting of the small steps west, north, and south-east.
The name originates from queuing theory:
tandem walks indeed represent the behaviour of two M/M/1 queues in tandem.
See for example~\cite[Section~4.7]{FayolleIasnogorodskiMalyshev-1999-RWQ}.
Tandem walks are one simple model
in a larger family of models of walks on the quarter plane
with a prescribed step set.
In 2010, Bousquet-Mélou and Mishna~\cite{BousquetMelouMishna-2010-WSS}
identified 79~essentially different models of walks
with step sets consisting of nontrivial steps in $\{-1,0+1\}^2$. These models have since then received a vast amount of attention in the combinatorial litterature.
From the 79, another class of interest to the present work is obtained
by symmetrizing the step set with respect to the line $x = y$,
leading to a generalization of tandem walks to walks over the step set
\[ S_{1,\text{sym}} = S_1 \cup \{(0,-1),(1,0),(-1,1)\} , \]
that is, the six small steps
west, north, south-east, south, east, and north-west.
Both classes are known to have algebraic generating series
when enumerated by length (the number of steps).
For the former, Gouyou-Beauchamps gave in 1989
an implicit bijection in terms of Motzkin walks and standard Young tableaux%
~\cite{GouyouBeauchamps-1989-SYT}.
For the latter, Bousquet-Mélou and Mishna gave in 2010
an explicit algebraic calculation of the generating series,
by an averaging of a functional equation satisfied by the series%
~\cite{BousquetMelouMishna-2010-WSS}.
In addition, the last two authors remarked
that the symmetric duplication of the step set
used to define the second class
is likely to be mirrored,
in a manner not yet explained at the time,
by the factor~$2^n$ that they observed in their enumeration.
They asked for a bijective explanation of the phenomenon.

While the bijection suggested
by Gouyou-Beauchamps~\cite{GouyouBeauchamps-1989-SYT}
proceeds by composing several known bijections,
more recent works by Eu and collaborators%
~\cite{Eu-2010-SST,EuFuHouHsu-2013-SYT}
introduced a more direct transformation
that operates by searching for patterns on the walk viewed as a word
and successively rewriting their occurrences until none is found.
The patterns used are arbitrarily distant pairs of matching parentheses.
Very recently, Bostan, Chyzak, and Mahboubi revisited the problem
to get bijections described
by single-pass algorithms~\cite{BostanChyzakMahboubi--BBC}.%
\footnote{In the absence of a written text yet,
one can consult~\cite{Chyzak-2016-ALB} for an earlier, related presentation.}
They developed an approach inspired by the theory of automata:
the algorithm for each of the two converse bijections consists
of a single unidirectional traversal of the input walk,
read step by step by a transducer
that additionally maintains two nonnegative integer counters%
\footnote{The original presentation in~\cite{BostanChyzakMahboubi--BBC}
uses automata with infinitely many states and no auxiliary memory,
but, in view of the generalization to come,
we prefer an equivalent presentation with finitely many states
and counters that encode the numbering of the states
considered in~\cite{BostanChyzakMahboubi--BBC}.}.
The evolution of the counters follows very local rules.
This more explicit presentation of the bijection also allows for
following information about the final position.

The present work adapts the use of automata in~\cite{BostanChyzakMahboubi--BBC}
to study two models that generalize tandem walks.
On the one hand (Section~\ref{sec:six-step}),
a suitable interlacing of two copies of the bijection
between Motzkin walks and tandem walks
provides a bijection
between bicoloured Motzkin walks and 6-step walks in the quarter plane.
This leads to a combinatorial explanation
to the $2^n$~factor for the symmetric step-set model.
On the other hand (Section~\ref{sec:p-tandem}),
proceeding beyond Motzkin walks and tandem walks,
we introduce walks with large nonnegative steps.
Given a nonnegative integer~$p$,
we consider the class of walks of the upper half plane
that use the steps $(1, i)$ for $-1 \leq i \leq p$
and return to the $x$-axis,
which we call \emph{$p$-Łukasiewicz walks}.
We also introduce the class of walks of the quarter plane
that use the small step $(1, -1)$
and the large steps $(-i, p-i)$ for $0 \leq i \leq p$,
which we call \emph{$p$-tandem walks}.
Our principal result is to obtain a bijection between these two classes
(Theorem~\ref{thm:bij-Phi-p-Psi-p}).

Łukasiewicz walks have been studied for a long time in several forms.
They were initially introduced in logic by Łukasiewicz,
see, e.g., the book by Rosenbloom~\cite{Rosenbloom-1950-EML},
and introduced to combinatorics seemingly for the first time
by Raney~\cite{Raney-1960-FCP}.
Traditionally, they are defined as words
that correspond to walks confined to the upper half plane,
with the exception of their very last step,
which has to be $(1,-1)$ and passes below the axis.
Keeping the last SE step implicit makes a bijection to our class more direct,
and our $p$-Łukasiewicz walks have been studied for instance
by Brak, Iliev, and Prellberg
under the name of “$(0, p)$-restricted Łukasiewicz paths”%
~\cite{BrakIlievPrellberg-2011-IFA},
as models of polymer adsorption.
The traditional form is better suited to studying
their recursive, tree-like structure,
and applications to random sampling and random generation of trees
are to be found for instance in~\cite{BodiniDavidMarchal-2016-RBO}.

Although we claim a proximity of our work to automata theory,
we will not use more than its vocabulary in most of the present article.
We will indeed describe the procedure implementing our bijections
in a more relaxed way,
hiding their automatic content in a presentation by “boxes”
(see Sections~\ref{sec:3-step-model}--\ref{sec:six-step}).
The interested reader will wait until Section~\ref{sec:automata}
to see a formal comparison of our presentation to classical automata theory.

A combinatorial explanation of the factor~$2^n$
in the symmetric step-set model
was already given by Yeats~\cite{Yeats-2014-BBC}.
Her approach shares similarities
with the work by Eu, Fu, Hou, and Hsu
for the simpler tandem model.
For comparison's sake, we provide it here as well,
together with a generalization to $p$-tandem walks
(Section~\ref{sec:raising}).

Remarkably enough, similar bijective results appeared recently
in an ongoing work by Bousquet-Mélou, Fusy, and Raschel%
~\cite{BousquetMelouFusyRaschel-2017-OBC}.
There, the authors use a recent bijection
between the class of all tandem walks for all~$p$
and the class of so-called marked bipolar-oriented planar maps,
by Kenyon, Miller, Sheffield, and Wilson%
~\cite{KenyonMillerSheffieldWilson-2016-BOP}.
Our bijections are different from those based on maps,
with no clearly studied potential connection so far.
This is discussed further in Section~\ref{sec:conclusions}.

We see three main benefits to our bijections.  First of all, they are described algorithmically with only a single pass through the input word.  This makes it clear they have short, fast implementations.  The single pass algorithmic description leads to nicer, clearer proofs than the approach of \cite{Yeats-2014-BBC}, and also made the generalization to larger step sets possible.  Then, after the fact, the bijections for larger step sets could also be given an alternate explanation in the spirit of \cite{Yeats-2014-BBC}.  This gives us two ways, each with different benefits, to understand all of our bijections.

\subsection{General setup. Notation}\label{sec:notation}


All of our bijections are described algorithmically
and proceed in a similar fashion:
to compute the image of a given word~$w = w_1\dots w_n$,
a counter~$c$ is introduced with values associated to the prefixes of~$w$,
and an output word~$\bar w = \bar w_1\dots \bar w_n$
is computed as the image of~$w$
by computing $(\bar w_i, c_i)$ from $(w_{i-1}, c_{i-1})$
for successive values of~$i$.
The value of $(\bar w_i, c_i)$ depends on $(w_{i-1}, c_{i-1})$
in a deterministic way,
according to a piecewise definition
given by possible transitions.
We describe such transitions by boxes in what follows.


Steps in 2-dimensional step sets may be depicted by arrows,
which then represent vectors with integer coordinates.
Step sets made of such steps are named using the letter~$S$.
Step sets for 1-dimensional walks are subsets of~$\bZ$.
They are named using the letter~$\Sigma$.
We will freely identify such a step set~$\Sigma$
with the set of vectors~$\{1\}\times\Sigma$.
For future reference,
we summarize the use of these alphabets in the following list,
together with the specific classes of words that we will build on them.
For this description, given an alphabet~$A$,
$A^\ast$~denotes as usual the set of finite words on~$A$,
including the empty word~$\epsilon$,
and we write~$w' \leq w$ when $w'$~is a prefix of~$w$,
and $|w|_a$~for the number of occurrences of the letter~$a$ in~$w$.

\bigskip

\begin{easylist}[itemize]
& $\Sigma_1 = \{-1,0,+1\} \simeq \{\, \ldir{3bb}, \ldir{2bb}, \ldir{1bb} \,\}$.
   Motzkin words are those words on~$\Sigma_1$
   whose prefixes contain at least as many up steps as down steps:
   \begin{equation*}
   \mathcal M =
   \{ w \in \Sigma_1^\ast :
     |w|_{\sldir{1bb}} = |w|_{\sldir{3bb}} \text{ and }
     \forall w', \ w' \leq w \Rightarrow
       |w'|_{\sldir{1bb}} \geq |w'|_{\sldir{3bb}} \} .
   \end{equation*}
& $S_1 = \{\, \ldir{0bb}, \ldir{6bb}, \ldir{3bb} \,\}$.
&& Tandem words are classically those words on~$S_1$
   giving rise to a walk in the quarter plane.
   We will also call them quarter-plane 1-tandem words:
   \begin{equation*}
   \phantom{\rule{20mm}{0pt}}\mathcal Q =
   \{ w \in S_1^\ast :
     \forall w', \ w' \leq w \Rightarrow
       |w'|_{\sldir{0bb}} \geq |w'|_{\sldir{3bb}} \geq |w'|_{\sldir{6bb}} \} .
   \end{equation*}
&& By relaxing the quarter-plane constraint,
   but forcing the final height to be $0$,
   we obtain what we will call half-plane 1-tandem words:
   \begin{equation*}
   \phantom{\rule{20mm}{0pt}}\mathcal H =
   \{ w \in S_1^\ast :
     |w|_{\sldir{0bb}} = |w|_{\sldir{3bb}} \text{ and }
     \forall w', \ w' \leq w \Rightarrow
       |w'|_{\sldir{0bb}} \geq |w'|_{\sldir{3bb}} \} .
   \end{equation*}
   The bijection $\mathcal M \simeq \mathcal H$ is obvious.
& $\{1,2,3\}$.
   We restrict the classical class of Yamanouchi words
   (using any letters from~$\bN^\ast$)
   to those words using only the first three integers:
   \begin{equation*}
   \mathcal Y_3 =
   \{ w \in \{1,2,3\}^\ast :
     \forall w', \ w' \leq w \Rightarrow |w'|_1 \geq |w'|_2 \geq |w'|_3 \} .
   \end{equation*}
   The bijection $\mathcal Y_3 \simeq \mathcal Q$ is obvious.
& $S_{1,\text{sym}} = \{\, \ldir{0bb}, \ldir{6bb}, \ldir{3bb},
                           \ldir{4bb}, \ldir{2bb}, \ldir{7bb} \,\}$.
   By augmenting the step set of classical tandem walks by symmetry,
   we obtain a superset~$\mathcal Q^{\text{sym}}$ of~$\mathcal Q$:
   \begin{multline*}
   \phantom{\rule{15mm}{0pt}}\mathcal Q^{\text{sym}} =
   \{ w \in S_{1,\text{sym}}^\ast :
     \forall w', \ w' \leq w \Rightarrow \\
       |w'|_{\sldir{0bb}} + |w'|_{\sldir{7bb}} \geq |w'|_{\sldir{3bb}} + |w'|_{\sldir{4bb}} \text{ and }
       |w'|_{\sldir{3bb}} + |w'|_{\sldir{2bb}} \geq |w'|_{\sldir{6bb}} + |w'|_{\sldir{7bb}} \} .
   \end{multline*}
& $\Sigma_{1,\text{bicol}} =
  \{ {-\boldsymbol{1}}, {\boldsymbol{0}}, {+\boldsymbol{1}},
     {-\mathbb{1}}, {\mathbb{0}}, {+\mathbb{1}} \}
  \simeq
  \{\, \ldir{1Lb}, \ldir{2Lb}, \ldir{3Lb},
       \ldir{1Rb}, \ldir{2Rb}, \ldir{3Rb} \,\}$.
   For a bicolouring of Motzkin words,
   we use the two colours “solid” and “striped”:
   \begin{multline*}
   \phantom{\rule{15mm}{0pt}}\mathcal M^{\text{bicol}} =
   \{ w \in \Sigma_{1,\text{bicol}}^\ast :
     |w|_{\sldir{1Lb}} + |w|_{\sldir{1Rb}} = |w|_{\sldir{3Lb}} + |w|_{\sldir{3Rb}} \text{ and } \\
     \forall w', \ w' \leq w \Rightarrow
       |w'|_{\sldir{1Lb}} + |w'|_{\sldir{1Rb}} \geq |w'|_{\sldir{3Lb}} + |w'|_{\sldir{3Rb}} \} .
   \end{multline*}
& $\Sigma_p = \{-1,0,1,\dots,p\}$.
   Our $p$-Łukasiewicz words generalize Motzkin words
   by allowing longer up steps:
   \begin{multline*}
   \phantom{\rule{20mm}{0pt}}\mathcal L_p =
   \{ w \in \Sigma_p^\ast :
     |w|_1 + 2|w|_2 + \dots + p|w|_p = |w|_{-1} \\
     \forall w', \ w' \leq w \Rightarrow |w'|_1 + 2|w'|_2 + \dots + p|w'|_p \geq |w'|_{-1} \} .
   \end{multline*}
   In particular, $\mathcal L_1 = \mathcal M$.
& $S_p = \{\, \pdir{p}, \dots, \pdir{m}, \dots, \pdir{0}, \pdir{-} \,\}$.
   Our quarter-plane $p$-tandem words generalize classical tandem words
   by allowing longer steps into the north-west quadrant:
   \begin{multline*}
   \phantom{\rule{20mm}{0pt}}\mathcal T_p =
   \{ w \in S_p^\ast : \forall w', \ w' \leq w \Rightarrow \\
     \textstyle 
     \sum_{j=0}^p (p-j) |w'|_{(-j,p-j)} \geq |w'|_{(1,-1)} \geq
     \sum_{j=0}^p j |w'|_{(-j,p-j)} \} .
   \end{multline*}
   In particular, $\mathcal T_1 = \mathcal Q$.
& $A_p := \{ a_{\ell,m} : 0 \leq \ell+m \leq p-1 \}$.
   We will use the class~$A_p^\ast$ of unrestricted finite words on~$A_p$
   to represent generalized counters (stack contents).
\end{easylist}

\subsection{Outline of the paper. Contributions}\label{sec:outline}

Section~\ref{sec:3-step-model} discusses the case~$p = 1$.
It is bibliographical,
along the lines of \cite{BostanChyzakMahboubi--BBC,Chyzak-2016-ALB}.
(But we prefer the present description in terms of counters
over the equivalent description in terms of infinite automata
developed in~\cite{BostanChyzakMahboubi--BBC}.)
There, we provide a bijection $\mathcal H \simeq \mathcal Q$
between half-plane and quarter-plane words in the 1-tandem case.
The bijection is essentially the bijection $\mathcal M \simeq \mathcal Y_3$
between Motzkin words and (restricted) Yamanouchi words,
owing to the trivial relations
$\mathcal M \simeq \mathcal H$ and $\mathcal Y_3 \simeq \mathcal Q$.
This description is given without proofs,
as it can be viewed as a special case of the case of general~$p$.
This new result is a bijection
$\mathcal L_p \simeq \mathcal T_p$ for $1 \leq p \in \bN$.
It is developed in full in Section~\ref{sec:p-tandem}.
For the proofs, we introduce \emph{parameters}
attached to the constructions,
and we study their variations and invariants.
We then develop the 6-step model
and the corresponding bijection
$\mathcal M^{\text{bicol}} \simeq \mathcal Q^{\text{sym}}$
in a more sketchy way in Section~\ref{sec:six-step},
providing just what has to be adapted to obtain complete proofs.
Alternate interpretations follow,
via explicit automata in Section~\ref{sec:automata}
and via raisings of steps in Section~\ref{sec:raising}.

Although we have developped the automata approach of the present article
as a more intuitive explanation of the phenomenon for some of us,
other readers may prefer the approach by raising followed in~\cite{Yeats-2014-BBC}.
To save those readers from waiting until Section~\ref{sec:raising},
we work out in Section~\ref{sec:worked-example}
an example of execution of the automata and raising algorithms
to transform an explicit quarter-plane walk to the corresponding Motzkin walk.
We provide this comparison without any formalism and any proof of the bijections.
We thus propose two reads of the article.
Most readers will consider a linear read of it,
possibly skipping Section~\ref{sec:worked-example}
and postponing its read until after reading Section~\ref{sec:raising}.
On the other hand, readers already aware of the raising approach
will probably prefer reading the worked example in Section~\ref{sec:worked-example}
before jumping directly to Section~\ref{sec:raising}
and returning to Section~\ref{sec:3-step-model} when enclined to do so.

The interested reader will find a Julia implementation of the algorithms
described in this article,
both on Chyzak's web page and as ancillary files to the arXiv preprint of this work.

\subsection{Worked example}\label{sec:worked-example}

\begin{figure}
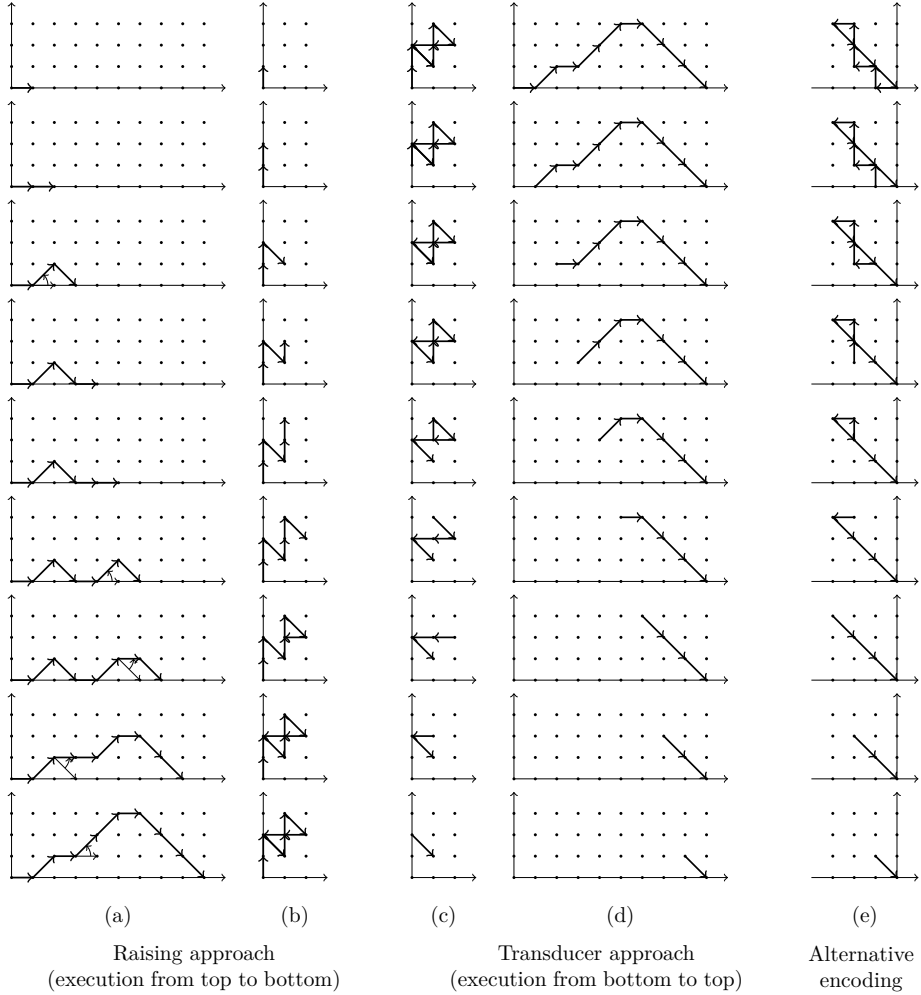

\resizebox{\textwidth}{!}{%
\begin{tabular}{cc@{\color{white}\rule{10mm}{9mm}}cc@{\color{white}\rule{10mm}{9mm}}c}
\exMprefix1 & \exQprefix1 & \exQsuffix9 & \exMsuffix9 & \exHsuffix9 \\
\exMprefix2 & \exQprefix2 & \exQsuffix8 & \exMsuffix8 & \exHsuffix8 \\
\exMprefix3 & \exQprefix3 & \exQsuffix7 & \exMsuffix7 & \exHsuffix7 \\
\exMprefix4 & \exQprefix4 & \exQsuffix6 & \exMsuffix6 & \exHsuffix6 \\
\exMprefix5 & \exQprefix5 & \exQsuffix5 & \exMsuffix5 & \exHsuffix5 \\
\exMprefix6 & \exQprefix6 & \exQsuffix4 & \exMsuffix4 & \exHsuffix4 \\
\exMprefix7 & \exQprefix7 & \exQsuffix3 & \exMsuffix3 & \exHsuffix3 \\
\exMprefix8 & \exQprefix8 & \exQsuffix2 & \exMsuffix2 & \exHsuffix2 \\
\exMprefix9 & \exQprefix9 & \exQsuffix1 & \exMsuffix1 & \exHsuffix1 \\[-2mm]
(a) & (b) & (c) & (d) & (e) \\
\multicolumn{2}{c}{\rule{0pt}{20pt}\parbox{.38\textwidth}{\centering Raising approach \\ (execution from top to bottom)}} &
  \multicolumn{2}{c}{\rule{0pt}{20pt}\parbox{.38\textwidth}{\centering Transducer approach \\ (execution from bottom to top)}} &
  \parbox{.15\textwidth}{\centering Alternative \\ encoding}
\end{tabular}
}
\caption{\label{fig:early-comparison}
  Computing the Motzkin walk associated with a given quarter-plane plane:
  comparison of the traditional raising approach and the new transducer approach.
  \emph{Raising algorithm}: input on the right~(b), output on the left~(a).
  These columns (a) and~(b) display the same data as Figure~\ref{fig raising eg}.
  This exemplifies the algorithm provided formally
  as Figure~\ref{fig:algo-p=1-by-raisings} in Section~\ref{sec:tandem-by-raising},
  with the exception that the raisability marks are not shown here.
  \emph{Transducer approach}: input on the left~(c), output on the right~(d).
  This algorithm is provided formally in Section~\ref{sec:3-step-model} (map~$\psi$)
  and generalized to long up steps in Section~\ref{sec:p-tandem} (map~$\Psi_p$).
  Column~(e) provides the walks homologous to those of column~(d)
  under the trivial bijection that exchanges
  \protect\ldir{0bb} and \protect\ldir{1bb}, and \protect\ldir{6bb} and \protect\ldir{2bb}.}
\end{figure}

The difference between the traditional approach by raising
and our new approach by transducers
is most striking in the divergence of their intermediate calculations
when converting an input walk in the quarter plane
with steps from $\{\, \ldir{0bb}, \ldir{6bb}, \ldir{3bb} \,\}$
into a walk in the upper half plane using the same steps
and returning to the $x$-axis,
or, equivalently, into a Motzkin walk using steps from
$\{\, \ldir{3bb}, \ldir{2bb}, \ldir{1bb} \,\}$.
Both approaches realize the map~$\psi$
that we will define and describe using transducers in Section~\ref{sec:p-tandem},
and that we will describe using raising in Section~\ref{sec:tandem-by-raising}.

In order to transform a quarter-plane walk,
like the one at the bottom of column~(b),
the traditional approach is to read its steps from beginning to end,
inducing the prefixes that can be seen from top to bottom in that column,
and to progressively generate the corresponding Motzkin walks in column~(a).
To maintain the Motzkin property,
this induces the so-called raising of an earlier step
when the added step would otherwise go below the $x$-axis.
Such raising instances are marked by round arrows and shift
the rest of the walk vertically by one.
The global transformation returns the final Motzkin walk,
given at the bottom of column~(a).

The approach introduced in the present work is
to read the quarter-plane walk backwards,
thus considering the suffixes that appear from bottom to top in column~(c).
This family of walks has no obvious property
with regard to a quarter plane or even a half plane,
and so we have not bothered to translate them in any way,
so as to have the origin as final vertex for example.
The new transformation is depicted in column~(d),
which shows suffixes of the final Motzkin walk.
Observe that those suffixes are not altered after they have been produced
in any other way than prepending another step.
In the end, the completed Motzkin walks obtained by both algorithms
and drawn at the bottom of column~(a) and top of column~(d)
are the same.
If we insist on writing upper-half-plane walks in terms of the same allowed steps
as for the quarter-plane walks,
we obtain in column~(e)
the suffixes of an upper-half-plane walk returning to the $x$-axis.
The fact that the complete walk, at the top of column~(e),
returns to the origin is specific to the example,
in relation to the quarter-plane walk ending on the main diagonal.

\section{Three-step model}\label{sec:3-step-model}

We start with the three-step model,
to obtain a bijection $\mathcal H \simeq \mathcal Q$,
which we realize as two converse bijective maps
$\phi: \mathcal H \to \mathcal Q$
and $\mathcal Q \to \mathcal H$.
This implicitly contains the bijection $\mathcal M \simeq \mathcal Y_3$,
which is the foundation of the case~$p = 1$ in the literature.

Following the strategy announced in the introduction,
we introduce the transitions that we use to define
the forward map~$\phi$ and the backward map~$\psi$.
This is done by the 6 boxes on the right of Figure~\ref{fig:phi}
and the 5 boxes on the right of Figure~\ref{fig:psi}.
We describe the meaning of the boxes using the notation introduced by the box~(TR) on their left.
There:
\begin{itemize}
\item \ldir{{$a$}sb} and~\ldir{{$\bar a$}sb} stand for letters from the alphabet
  $S_1 := \{\, \ldir{0bb}, \ldir{6bb}, \ldir{3bb} \,\}$;
\item $\mu$~is the $y$-coordinate of the letter~$a$,
  that is, the vertical displacement caused by~$a$,
  with value~$\mu$ in $\{ -1, 0, +1 \}$;
\item \ldir{{$m$}sb}~stands for a letter from the alphabet
  $\Sigma_1 := \{\, \ldir{1bb}, \ldir{2bb}, \ldir{3bb} \,\}$;
\item $(h_L, v_L)$ and~$(h_R, v_R)$ are pairs of counters, from~$\bN^2$
  (henceforth, we will simply say “counter” for short,
  instead of “pair of counters”);
\item $\chi$~stands for an integer from~$\{1, 2, 3\}$.
\end{itemize}

\begin{figure}
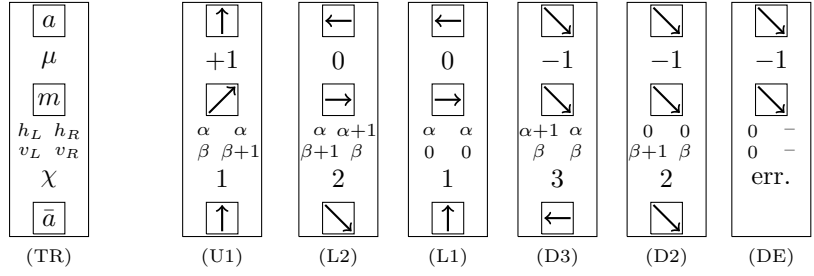

\noindent\ \hfill
\trsIIIstep in:{$a$}sb/$\mu$/{$m$}sb, h:$h_L$->$h_R$, v:$v_L$->$v_R$, out:$\chi$/{$\bar a$}sb (TR).
\ \hfill\
\trsIIIstep in:0bb/$+1$/1bb,  h:$\alpha$->$\alpha$,   v:$\beta$->$\beta{+}1$, out:1/0bb   (U1).
\,
\trsIIIstep in:6bb/0/2bb,     h:$\alpha$->$\alpha{+}1$, v:$\beta{+}1$->$\beta$, out:2/3bb   (L2).
\,
\trsIIIstep in:6bb/0/2bb,     h:$\alpha$->$\alpha$,   v:0->0,               out:1/0bb   (L1).
\,
\trsIIIstep in:3bb/$-1$/3bb,  h:$\alpha{+}1$->$\alpha$, v:$\beta$->$\beta$,   out:3/6bb   (D3).
\,
\trsIIIstep in:3bb/$-1$/3bb,  h:0->0,                 v:$\beta{+}1$->$\beta$, out:2/3bb   (D2).
\,\hspace{1pt}
\trsIIIstep in:3bb/$-1$/3bb,  h:0->--,                v:0->--,              out:{err.}/--n (DE).
\hfill\
\caption{\label{fig:phi}
  Transitions used to define the forward map $\phi$:
  the letter~\protect\ldir{{$m$}sb} is read from the input Motzkin word~$w$;
  the letter~$\chi$ is written to the output Yamanouchi word~$\bar w$.
  At the top, the rows \protect\ldir{{$a$}sb} and~$\mu$ present
  alternate encodings of~\protect\ldir{{$m$}sb}
  under obvious bijections.
  At the bottom, the row \protect\ldir{{$\bar a$}sb} presents
  an alternate encodings of~$\chi$.
  A transition named (X$\chi$) reads a letter~\protect\ldir{{$m$}sb} indicated by~X:
  respectively \protect\ldir{1bb}, \protect\ldir{2bb}, \protect\ldir{3bb}, if X~stands for U(p), L(evel), D(own).
  With the exception of (DE), all transitions (X$\chi$) output the letter~$\chi$.}
\end{figure}

\begin{figure}
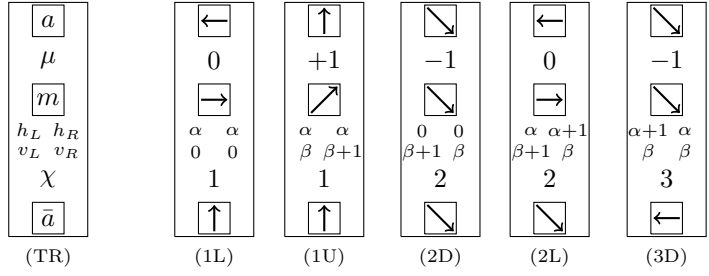

\noindent\ \hfill
\,
\trsIIIstep in:{$a$}sb/$\mu$/{$m$}sb, h:$h_L$->$h_R$, v:$v_L$->$v_R$, out:$\chi$/{$\bar a$}sb (TR).
\ \hfill\
\trsIIIstep in:6bb/0/2bb,    h:$\alpha$->$\alpha$,     v:0->0,                 out:1/0bb (1L).
\,\hspace{3pt}
\trsIIIstep in:0bb/$+1$/1bb, h:$\alpha$->$\alpha$,     v:$\beta$->$\beta{+}1$, out:1/0bb (1U).
\,
\trsIIIstep in:3bb/$-1$/3bb, h:0->0,                   v:$\beta{+}1$->$\beta$, out:2/3bb (2D).
\,
\trsIIIstep in:6bb/0/2bb,    h:$\alpha$->$\alpha{+}1$, v:$\beta{+}1$->$\beta$, out:2/3bb (2L).
\,
\trsIIIstep in:3bb/$-1$/3bb, h:$\alpha{+}1$->$\alpha$, v:$\beta$->$\beta$,     out:3/6bb (3D).
\,\hspace{1pt}
\phantom{\trsIIIstep in:3bb/$-1$/3bb,  h:0->--,                v:0->--,              out:{err.}/--n (DE).}
\hfill\
\caption{\label{fig:psi}
  Transitions used to define the backward map $\psi$.
  Each transition named ($\chi$X) is the reverse of the transition named (X$\chi$) in Figure~\ref{fig:phi}.
  The letter~$\chi$ is read from the input Yamanouchi word~$\bar w$;
  the letter~\protect\ldir{{$m$}sb} is written to the output Motzkin word~$w$.
  Alternate representations are as described in that figure.}
\end{figure}

The forward map~$\phi$ converts
a walk of the 3-step model~$\mathcal H$ in the upper half plane
to a walk of the 3-step model~$\mathcal Q$ in the quarter plane.
To ease its definition, the map is defined as a partial function
of any word from~$S_1^\ast$.  This is convenient because
the construction potentially leads to an error.
The map is defined in terms of the 6~forward transitions (U1) to~(DE)
listed in Figure~\ref{fig:phi}.
Any application of (U1) to~(D2)
consumes a letter \ldir{{$a$}sb} of the 3-step model
as well as a counter~$(h_L, v_L) \in \bN^2$,
and produces a new counter~$(h_R, v_R)$
together with a letter \ldir{{$\bar a$}sb} of the 3-step model.
The values of the counters $(h_L, v_L)$ and~$(h_R, v_R)$ in the boxes
are given as patterns,
where $\alpha$ and~$\beta$ have to be viewed as matching nonnegative integers.
So, for example, if the input step is \ldir{6bb} and the counters are $(h_L, v_L)$, then we use transition (L2) if $v_L>0$ and (L3) if $v_L=0$.
The special transition~(DE) results in an error,
instead of producing a new counter and a new letter.
The triples of consumed letters~\ldir{{$a$}sb} and patterns
for $h_L$ and~$v_L$
over all 6~boxes of Figure~\ref{fig:phi}
cover all possible input cases from $S_1 \times \bN^2$.

Given a word $w = w_1\dots w_n \in S_1^\ast$ of length~$n$,
the forward transform~$\phi(w)$ is now defined as follows.
Set $(\sLR{h}_0,\sLR{v}_0) = (0,0)$.
For $i$ from~1 to~$n$, find the only forward transition
that can be applied to the letter~$w_i$
and the counter~$(\sLR{h}_{i-1}, \sLR{v}_{i-1})$.
If this is (DE), then $\phi$~is not defined at~$w$.
Otherwise, apply the transition
to set~$(\sLR{h}_i, \sLR{v}_i)$
to the returned counter and $\bar w_i$~to the returned letter.
If the loop terminates without using the transition~(DE),
then $\phi(w)$~is defined as $\bar w_1\dots \bar w_n$.

The backward map~$\psi$ is defined similarly
using the backward transitions (1L) to~(3E)
listed in Figure~\ref{fig:psi},
with the added simplicity that no error can occur.
Applications of backward transitions
consume a letter \ldir{{$\bar a$}sb} and a counter~$(h_R, v_R)$,
and produce a counter~$(h, v)$ together with a letter \ldir{{$a$}sb}.
Again, the triples~$(\ldir{{$\bar a$}sb}, h_R, v_R)$ cover all possible input cases
from $S_1 \times \bN^2$.
(For other models in future sections, we will not list the backward transitions separately since they are the same boxes interpreted in reverse.)

Given a word $\bar w = \bar w_1\dots \bar w_n \in S_1^\ast$ of length~$n$,
the backward transform~$\psi(\bar w)$ is now defined as follows.
Set $(\sRL{h}_n,\sRL{v}_n) = (0,0)$.
For $i$ from~$n$ down to~1, find the only backward transition
that can be applied to the letter~$\bar w_i$
and the counter~$(\sRL{h}_i, \sRL{v}_i)$,
so as to define a letter~$w_i$
and a counter~$(\sRL{h}_{i-1}, \sRL{v}_{i-1})$.
Then $\psi(\bar w)$~is defined as~$w_1\dots w_n$.

One can prove that, after calculating some~$\phi(w)$,
the last counter~$(\sLR{h}_n, \sLR{v}_n)$ is~$(0, 0)$
if and only if $w$~is a Motzkin word.
Likewise, after calculating some~$\psi(\bar w)$,
the last counter~$(\sRL{h}_0, \sRL{v}_0)$ is~$(0, 0)$
if and only if $\bar w$~is a quarter-plane word.
Also, one can prove that if some
$(\sLR{h}_i, \sLR{v}_i, \bar w_i)$
is the result of applying the forward transition~(U1),
then the only backward transition that can be applied to this result,
in a right to left process having
this specific letter~$\bar w_i$ at position~$i$,
is~(1U).
Like properties hold for all other natural pairs of forward and backward transitions.
The proof that $\phi$ and~$\psi$ are inverse bijections
between Motzkin words of~$\mathcal M$ and quarter-plane words of~$\mathcal Q$
follows from this local reversibility property.
We do not exhibit those proofs here as they will be seen as special cases
of the case of a general~$p$ in Section~\ref{sec:p-tandem}.

Because of an interpretation specific to the case~$p = 1$
and in view of the generalization to larger~$p$ in Section~\ref{sec:p-tandem},
the input and output letters are also given in the boxes in different alphabets.
Three-step walks in the upper half plane are in bijection
with Motzkin words,
by replacing each~\ldir{0bb} with a~\ldir{1bb},
each~\ldir{6bb} with a~\ldir{2bb},
and leaving all~\ldir{3bb}'s unchanged.
The letter~\ldir{{$m$}sb} corresponding to~\ldir{{$a$}sb} is given in the boxes,
together with its vertical displacement~$\mu$.
Here, the letters are those for Motzkin walks,
implying~$-1 \leq \mu \leq 1$,
but the bijection in Section~\ref{sec:p-tandem} will use $-1 \leq \mu \leq p$,
so this notation will be preferred.
On the other hand, the constraint on a 3-step walk
to remain in the quarter plane
leads to each prefix of such a 3-step walk having
more \ldir{0bb}'s than \ldir{3bb}'s and
more \ldir{3bb}'s than \ldir{6bb}'s.
Replacing \ldir{0bb}, \ldir{3bb}, and~\ldir{6bb} respectively with 1, 2, and~3
puts walks in the quarter plane in bijection with Yamanouchi words.
The bijection for~$p = 1$ is thus essentially a bijection
between Motzkin words and Yamanouchi words using only 1s, 2s, and~3s.

\begin{figure}
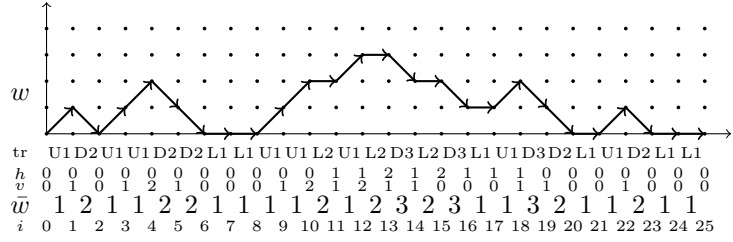


\centerline{%
\begin{mypic}
\foreach \x in {0,1,2,3,4,5,6,7,8,9,10,11,12,13,14,15,16,17,18,19,20,21,22,23,24,25}
  \foreach \y in {0,1,2,3,4}
    \fill(\x,\y) circle[radius=2pt];
\draw[->](0,0)--(26,0);
\draw[->](0,0)--(0,5);
\draw(-1,1.5) node{$w$};
\draw(-1,-.7) node{\tiny tr};
\draw(-1,-1.5) node{\tiny$h$};
\draw(-1,-2) node{\tiny$v$};
\draw(-1,-2.7) node{$\bar w$};
\draw[thick,->](0,0)--(1,1);    \transition{0}{U1}{0}{0}{1}
\draw[thick,->](1,1)--(2,0);    \transition{1}{D2}{0}{1}{2}
\draw[thick,->](2,0)--(3,1);    \transition{2}{U1}{0}{0}{1}
\draw[thick,->](3,1)--(4,2);    \transition{3}{U1}{0}{1}{1}
\draw[thick,->](4,2)--(5,1);    \transition{4}{D2}{0}{2}{2}
\draw[thick,->](5,1)--(6,0);    \transition{5}{D2}{0}{1}{2}
\draw[thick,->](6,0)--(7,0);    \transition{6}{L1}{0}{0}{1}
\draw[thick,->](7,0)--(8,0);    \transition{7}{L1}{0}{0}{1}
\draw[thick,->](8,0)--(9,1);    \transition{8}{U1}{0}{0}{1}
\draw[thick,->](9,1)--(10,2);   \transition{9}{U1}{0}{1}{1}
\draw[thick,->](10,2)--(11,2);  \transition{10}{L2}{0}{2}{2}
\draw[thick,->](11,2)--(12,3);  \transition{11}{U1}{1}{1}{1}
\draw[thick,->](12,3)--(13,3);  \transition{12}{L2}{1}{2}{2}
\draw[thick,->](13,3)--(14,2);  \transition{13}{D3}{2}{1}{3}
\draw[thick,->](14,2)--(15,2);  \transition{14}{L2}{1}{1}{2}
\draw[thick,->](15,2)--(16,1);  \transition{15}{D3}{2}{0}{3}
\draw[thick,->](16,1)--(17,1);  \transition{16}{L1}{1}{0}{1}
\draw[thick,->](17,1)--(18,2);  \transition{17}{U1}{1}{0}{1}
\draw[thick,->](18,2)--(19,1);  \transition{18}{D3}{1}{1}{3}
\draw[thick,->](19,1)--(20,0);  \transition{19}{D2}{0}{1}{2}
\draw[thick,->](20,0)--(21,0);  \transition{20}{L1}{0}{0}{1}
\draw[thick,->](21,0)--(22,1);  \transition{21}{U1}{0}{0}{1}
\draw[thick,->](22,1)--(23,0);  \transition{22}{D2}{0}{1}{2}
\draw[thick,->](23,0)--(24,0);  \transition{23}{L1}{0}{0}{1}
\draw[thick,->](24,0)--(25,0);  \transition{24}{L1}{0}{0}{1}
                                \transition{25}{}{0}{0}{}
\foreach \x in {0,1,2,3,4,5,6,7,8,9,10,11,12,13,14,15,16,17,18,19,20,21,22,23,24,25}
  \draw(\x,-3.5) node{\tiny\x};
\draw(-1,-3.5) node{\tiny$i$};
\end{mypic}
}

\caption{An example of a Motzkin walk~$w$ (letters \protect\ldir{{$m$}sb}) and of a Yamanouchi word~$\bar w$ (letters~$\chi$),
  both of length~25,
  associated with one another by the maps $\phi$ and~$\psi$.
  The row marked `tr' indicates the transitions used by the LR process,
  computing~$\bar w$ from~$w$.
  For the computation of~$w$ from~$\bar w$ by the RL process, each transition is replaced by its reverse form, according to Figure~\ref{fig:phi}.}
\label{fig:bij-ex-motzkin-yamanouchi}

\end{figure}

\begin{figure}
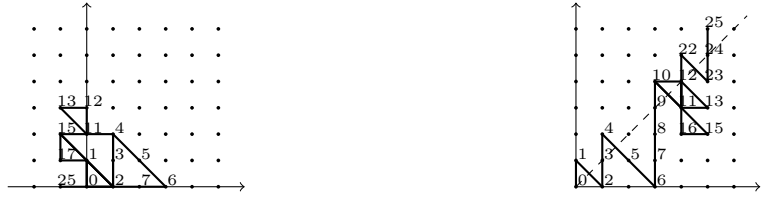


\noindent
\rule{12mm}{0mm}
%
\begin{mypic}
\foreach \x in {-2,-1,0,1,2,3,4,5}
  \foreach \y in {0,1,2,3,4,5,6}
    \fill(\x,\y) circle[radius=2pt];
\draw[->](-3,0)--(6,0);
\draw[->](0,0)--(0,7);
\draw[thick](0,0)--(0,1);
\draw[thick](0,1)--(1,0);
\draw[thick](1,0)--(1,1);
\draw[thick](1,1)--(1,2);
\draw[thick](1,2)--(2,1);
\draw[thick](2,1)--(3,0);
\draw[thick](3,0)--(2,0);
\draw[thick](2,0)--(1,0);
\draw[thick](1,0)--(1,1);
\draw[thick](1,1)--(1,2);
\draw[thick](1,2)--(0,2);
\draw[thick](0,2)--(0,3);
\draw[thick](0,3)--(-1,3);
\draw[thick](-1,3)--(0,2);
\draw[thick](0,2)--(-1,2);
\draw[thick](-1,2)--(0,1);
\draw[thick](0,1)--(-1,1);
\draw[thick](-1,1)--(-1,2);
\draw[thick](-1,2)--(0,1);
\draw[thick](0,1)--(1,0);
\draw[thick](1,0)--(0,0);
\draw[thick](0,0)--(0,1);
\draw[thick](0,1)--(1,0);
\draw[thick](1,0)--(0,0);
\draw[thick](0,0)--(-1,0);
\foreach \x / \y / \t in {.25/.25/0, .25/1.25/1, 1.25/.25/2, 1.25/1.25/3, 1.25/2.25/4, 2.25/1.25/5, 3.25/.25/6, 2.25/.25/7, .25/2.25/11, .25/3.25/12, -.75/3.25/13, -.75/2.25/15, -.75/1.25/17, -.75/.25/25}
  \draw(\x,\y) node{\tiny\t};
\end{mypic}
\hfill
%
\begin{mypic}
\foreach \x in {0,1,2,3,4,5,6}
  \foreach \y in {0,1,2,3,4,5,6}
    \fill(\x,\y) circle[radius=2pt];
\draw[->](0,0)--(7,0);
\draw[->](0,0)--(0,7);
\draw[dashed](0,0)--(6.5,6.5);
\draw[thick](0,0)--(0,1);
\draw[thick](0,1)--(1,0);
\draw[thick](1,0)--(1,1);
\draw[thick](1,1)--(1,2);
\draw[thick](1,2)--(2,1);
\draw[thick](2,1)--(3,0);
\draw[thick](3,0)--(3,1);
\draw[thick](3,1)--(3,2);
\draw[thick](3,2)--(3,3);
\draw[thick](3,3)--(3,4);
\draw[thick](3,4)--(4,3);
\draw[thick](4,3)--(4,4);
\draw[thick](4,4)--(5,3);
\draw[thick](5,3)--(4,3);
\draw[thick](4,3)--(5,2);
\draw[thick](5,2)--(4,2);
\draw[thick](4,2)--(4,3);
\draw[thick](4,3)--(4,4);
\draw[thick](4,4)--(3,4);
\draw[thick](3,4)--(4,3);
\draw[thick](4,3)--(4,4);
\draw[thick](4,4)--(4,5);
\draw[thick](4,5)--(5,4);
\draw[thick](5,4)--(5,5);
\draw[thick](5,5)--(5,6);
\foreach \x / \y / \t in {.25/.25/0, .25/1.25/1, 1.25/.25/2, 1.25/1.25/3, 1.25/2.25/4, 2.25/1.25/5, 3.25/.25/6, 3.25/1.25/7, 3.25/2.25/8, 3.25/3.25/9, 3.25/4.25/10, 4.25/3.25/11, 4.25/4.25/12, 5.25/3.25/13, 5.25/2.25/15, 4.25/2.25/16, 4.25/5.25/22, 5.25/4.25/23, 5.25/5.25/24, 5.25/6.25/25}
  \draw(\x,\y) node{\tiny\t};
\end{mypic}
\rule{12mm}{0mm}

\caption{Motzkin walk of Figure~\ref{fig:bij-ex-motzkin-yamanouchi}
  interpreted as an upper-half-plane walk (left)
  and Yamanouchi word of Figure~\ref{fig:bij-ex-motzkin-yamanouchi}
  interpreted as a quarter-plane walk (right).}
\label{fig:bij-ex-as-walks}

\end{figure}

\begin{example}\label{ex:bij-ex-motzkin-yamanouchi}
Figure~\ref{fig:bij-ex-motzkin-yamanouchi} is an example of the bijection,
expressed as a relation between a Motzkin walk and a Yamanouchi word.
The letters $w_i = \ldir{{$m$}sb}$ (at the top) have been glued into a walk~$w$.  The corresponding letters $\bar w_i = \chi$ are given at the bottom,
below the counters $h$ and~$v$.
The index~$i$ on the final row mark the “time” in the algorithmic transformations.
Specifically, the value of~$h$, respectively, of~$v$, in the same column as a value of~$i$
are the values of $\sLR{h}_i$ or~$\sRL{h}_i$,
respectively, of $\sLR{v}_i$ or~$\sRL{v}_i$,
in our previous description.
Pay attention to how a value $\chi = \bar w_i$
appears between times $i-1$ and~$i$.

Interpreted as a bijection between walks in the 3-step models,
the bijection maps the upper-half-plane walk on the left
of Figure~\ref{fig:bij-ex-as-walks}
to the quarter-plane walk on the right.
To help the reader understand the order of these tangled walks,
we have put the time of first reach close to each position.
This corresponds to the value of~$i$
in Figure~\ref{fig:bij-ex-motzkin-yamanouchi}.
Additionally,
for the upper-half-plane walk,
8, 20 and~23 coincide with~2,
9~with~3,
10~with~4,
14~with~11,
16, 19 and~22 with~1,
18~with~15,
21 and~24 with~0;
for the quarter-plane walk,
14, 17 and~20 coincide with~11,
19~with~10,
18 and~21 with~12.

An interesting property is made apparent by this example.
Focus on those positions between times $i$ and~$i+1$
for which \ldir{{$m$}sb} = \ldir{2bb} in the Motzkin walk
and for which $\chi = 1$ in the Yamanouchi word.
In other words, consider
$i \in \{ 6, 7, 16, 20, 23, 24 \}$,
but not $i \in \{ 10, 12, 14 \}$.
Then, we observe that
the suffix of the quarter-plane walk for the position at that time~$i$
never goes back to a lower abscissa than that of its vertex marked~$i$.
(On the other hand, this is not true for 10, 12, 14.)  See the first point of Proposition~\ref{prop:observations}.
\end{example}

The reader in need of an early intuition
about our bijections between models with long steps
can compare
Figure~\ref{fig:bij-ex-motzkin-yamanouchi} (three-step model, $p = 1$)
and Figure~\ref{fig:transduction-ex-p=5} ($p = 5$).

\section{$p$-Tandem Walks}\label{sec:p-tandem}

In this section, we generalize the bijections of Section~\ref{sec:3-step-model}
to general~$p\geq1$.

\subsection{Construction of maps between $\Sigma_p^\ast$ and $S_p^\ast$}
\label{sec:cons-maps}

For the description of the bijection for general~$p$,
we proceed to introduce transitions akin to those for the case~$p=1$,
but with two kinds of generalizations needed.
\begin{itemize}
\item The letters \ldir{1bb}, \ldir{2bb}, and \ldir{3bb},
  used to encode Motzkin words for~$p=1$ and
  also represented by their $y$-coordinates
  (that is, their vertical displacements),
  respectively as $+1$, 0,~$-1$,
  need to be generalized to a set of $p+2$ letters~$\mu$,
  named by numbers from $-1$ to~$p$.
\item The letters \ldir{0bb}, \ldir{6bb}, and~\ldir{3bb},
  used for the 3-step quarter-plane model,
  need to be generalized to another set of $p+2$ letters~$\bar a$:
  the $p+1$ vectors~$(-\ell,p-\ell)$ for~$0\leq\ell\leq p$,
  together with the vector~$(1,-1)$.
  In the boxes of Figure~\ref{fig:phi-p}, the parameter~$\bar\mu$ is the $y$-coordinate of~$\bar a$.
  It also varies from $-1$ to~$p$ and could be used as another encoding.
\item The parameter~$h\in\bN$ of the case~$p=1$ needs to be generalized
  to a word~$H$ over the alphabet~$A_p$ consisting
  of the $p(p-1)/2$ letters~$a_{\ell,m}$ for $0 \leq \ell+m \leq p-1$.
  For~$p=1$, this alphabet reduces to a single letter~$a_{0,0}$,
  so that the values of~$H$ for~$p=1$ are in natural bijection
  with the possible values of~$h\in\bN$.
  This explains why a mere integer was sufficient
  in Section~\ref{sec:3-step-model}.
\end{itemize}

We could find no better way to name the transitions in Figure~\ref{fig:phi-p}
than simply naming them from~\boxref{T1} to~\boxref{T8}.
We describe their meaning using the notation introduced by the box~\boxref{TG}.
There:
\begin{itemize}
\item $\mu$~is the $y$-coordinate of a letter $(1,\mu)$ of a Łukasiewicz word,
  that is, the vertical displacement caused by it,
  with value in $\Sigma_p := \{ -1, 0, 1, \dots, p \}$;
\item $(H_L, v_L)$ and~$(H_R, v_R)$ are generalized counters,
  from~$A_p^\ast\times\bN$
  where $A_p$~is the alphabet $\{ a_{\ell,m} : 0 \leq \ell+m \leq p-1 \}$;
\item \ldir{{$\bar a$}sb} stands for a letter from the alphabet
  $S_p := \{\, \pdir{p}, \dots, \pdir{m}, \dots, \pdir{0}, \pdir{-} \,\}$;
\item $\bar\mu$~also stands for a letter from the alphabet~$\Sigma_p$.
\end{itemize}

\begin{figure}
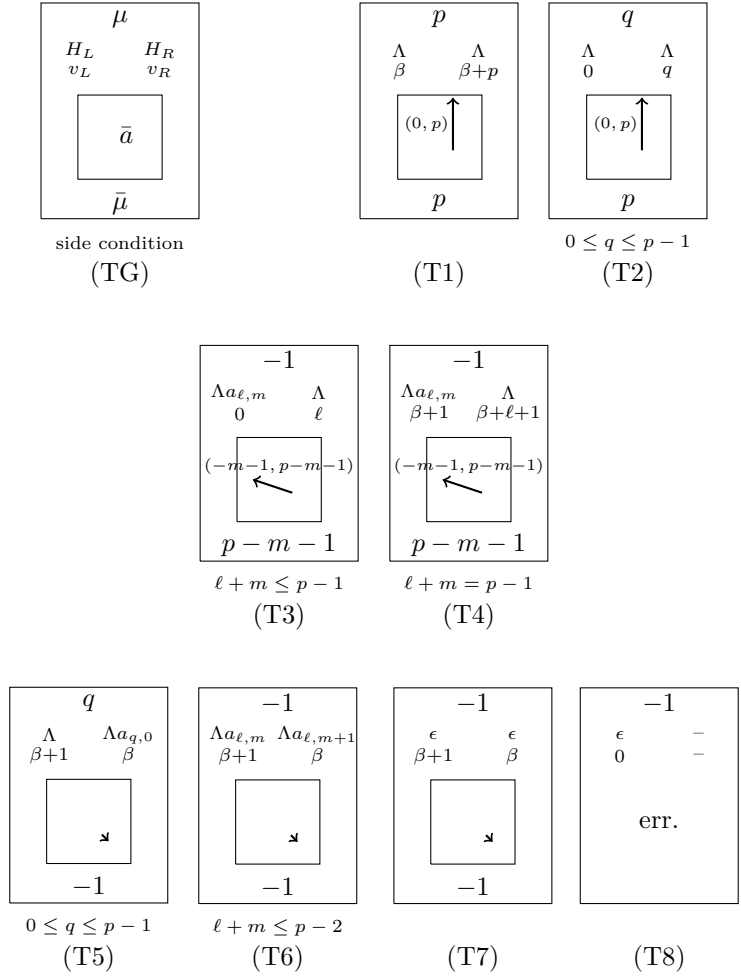

\noindent\ \hfill
\trsptandem p:g, in:$\mu$, h:$H_L$->$H_R$, v:$v_L$->$v_R$, out:s/$\bar\mu$, cond:{side condition}, label:\boxlabel{TG}.
\ \hfill\
\trsptandem p:g, in:$p$,   h:$\Lambda$->$\Lambda$,          v:$\beta$->$\beta{+}p$,            out:0/$p$, cond:{$\phantom{f}$}, label:\boxlabel{T1}.
\,\hspace{3pt}
\trsptandem p:g, in:$q$,   h:$\Lambda$->$\Lambda$,          v:0->$q$,                          out:0/$p$, cond:{$0\leq q\leq p-1$}, label:\boxlabel{T2}.
\hfill\

\bigskip\bigskip

\noindent\ \hfill
\trsptandem p:g, in:$-1$,  h:$\Lambda a_{\ell,m}$->$\Lambda$, v:0->$\ell$,                       out:1/$p-m-1$, cond:{$\ell+m\leq p-1$}, label:\boxlabel{T3}.
\hspace{-5pt}
\trsptandem p:g, in:$-1$,  h:$\Lambda a_{\ell,m}$->$\Lambda$, v:$\beta{+}1$->$\beta{+}\ell{+}1$, out:1/$p-m-1$, cond:{$\ell+m=p-1$}, label:\boxlabel{T4}.
\hfill\

\bigskip\bigskip

\noindent\ \hfill
\trsptandem p:g, in:$q$,   h:$\Lambda$->$\Lambda a_{q,0}$,              v:$\beta{+}1$->$\beta$, out:-/$-1$, cond:{$0\leq q\leq p-1$}, label:\boxlabel{T5}.
\,\hspace{3pt}
\trsptandem p:g, in:$-1$,  h:$\Lambda a_{\ell,m}$->$\Lambda a_{\ell,m+1}$, v:$\beta{+}1$->$\beta$, out:-/$-1$, cond:{$\ell+m\leq p-2$}, label:\boxlabel{T6}.
\,\hspace{2pt}
\trsptandem p:g, in:$-1$,  h:$\epsilon$->$\epsilon$,                   v:$\beta{+}1$->$\beta$, out:-/$-1$, cond:{$\phantom{f}$}, label:\boxlabel{T7}.
\,\hspace{2pt}
\trsptandem p:g, in:$-1$,  h:$\epsilon$->--,                           v:0->--,                out:e/{},   cond:{$\phantom{f}$}, label:\boxlabel{T8}.
\hfill\
\caption{\label{fig:phi-p}
  Transitions used for the bijections for general~$p$.
  When~$p=1$, (T3) and~(T6) disappear,
  and (T1), (T2), (T4), (T5), (T7), (T8), respectively,
  become (U1), (L1), (D3), (L2), (D2), (DE), respectively, in Figure~\ref{fig:phi}.}
\end{figure}

\begin{figure}
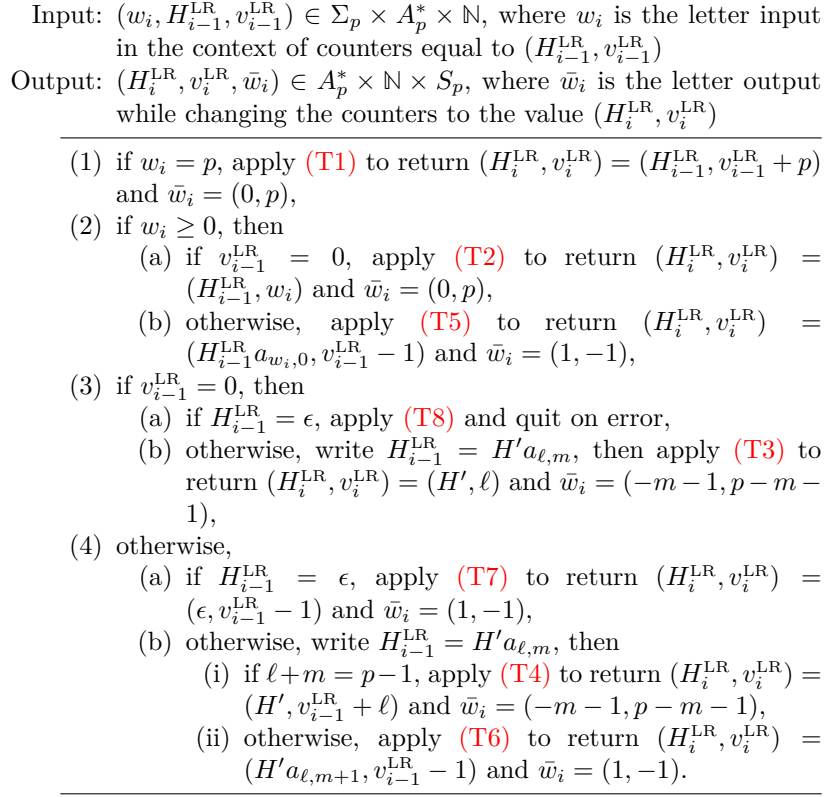

\begin{ouralgo}
\Input{$(w_i, \sLR{H}_{i-1}, \sLR{v}_{i-1}) \in \Sigma_p \times A_p^\ast \times \bN$, where $w_i$~is the letter input in the context of counters equal to $(\sLR{H}_{i-1}, \sLR{v}_{i-1})$}
\Output{$(\sLR{H}_i, \sLR{v}_i, \bar w_i) \in A_p^\ast \times \bN \times S_p$, where $\bar w_i$~is the letter output while changing the counters to the value $(\sLR{H}_i, \sLR{v}_i)$}
\midalgo
\item if $w_i = p$, apply~\boxref{T1}
  to return $(\sLR{H}_i, \sLR{v}_i) = (\sLR{H}_{i-1}, \sLR{v}_{i-1}+p)$
  and $\bar w_i = (0, p)$,
\item\label{it:w-intermediate}
  if $w_i \geq 0$, then
  \begin{enumerate}[ref=\ref{it:w-intermediate}(\alph*)]
  \item if $\sLR{v}_{i-1} = 0$, apply~\boxref{T2}
    to return $(\sLR{H}_i, \sLR{v}_i) = (\sLR{H}_{i-1}, w_i)$
    and $\bar w_i = (0, p)$,
  \item\label{it:T5}
    otherwise, apply~\boxref{T5}
    to return
    $(\sLR{H}_i, \sLR{v}_i) = (\sLR{H}_{i-1}a_{w_i,0}, \sLR{v}_{i-1}-1)$
    and $\bar w_i = (1, -1)$,
  \end{enumerate}
\item\label{it:v-is-0}
  if $\sLR{v}_{i-1} = 0$, then
  \begin{enumerate}[ref=\ref{it:v-is-0}(\alph*)]
  \item if $\sLR{H}_{i-1} = \epsilon$, apply~\boxref{T8} and quit on error,
  \item\label{it:T3}
    otherwise, write $\sLR{H}_{i-1} = \prefH a_{\ell,m}$, then apply~\boxref{T3}
    to return $(\sLR{H}_i, \sLR{v}_i) = (\prefH, \ell)$
    and $\bar w_i = (-m-1, p-m-1)$,
  \end{enumerate}
\item\label{it:v-is-non0}
  otherwise,
  \begin{enumerate}[ref=\ref{it:v-is-non0}(\alph*)]
  \item\label{it:T7}
    if $\sLR{H}_{i-1} = \epsilon$, apply~\boxref{T7}
    to return $(\sLR{H}_i, \sLR{v}_i) = (\epsilon, \sLR{v}_{i-1}-1)$
    and $\bar w_i = (1, -1)$,
  \item\label{it:H-is-nonempty}
    otherwise, write $\sLR{H}_{i-1} = \prefH a_{\ell,m}$, then
    \begin{enumerate}[ref=\ref{it:H-is-nonempty}(\roman*)]
    \item\label{it:T4}
      if $\ell+m = p-1$, apply~\boxref{T4}
      to return $(\sLR{H}_i, \sLR{v}_i) = (\prefH, \sLR{v}_{i-1}+\ell)$
      and $\bar w_i = (-m-1, p-m-1)$,
    \item\label{it:T6}
      otherwise, apply~\boxref{T6}
      to return
      $(\sLR{H}_i, \sLR{v}_i) = (\prefH a_{\ell,m+1}, \sLR{v}_{i-1}-1)$
      and $\bar w_i = (1, -1)$.
    \end{enumerate}
  \end{enumerate}
\end{ouralgo}
\caption{\label{fig:algo-LR-elementary-step}
  One step of the LR process ($i$th step).}
\end{figure}

As in the definition of the forward map~$\phi$
in Section~\ref{sec:3-step-model},
those transitions are used in a forward manner
to produce a counter~$(\sLR{H}_i, \sLR{v}_i)$ and a letter~$\bar w_i$
from a given letter~$w_i$ and a counter~$(\sLR{H}_{i-1}, \sLR{v}_{i-1})$.
The procedure, given in Figure~\ref{fig:algo-LR-elementary-step},
explains how transitions are selected unambiguously.
We prove that, given an input $(w_i, \sLR{H}_{i-1}, \sLR{v}_{i-1})$
from $\Sigma_p \times A_p^\ast \times \bN$,
this procedure returns an output $(\sLR{H}_i, \sLR{v}_i, \bar w_i)$
from $A_p^\ast \times \bN \times S_p$,
unless it applies~\boxref{T8}.
By inspection:
\begin{itemize}
\item $\bar w_i$~is either $(0,p)$, or~$(1,-1)$,
  or, at steps \ref{it:T3} and~\ref{it:T4},
  $(-m-1,p-m-1)$ for~$0 \leq m \leq p-1-\ell$.
  In all cases, $\bar w_i$~is in~$\Sigma_p$.
\item $\sLR{v}_i$~is the sum of one or two nonnegative integers,
  except at steps \ref{it:T5}, \ref{it:T7}, and~\ref{it:T6},
  where is it~$\sLR{v}_{i-1}-1$, for positive~$\sLR{v}_{i-1}$.
  In all cases, $\sLR{v}_i$~is in~$\bN$.
\item $\sLR{H}_i$~is always a prefix of~$\sLR{H}_{i-1}$,
  possibly augmented by a letter~$b$:
  at step~\ref{it:T5}, by $b = a_{w_i,0}$ when $0 \leq w_i \leq p-1$;
  at step~\ref{it:T6}, by $b = a_{\ell,m+1}$ when $0 \leq \ell+m \leq p-2$.
  In all cases, $\sLR{H}_i$~is in~$A_p^\ast$.
\end{itemize}

As in the case~$p=1$ of Section~\ref{sec:3-step-model},
we next define a (partial) map~$\Phi_p$ from~$\Sigma_p^\ast$ to~$S_p^\ast$.
Given an input word~$w = w_1\dots w_n$ from~$\Sigma_p^\ast$,
iterating the process above
from the initial counter $(\sLR{H}_0, \sLR{v}_0) = (\epsilon, 0)$
results either in an error,
or in a word~$\bar w = \bar w_1\dots \bar w_n$.
In the former case, $\Phi_p$~remains undefined at~$w$,
while in the latter case, $\Phi_p(w)$~is defined as~$\bar w$.

\begin{figure}
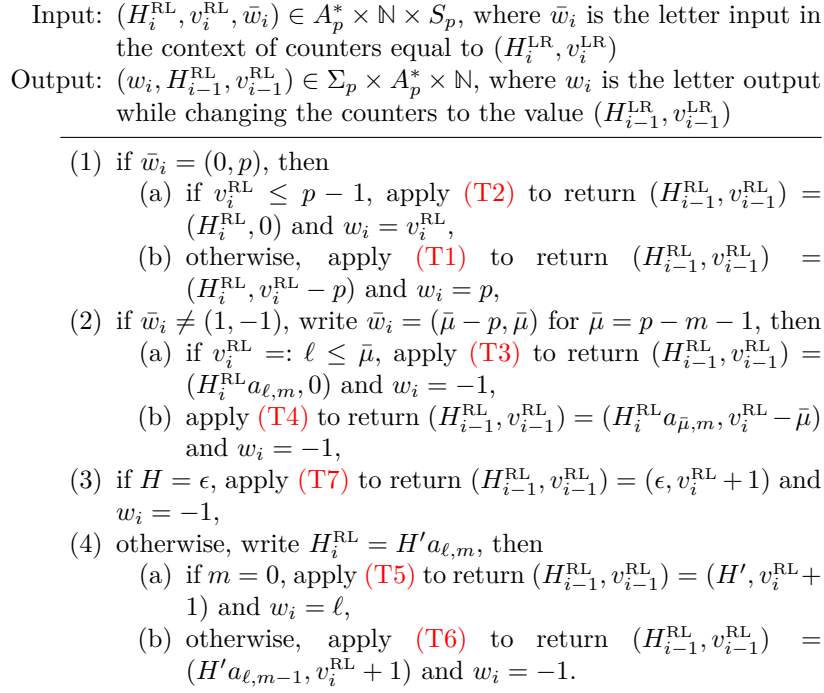

\begin{ouralgo}
\Input{$(\sRL{H}_i, \sRL{v}_i, \bar w_i) \in A_p^\ast \times \bN \times S_p$, where $\bar w_i$~is the letter input in the context of counters equal to $(\sLR{H}_i, \sLR{v}_i)$}
\Output{$(w_i, \sRL{H}_{i-1}, \sRL{v}_{i-1}) \in \Sigma_p \times A_p^\ast \times \bN$, where $w_i$~is the letter output while changing the counters to the value $(\sLR{H}_{i-1}, \sLR{v}_{i-1})$}
\midalgo
\item\label{it:ww-is-p}
  if $\bar w_i = (0, p)$, then
  \begin{enumerate}[ref=\ref{it:ww-is-p}(\alph*)]
  \item\label{it:T2inv}
    if $\sRL{v}_i \leq p-1$, apply~\boxref{T2}
    to return $(\sRL{H}_{i-1}, \sRL{v}_{i-1}) = (\sRL{H}_i, 0)$
    and $w_i = \sRL{v}_i$,
  \item\label{it:T1inv}
    otherwise, apply~\boxref{T1}
    to return $(\sRL{H}_{i-1}, \sRL{v}_{i-1}) = (\sRL{H}_i, \sRL{v}_i-p)$
    and $w_i = p$,
  \end{enumerate}
\item\label{it:ww-is-nonneg}
  if $\bar w_i \neq (1, -1)$, write $\bar w_i = (\bar\mu-p, \bar\mu)$ for $\bar\mu=p-m-1$, then
  \begin{enumerate}[ref=\ref{it:ww-is-nonneg}(\alph*)]
  \item\label{it:T3inv}
    if $\sRL{v}_i =: \ell \leq \bar\mu$, apply~\boxref{T3}
    to return $(\sRL{H}_{i-1}, \sRL{v}_{i-1}) = (\sRL{H}_ia_{\ell,m}, 0)$
    and $w_i = -1$,
  \item\label{it:T4inv}
    apply~\boxref{T4}
    to return
    $(\sRL{H}_{i-1}, \sRL{v}_{i-1}) = (\sRL{H}_ia_{\bar\mu,m}, \sRL{v}_i-\bar\mu)$
    and $w_i = -1$,
  \end{enumerate}
\item if $H = \epsilon$, apply~\boxref{T7}
  to return $(\sRL{H}_{i-1}, \sRL{v}_{i-1}) = (\epsilon, \sRL{v}_i+1)$
  and $w_i = -1$,
\item\label{it:HH-is-nonempty}
  otherwise, write $\sRL{H}_i = \prefH a_{\ell,m}$, then
  \begin{enumerate}[ref=\ref{it:HH-is-nonempty}(\alph*)]
  \item\label{it:T5inv}
    if $m = 0$, apply~\boxref{T5}
    to return $(\sRL{H}_{i-1}, \sRL{v}_{i-1}) = (\prefH, \sRL{v}_i+1)$
    and $w_i = \ell$,
  \item\label{it:T6inv}
    otherwise, apply~\boxref{T6}
    to return
    $(\sRL{H}_{i-1}, \sRL{v}_{i-1}) = (\prefH a_{\ell,m-1}, \sRL{v}_i+1)$
    and $w_i = -1$.
  \end{enumerate}
\end{ouralgo}
\caption{\label{fig:algo-RL-elementary-step}
  One step of the RL process ($(n+1-i)$th step).}
\end{figure}

Likewise, as in the definition of the backward map~$\psi$
in Section~\ref{sec:3-step-model},
the transitions \boxref{T1} to~\boxref{T8} in Figure~\ref{fig:phi-p} are used in a backward manner
to produce a counter $(\sRL{H}_{i-1}, \sRL{v}_{i-1})$ and a letter~$w_i$
from a given letter~$\bar w_i$ and a counter~$(\sRL{H}_i, \sRL{v}_i)$.
We provide the procedure in Figure~\ref{fig:algo-RL-elementary-step} to explain
how transitions, meant to be applied backwards,
are selected unambiguously.
We prove that, given an input $(\sRL{H}_i, \sRL{v}_i, \bar w_i)$
from $A_p^\ast \times \bN \times S_p$,
this procedure returns an output $(w_i, \sRL{H}_{i-1}, \sRL{v}_{i-1})$
from $\Sigma_p \times A_p^\ast \times \bN$.
By inspection:
\begin{itemize}
\item $w_i$~is either $p$ or~$-1$,
  or, at steps \ref{it:T2inv} and~\ref{it:T5inv}, an integer between 0 and~$p-1$.
  In all cases, $w_i$~is in~$\Sigma_p$.
\item $\sRL{v}_{i-1}$~is either 0 or~$\sRL{v}_i+1$,
  or, at steps \ref{it:T1inv} and~\ref{it:T4inv}, of the form~$\sRL{v}_i-b$
  in situations where~$\sRL{v}_i \geq b$.
  In all cases, $\sRL{v}_{i-1}$~is in~$\bN$.
\item $\sRL{H}_{i-1}$~is always a prefix of~$\sRL{H}_i$,
  possibly augmented by a letter~$b$:
  at step~\ref{it:T3inv}, by $b = a_{\ell,m}$ when $0 \leq \ell+m \leq \bar\mu+m = p-1$;
  at step~\ref{it:T4inv}, by $b = a_{\bar\mu,m}$ when $0 \leq \bar\mu+m = p-1$;
  at step~\ref{it:T6inv}, by $b = a_{\ell,m-1}$ when $m\geq1$
  and $1 \leq \ell+m \leq p-1$.
  In all cases, $\sRL{H}_{i-1}$~is in~$A_p^\ast$.
\end{itemize}

As in the case~$p=1$ of Section~\ref{sec:3-step-model},
we next define a map~$\Psi_p$ from~$S_p^\ast$ to~$\Sigma_p^\ast$ as follows.
Given an input word~$\bar w = \bar w_1\dots \bar w_n$ from~$S_p^\ast$,
iterating the process above
from the initial counter $(\sRL{H}_n, \sRL{v}_n) = (\epsilon, 0)$
results in a word~$w = w_1\dots w_n$,
and $\Psi_p(\bar w)$~is defined as~$w$.

In Figure~\ref{fig:partition},
we provide again, this time in usual set notation, the input and output sets
of the forward transitions \boxref{T1} to~\boxref{T8}
and of the corresponding backward transitions.
The careful reader will realize
that the eight sets in the middle column form
a partition of $\Sigma_p \times A_p^\ast \times \bN$,
while the seven sets in the right-hand column form
a partition of $A_p^\ast \times \bN \times S_p$.
\begin{figure}
\begin{small}
\centerline{\begin{tabular}{c|c|c}
{} & Input: $(\mu, H_L, v_L)$ \phantom{xxxxxxxx} & Output: $(H_R, v_R, \bar\mu)$ \phantom{xxxxxxxx} \\
{} & from $\Sigma_p \times A_p^\ast \times \bN$ & from $A_p^\ast \times \bN \times S_p$ \\
\hline
\boxref{T1} & $\{p\} \times A_p^\ast \times \bN$ & $A_p^\ast \times (p+\bN) \times \{(0,p)\}$ \\
\boxref{T2} & $\{0,\dots,p-1\} \times A_p^\ast \times \{0\}$ & $A_p^\ast \times \{0,\dots,p-1\} \times \{(0,p)\}$ \\
\boxref{T3} & $\{-1\} \times (A_p^\ast\smallsetminus\{\epsilon\}) \times \{0\}$ & $A_p^\ast \times \bigcup_{j=1}^p \{0,\dots,p-j\} \times \{(-j,p-j)\}$ \\
\boxref{T4} & $\{-1\} \times \bigcup_{\ell+m=p-1}A_p^\ast a_{\ell,m} \times (1+\bN)$ & $A_p^\ast \times \bigcup_{j=1}^p (p-j+1+\bN) \times \{(-j,p-j)\}$ \\
\boxref{T5} & $\{0,\dots,p-1\} \times A_p^\ast \times (1+\bN)$ & $\bigcup_{\ell=0}^{p-1} A_p^\ast a_{\ell,0} \times \bN \times \{(1,-1)\}$ \\
\boxref{T6} & $\{-1\} \times \bigcup_{\ell+m\leq p-2} A_p^\ast a_{\ell,m} \times (1+\bN)$ & $\bigcup_{\ell=0}^{p-2}\bigcup_{m=1}^{p-1-\ell} A_p^\ast a_{\ell,m} \times \bN \times \{(1,-1)\}$ \\
\boxref{T7} & $\{-1\} \times \{\epsilon\} \times (1+\bN)$ & $\{\epsilon\} \times \bN \times \{(1,-1)\}$ \\
\boxref{T8} & $\{-1\} \times \{\epsilon\} \times \{0\}$ &
\end{tabular}}
\end{small}
\caption{\label{fig:partition}
  Partitioning the input and output sets of the forward and backward transitions.
  For each column, the sets in the 8 rows partition the set given in the header.}
\end{figure}
Thus, the forward and backward procedures described above
induce inverse bijections between
$(\Sigma_p \times A_p^\ast \times \bN) \smallsetminus \{(-1,\epsilon,0)\}$
and $A_p^\ast \times \bN \times S_p$.

This local reversibility property and the constructions of $\Phi_p$ and~$\Psi_p$
immediately make it possible to identify simple sets of words
on which the restriction of~$\Phi_p$ has a left inverse given by~$\Psi_p$,
and
on which the restriction of~$\Psi_p$ has a left inverse given by~$\Phi_p$.
This is made precise in the following proposition.

\begin{proposition}
If $w \in \Sigma_p^\ast$ is in the domain of definition of~$\Phi_p$,
that is, if \boxref{T8}~is never used in the forward construction applied to~$w$,
and if additionally this construction terminates
with~$(\sLR{H}_n, \sLR{v}_n) = (\epsilon, 0)$,
thus defining some $\bar w = \Phi_p(w) \in S_p^\ast$,
then $\Psi_p$~is well defined at~$\bar w$,
with the property $\Psi_p(\bar w) = w$,
that is, $\Psi_p(\Phi_p(w)) = w$.

Symmetrically, given $\bar w \in S_p^\ast$,
if the backward construction applied to~$\bar w$ terminates
with~$(\sRL{H}_0, \sRL{v}_0) = (\epsilon, 0)$,
thus defining some $w = \Psi_p(\bar w) \in \Sigma_p^\ast$,
then $\Phi_p$~is well defined at~$w$,
with the property $\Phi_p(w) = \bar w$,
that is, $\Phi_p(\Psi_p(\bar w)) = \bar w$.
\end{proposition}





\begin{figure}
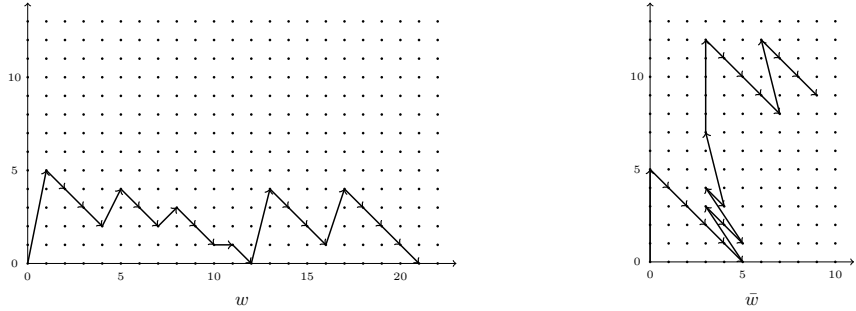


\noindent
\rule{5mm}{0mm}
\resizebox{!}{.1\textheight}{%
\begin{mypic}
\foreach \x in {0,1,2,3,4,5,6,7,8,9,10,11,12,13,14,15,16,17,18,19,20,21,22}
  \foreach \y in {0,1,2,3,4,5,6,7,8,9,10,11,12,13}
    \fill(\x,\y) circle[radius=2pt];
\draw[->](0,0)--(23,0);
\draw[->](0,0)--(0,14);
\draw[thick,->](0,0)--(1,5);
\draw[thick,->](1,5)--(2,4);
\draw[thick,->](2,4)--(3,3);
\draw[thick,->](3,3)--(4,2);
\draw[thick,->](4,2)--(5,4);
\draw[thick,->](5,4)--(6,3);
\draw[thick,->](6,3)--(7,2);
\draw[thick,->](7,2)--(8,3);
\draw[thick,->](8,3)--(9,2);
\draw[thick,->](9,2)--(10,1);
\draw[thick,->](10,1)--(11,1);
\draw[thick,->](11,1)--(12,0);
\draw[thick,->](12,0)--(13,4);
\draw[thick,->](13,4)--(14,3);
\draw[thick,->](14,3)--(15,2);
\draw[thick,->](15,2)--(16,1);
\draw[thick,->](16,1)--(17,4);
\draw[thick,->](17,4)--(18,3);
\draw[thick,->](18,3)--(19,2);
\draw[thick,->](19,2)--(20,1);
\draw[thick,->](20,1)--(21,0);
\foreach \x in {0,5,10,15,20} \draw(\x,-.7) node{\tiny\x};
\foreach \y in {0,5,10} \draw(-.7,\y) node{\tiny\y};
\draw(11.5,-2) node{$w$};
\end{mypic}
}
\hfill
\resizebox{!}{.1\textheight}{%
\begin{mypic}
\foreach \x in {0,1,2,3,4,5,6,7,8,9,10}
  \foreach \y in {0,1,2,3,4,5,6,7,8,9,10,11,12,13}
    \fill(\x,\y) circle[radius=2pt];
\draw[->](0,0)--(11,0);
\draw[->](0,0)--(0,14);
\draw[thick,->](0,0)--(0,5);
\draw[thick,->](0,5)--(1,4);
\draw[thick,->](1,4)--(2,3);
\draw[thick,->](2,3)--(3,2);
\draw[thick,->](3,2)--(4,1);
\draw[thick,->](4,1)--(5,0);
\draw[thick,->](5,0)--(3,3);
\draw[thick,->](3,3)--(4,2);
\draw[thick,->](4,2)--(5,1);
\draw[thick,->](5,1)--(3,4);
\draw[thick,->](3,4)--(4,3);
\draw[thick,->](4,3)--(3,7);
\draw[thick,->](3,7)--(3,12);
\draw[thick,->](3,12)--(4,11);
\draw[thick,->](4,11)--(5,10);
\draw[thick,->](5,10)--(6,9);
\draw[thick,->](6,9)--(7,8);
\draw[thick,->](7,8)--(6,12);
\draw[thick,->](6,12)--(7,11);
\draw[thick,->](7,11)--(8,10);
\draw[thick,->](8,10)--(9,9);
\foreach \x in {0,5,10} \draw(\x,-.7) node{\tiny\x};
\foreach \y in {0,5,10} \draw(-.7,\y) node{\tiny\y};
\draw(5.5,-2) node{$\bar w$};
\end{mypic}
}
\rule{5mm}{0mm}

\caption{An example of a Łukasiewicz walk~$w$ (left)
  and of the corresponding $p$-tandem word~$\bar w$ (right),
  for $p=5$ and length~21.}
\label{fig:bij-ex-p=5}

\end{figure}

\begin{figure}
\noindent
\begin{tikzpicture}[scale = .55]
\foreach \x / \t in {-/$w$, 1/5, 2/D, 3/D, 4/D, 5/2, 6/D, 7/D, 8/1, 9/D, 10/D, 11/0, 12/D, 13/4, 14/D, 15/D, 16/D, 17/3, 18/D, 19/D, 20/D, 21/D}
  \draw (\x.5,5) node {\t};
\foreach \x / \t in {-/tr, 1/T1, 2/T7, 3/T7, 4/T7, 5/T5, 6/T6, 7/T3, 8/T5, 9/T6, 10/T3, 11/T5, 12/T3, 13/T2, 14/T7, 15/T7, 16/T7, 17/T5, 18/T3, 19/T7, 20/T7, 21/T7}
  \draw (\x.5,4) node {\small\t};
\foreach \x / \t in {-.5/$H$,  1/$\epsilon$, 2/$\epsilon$, 3/$\epsilon$, 4/$\epsilon$, 5/$\epsilon$, 6/$a_{2,0}$, 7/$a_{2,1}$, 8/$\epsilon$, 9/$a_{1,0}$, 10/$a_{1,1}$, 11/$\epsilon$, 12/$a_{0,0}$, 13/$\epsilon$, 14/$\epsilon$, 15/$\epsilon$, 16/$\epsilon$, 17/$\epsilon$, 18/$a_{3,0}$, 19/$\epsilon$, 20/$\epsilon$, 21/$\epsilon$, 22/$\epsilon$}
  \draw (\x,3) node {\small\t};
\foreach \x / \t in {-.5/$v$, 1/0, 2/5, 3/4, 4/3, 5/2, 6/1, 7/0, 8/2, 9/1, 10/0, 11/1, 12/0, 13/0, 14/4, 15/3, 16/2, 17/1, 18/0, 19/3, 20/2, 21/1, 22/0}
  \draw (\x,2) node {\small\t};
\foreach \x / \t in {-/$\bar w$, 1/5, 2/D, 3/D, 4/D, 5/D, 6/D, 7/3, 8/D, 9/D, 10/3, 11/D, 12/4, 13/5, 14/D, 15/D, 16/D, 17/D, 18/4, 19/D, 20/D, 21/D}
  \draw (\x.5,1) node {\t};
\foreach \x / \t in {-.5/$i$, 1/0, 2/1, 3/2, 4/3, 5/4, 6/5, 7/6, 8/7, 9/8, 10/9, 11/10, 12/11, 13/12, 14/13, 15/14, 16/15, 17/16, 18/17, 19/18, 20/19, 21/20, 22/21}
  \draw (\x,0) node {\tiny\t};
\foreach \x / \t in {                                                                                      2/2}
  \draw (\x,-1) node {\scriptsize\t};
\foreach \x / \t in {                                          14/14,                                 3/3, 2/3}
  \draw (\x,-1.5) node {\scriptsize\t};
\foreach \x / \t in {              19/19,               15/15, 14/15,                            4/4, 3/4, 2/4}
  \draw (\x,-2) node {\scriptsize\t};
\foreach \x / \t in {       20/20, 19/20,        16/16, 15/16, 14/16,             8/8,      5/5, 4/5, 3/5, 2/5}
  \draw (\x,-2.5) node {\scriptsize\t};
\foreach \x / \t in {21/21, 20/21, 19/21, 17/17, 16/17, 15/17, 14/17, 11/11, 9/9, 8/9, 6/6, 5/6, 4/6, 3/6, 2/6, -.5/$V$}
  \draw (\x,-3) node {\scriptsize\t};
\foreach \x in {1, 7, 10, 12, 13, 18, 22}
  \draw (\x,-3) node {\scriptsize$\epsilon$};
\foreach \x in {.5} \draw(\x,-3.5)--(\x,5.5);
\foreach \y in {4.5, 1.5, .5} \draw(-1,\y)--(22.5,\y);
\end{tikzpicture}
\caption{\label{fig:transduction-ex-p=5}
  Transduction corresponding to Figure~\ref{fig:bij-ex-p=5} (top rows),
  augmented with the stack to be used in Section~\ref{sec:raising-p-gt-1} (bottom row).
  The letters~$\mu$ in row~$w$ are the input to the LR process as well as the output from the RL process.
  They give the Łukasiewicz walk~$w$ as the succession of steps~$(1,\mu)$, with D~standing for~$-1$.
  The letters~$\bar\mu$ in row~$\bar w$ are the output from the LR process, the input to the RL process.
  They give the $p$-tandem word~$\bar w$ as the succession of steps~$(\bar\mu-5,\bar\mu)$, or $(1,-1)$~if $\bar\mu={}$D.
  The transitions used are indicated by row~tr.
  How they transform the generalized counters $H$ and~$v$ is provided by two more rows.
  Along this run, $H$~remains either the empty word~$\epsilon$ or is a word consisting of a single letter~$a_{\ell,m}$.
  The last row provides the evolution of the stack~$V$ of the augmented RL transducer process of Section~\ref{sec:raising-p-gt-1}, with pushes to the top of the depicted columns ($\epsilon$~denoting the empty stack).
  Observe that the length of~$V$ constantly matches~$v$.}
\end{figure}

\begin{example}\label{ex:bij-ex-p=5}
Figure~\ref{fig:bij-ex-p=5} provides an example for~$p=5$:
on the left side, the 5-Łukasiewicz walk
\[ w = (5,-1,-1,-1,2,-1,-1,1,-1,-1,0,-1,4,-1,-1,-1,3,-1,-1,-1,-1) , \]
where each letter~$\mu$ represents the step~$(1,\mu)$;
on the right side, the corresponding 5-tandem walk, obtained by applying~$\Phi_5$,
\[ \bar w = (5,-1,-1,-1,-1,-1,3,-1,-1,3,-1,4,5,-1,-1,-1,-1,4,-1,-1,-1) , \]
where each letter~$\bar\mu$ represents the step~$(\bar\mu-5,\bar\mu)$
if $\bar\mu\geq0$,
and $(1,-1)$ if~$\bar\mu=-1$.
Figure~\ref{fig:transduction-ex-p=5} provides
the transitions used and the evolution of the counters $H$ and~$v$
along the construction of the 5-tandem walk from the 5-Łukasiewicz walk.
\end{example}

\subsection{Parameters depending on the constructions, and their variations}
\label{sec:parameters}

The counters $(H, v)$ of Section~\ref{sec:cons-maps}
(whether it be in LR or in RL form)
are needed explicitly
for the constructions of $\Phi_p$ and~$\Psi_p$.
In Section~\ref{sec:invariants}, we will introduce further parameters
that could be computed, along with $H$ and~$v$, over the course of the constructions.
Yet, these parameters are only needed for proving properties of the constructions, which is why they are not made part of the definition of the construction in the way that $H$ and~$v$ are.
A notational difficulty in proofs is that we need to compare
the counters, outputs, and parameters as computed by the left-to-right process~$\Phi_p$
with those computed by the right-to-left process~$\Psi_p$,
and in fact several sorts of variations in the parameters.
We introduce such variations in abstract in the present section,
before defining specific parameters in Section~\ref{sec:invariants}.
(We will freely use the abbreviations LR and RL to refer to the two processes.)

We start with notation to distinguish the counters $(H, v)$ and the words
produced by the two processes.
Given a $p$-Łukasiewicz word $w \in \Sigma_p^n$,
for any integer~$i$ satisfying $1 \leq i \leq n$
we denote by $\LR{H}{w}{i}$ and~$\LR{v}{w}{i}$
the values of $H$ and~$v$ obtained by the LR process after its $i$th step,
that is, after using the letter~$w_i$ from the input word~$w$,
and by~$\LR{\bar w}{w}{i}$ the letter output at this step.
Similarly, given a quarter-plane word $\bar w \in S_p^n$,
for any integer~$i$ satisfying $0 \leq i \leq n-1$
we denote by $\RL{H}{\bar w}{i}$ and~$\RL{v}{\bar w}{i}$
the values of $H$ and~$v$ obtained by the RL process after its $(n-i)$th step,
that is, after using the letter~$\bar w_{i+1}$ from the input word~$\bar w$,
and by~$\RL{w}{\bar w}{i+1}$ the letter output at this step.
To get total maps over~$0 \leq i \leq n$, we also set
$\LR{H}{w}{0} = \RL{H}{\bar w}{n} = \epsilon$
and~$\LR{v}{w}{0} = \RL{v}{\bar w}{n} = 0$.
(Note that
$(\LR{H}{w}{i}, \LR{v}{w}{i})$ and $(\RL{H}{\bar w}{i}, \RL{v}{\bar w}{i})$
were respectively denoted $(\sLR{H}_i, \sLR{v}_i)$ and $(\sRL{H}_i, \sRL{v}_i)$
in Section~\ref{sec:cons-maps}.)

The parameters we are interested in depend
on the words input to the constructing processes,
in~$\Sigma_p^\ast$~for~LR and in~$S_p^\ast$~for~RL,
on the words output by the same processes,
in~$S_p^\ast$~for~LR and in~$\Sigma_p^\ast$~for~RL,
as well as on the counter $(H, v)$ at each step in the process.
So, we formally introduce a \emph{parameter}~$q$
as a tuple~$(\alpha, \bar\alpha, \sigma_1, \sigma_2)$
of monoid morphisms with integer values.
That is, we assume that the four maps satisfy formulas
\begin{equation*}
\alpha(w) = \sum_{j=1}^i \alpha(w_j) , \quad
\bar\alpha(\bar w) = \sum_{j=1}^i \bar\alpha(\bar w_j) , \quad
\sigma_1(H) = \sum_{j=1}^i \sigma_1(\eta_i) , \quad
\sigma_2(v) = v \sigma_2(1) .
\end{equation*}
for all $i \in \bN$,
$w = w_1\dots w_i \in \Sigma_p^i$,
$\bar w = \bar w_1\dots \bar w_i \in S_p^i$,
$H = \eta_1\dots\eta_i \in A_p^i$,
and $v \in \bN$.
After fixing a length~$n$ for words,
we associate two sequences to the parameter~$q$,
which we name the \emph{values of~$q$ along the LR and RL processes}, respectively:
\begin{itemize}
\item For a $p$-Łukasiewicz word $w \in \Sigma_p^n$ and $0 \leq i \leq n$,
the value $\LR{q}{w}{i}$ of~$q$ obtained
after $i$~steps in the construction~$\Phi_p$
is given, assuming no error occurs, as
\begin{equation*}
\LR{q}{w}{i} =
  \alpha(w_1\dots w_i) + \bar\alpha(\LR{\bar w}{w}{1}\dots\LR{\bar w}{w}{i})
  + \sigma_1(\LR{H}{w}{i}) + \sigma_2(\LR{v}{w}{i}) .
\end{equation*}
\item Similarly, for a quarter-plane word $\bar w \in S_p^n$ and $0 \leq i \leq n$,
the value $\RL{q}{\bar w}{i}$ of~$q$ obtained
after $n-i$~steps in the construction~$\Psi_p$
is given as
\begin{equation*}
\RL{q}{\bar w}{i} =
  \alpha(\RL{w}{\bar w}{i+1}\dots\RL{w}{\bar w}{n}) + \bar\alpha(\bar w_{i+1}\dots\bar w_n)
  - \sigma_1(\RL{H}{\bar w}{i}) - \sigma_2(\RL{v}{\bar w}{i}) .
\end{equation*}
(Negative signs may look counter-intuitive.
They relate to RL's operating transitions backwards.)
\end{itemize}
Observe that $\LR{q}{w}{i}$~may be undefined,
as $\Phi_p$~may lead to error.
It will be our goal to relate
$\LR{q}{w}{i}$ and~$\RL{q}{\bar w}{i}$
when $w$ and~$\bar w$ are properly related by $\Phi_p$ and~$\Psi_p$
(and for suitable parameters~$q$).

Our plan is next to describe how parameters vary along the processes LR and RL,
and how (LR and RL) variations in parameters decompose into
variations across transitions.
We first define two kinds of variation operators on parameters:
\begin{itemize}
\item For a $p$-Łukasiewicz word $w \in \Sigma_p^n$ and $0 \leq i \leq j \leq n$,
  the LR-variation in~$q$ is defined as
  \[ \DLR{w}{i}{j}q = \LR{q}{w}{j} - \LR{q}{w}{i} , \]
  which is equal to
  \begin{equation*}
    \alpha(w_{i+1}\dots w_j) + \bar\alpha(\LR{\bar w}{w}{i+1}\dots\LR{\bar w}{w}{j})
    + \sigma_1(\LR{H}{w}{j}) - \sigma_1(\LR{H}{w}{i}) + \sigma_2(\LR{v}{w}{j} - \LR{v}{w}{i}) .
  \end{equation*}
\item Similarly, for a quarter-plane word $\bar w \in S_p^n$ and $0 \leq i \leq j \leq n$,
  the RL-variation in~$q$ is defined as
  \[ \DRL{\bar w}{i}{j}q = \RL{q}{\bar w}{i} - \RL{q}{\bar w}{j} , \]
  which is equal to
  \begin{equation*}
    \alpha(\RL{w}{\bar w}{i+1}\dots\RL{w}{\bar w}{j}) + \bar\alpha(\bar w_{i+1}\dots \bar w_j)
    + \sigma_1(\RL{H}{\bar w}{j}) - \sigma_1(\RL{H}{\bar w}{i}) + \sigma_2(\RL{v}{\bar w}{j} - \RL{v}{\bar w}{i}) .
  \end{equation*}
\end{itemize}

An extreme case is obtained when $i$ and~$j$ differ by~1, in which case
\begin{equation*}
  \DLR{w}{i-1}{i}q =
    \alpha(w_i) + \bar\alpha(\LR{\bar w}{w}{i})
    + \sigma_1(\LR{H}{w}{i}) - \sigma_1(\LR{H}{w}{i-1}) + \sigma_2(\LR{v}{w}{i} - \LR{v}{w}{i-1})
\end{equation*}
and
\begin{equation*}
  \DRL{\bar w}{i-1}{i}q =
    \alpha(\RL{w}{\bar w}{i}) + \bar\alpha(\bar w_i)
    + \sigma_1(\RL{H}{\bar w}{i}) - \sigma_1(\RL{H}{\bar w}{i-1}) + \sigma_2(\RL{v}{\bar w}{i} - \RL{v}{\bar w}{i-1}) .
\end{equation*}
Given two (independent) words $w$ and~$\bar w$ of same length~$n$,
it can happen that
the $i$th step of the LR process at~$w$
and the $(n+1-i)$th step of the RL process at~$\bar w$
use the same transition (in reverse forms).
That is, a possible situation is the existence of
a tuple $\tau = (\mu, H_L, v_L, H_R, v_R, \bar a)$ satisfying
\begin{alignat*}{3}
w_i = \RL{w}{\bar w}{i} &= \mu , &\quad
\LR{H}{w}{i-1} = \RL{H}{\bar w}{i-1} &= H_L , &\quad
\LR{v}{w}{i-1} = \RL{v}{\bar w}{i-1} &= v_L , \\
\LR{H}{w}{i} = \RL{H}{\bar w}{i} &= H_R , &\quad
\LR{v}{w}{i} = \RL{v}{\bar w}{i} &= v_R , &\quad
\LR{\bar w}{w}{i} = \bar w_i &= \bar a ,
\end{alignat*}
in which case $\DLR{w}{i-1}{i}q$ and~$\DRL{\bar w}{i-1}{i}q$ are equal.
In such a situation, we call~$\tau$ a \emph{transition instance}
(or more specifically an instance of a named transition).
We can then define the \emph{variation in~$q$ across~$\tau$}
as the common value
\begin{equation*}
\Delta_\tau q =
  \alpha(\mu) + \bar\alpha(\bar a) + \sigma_1(H_R) - \sigma_1(H_L) + \sigma_2(v_R - v_L) .
\end{equation*}
Obviously, $\Delta_\tau q$~depends on the 6-tuple~$\tau$,
but neither on the location~$i$ of its use,
nor on the context of the entire
 $p$-Łukasiewicz walk or $p$-tandem quarter-plane walk being processed.
Furthermore, both $\DLR{w}{i}{j}q$ and~$\DRL{\bar w}{i}{j}q$ decompose
as sums of~$\Delta_\tau q$'s, for the successive~$\tau$'s set up by LR or~RL.

\subsection{Specific parameters and some of their invariants}
\label{sec:invariants}

The specific parameters we are about to define are of two types:
first, parameters on counters;
second, parameters on input and output words.
Our proofs will require parameters that are linear combinations of the two types.

For the parameters on counters,
we introduce three monoid morphisms from~$A_p^\ast$ to~$\bN$.
Namely, we define $|\cdot|$, $\wt_\ell$, and~$\wt_m$,
by, respectively,
$|a_{\ell,m}| = 1$, $\wt_\ell(a_{\ell,m}) = \ell+1$,
and $\wt_m(a_{\ell,m}) = m+1$
for each~$a_{\ell,m} \in A_p$.
We obtain three parameters, named respectively $\lambda$, $\zeta$, and~$\xi$,
by setting $\alpha$, $\bar\alpha$, and~$\sigma_2$ to the zero map,
and defining~$\sigma_1$ by, respectively,
\begin{equation*}
\sigma_1(H) =  |H| , \qquad
\sigma_1(H) =  \wt_\ell(H) , \qquad
\sigma_1(H) =  \wt_m(H) .
\end{equation*}
A fourth parameter depending on the counters, named~$v$, is obtained
by setting~$\sigma_2$ to the identity
and $\alpha$, $\bar\alpha$, and~$\sigma_1$ to the zero map.
We state the following proposition with no reference to parameters,
but it clearly expresses constraints
for the first three parameters above to cancel.
\begin{proposition}\label{prop ways to be empty}
For any $H \in A_p^\ast$, the following four properties are equivalent:
$|H| = 0$; $\wt_\ell(H) = 0$; $\wt_m(H) = 0$; $H = \epsilon$.
\end{proposition}

The simplest parameters depending on input and output words we use
are given as coordinates of the vertices of the walks.
A first parameter,~$z$, is obtained by setting $\bar\alpha = \sigma_1 = \sigma_2 = 0$
and by defining $\alpha(\mu) = \mu$ for any letter~$\mu \in \Sigma_p$.
Two parameters, $\bar x$ and~$\bar y$, respectively, are obtained
by setting $\alpha = \sigma_1 = \sigma_2 = 0$ and by defining $\bar\alpha(\bar a)$
to be the first coordinate~$\bar a_1$,
respectively the second coordinate~$\bar a_2$
of any letter~$\bar a = (\bar a_1, \bar a_2) \in S_p$.
In particular, for any $p$-Łukasiewicz word~$w \in S_p^n$
we have the formulas
\begin{equation*}
(\LR{\bar x}{w}{i}, \LR{\bar y}{w}{i}) = \sum_{1 \leq j \leq i} \LR{\bar w}{w}{j} ,
\qquad
\LR{z}{w}{i} = \sum_{1 \leq j \leq i} w_j ,
\end{equation*}
so that the~$(i, \LR{z}{w}{i})$ form a walk
that remains in the upper half plane,
and the $(\LR{\bar x}{w}{i}, \LR{\bar y}{w}{i})$ form
a (free) $p$-tandem walk associated to
$\LR{\bar w}{w}{1}\dots\LR{\bar w}{w}{n} \in \Sigma_p^n$.
Another parameter, named~$\bar k$, measures
how far the $p$-tandem walk is from the main diagonal,
and is obtained by setting $\alpha = \sigma_1 = \sigma_2 = 0$
and by defining $\bar\alpha(\bar a)$
to be the difference~$\bar a_2 - \bar a_1$,
for any letter~$\bar a = (\bar a_1, \bar a_2) \in S_p$.
Specifically,
\begin{equation*}
\LR{\bar k}{w}{i} = \LR{\bar y}{w}{i} - \LR{\bar x}{w}{i}
  = \sum_{j=1}^n \left(\LR{\bar w}{w}{j}\right)_2 - \left(\LR{\bar w}{w}{j}\right)_1 .
\end{equation*}
Note that, across one step in the LR process,
the parameter~$\bar k$
either increases by~$p$,
if the letter~$\LR{\bar w}{w}{j}$ is one of the $(-j,p-j)$ for $0\leq j \leq p$,
or decreases by~2, if it is $(1,-1)$.
We mimic this definition on the walk associated to~$w$,
by setting now $\bar\alpha = \sigma_1 = \sigma_2 = 0$
and by defining the monoid morphism~$\alpha$ from~$\Sigma_p^\ast$ to~$\bN$
defined by $\alpha(\mu) = p$ if $0\leq \mu \leq p$ and $\alpha(-1) = -2$.
This defines a parameter, named~$k$.
Two additional parameters,
which we will need specifically for the analysis of the RL process,
keep track of the cumulative amount
by which the inequalities given below in~\eqref{eq:Delta-xy} are strict.
They can be understood visually by noting that under most transitions,
$\bar y$~and~$v$ move in parallel, as do $\bar x$ and~$\xi$.
The parameters we want, named $\bar r$ and~$\bar s$ below,
are the shifts separating the two parallel functions in each case.
More precisely,
define a parameter~$\bar r$ by setting
  \[ \alpha(\mu) = 0, \quad \bar\alpha(\bar a) = \bar a_1, \quad \sigma_1(H) = -\wt_m(H), \quad \sigma_2(v) = 0. \]
Similarly,
define a parameter~$\bar s$ by setting
  \[ \alpha(\mu) = 0, \quad \bar\alpha(\bar a) = \bar a_2, \quad \sigma_1(H) = 0, \quad \sigma_2(v) = -v. \]
Thus, the following formulas hold for $\bar w \in S_p^\ast$:
\begin{equation}\label{eq:r-bar-s-bar}
\RL{\bar r}{\bar w}{i} = \RL{\bar x}{\bar w}{i} - \RL{\xi}{\bar w}{i} ,
\qquad
\RL{\bar s}{\bar w}{i} = \RL{\bar y}{\bar w}{i} - \RL{v}{\bar w}{i} .
\end{equation}

It will prove convenient in what follows
to write identities like the last one
in a more compact, while more general, form.
To this end, remember that
our parameters are defined through morphisms to~$\bZ$.
Thus, the operations in~$\bZ$
of addition, subtract, and multiplication by an integer,
transpose in an obvious manner
into operations on morphisms, then into operations on parameters.
For instance, we could have defined $\bar r$ and~$\bar s$
as differences of parameters,
via the formulas $\bar r = \bar x - \xi$ and $\bar s = \bar y - v$.
These conveniently suggest the identities~\eqref{eq:r-bar-s-bar},
as well as similar identities for the RL process and for variations.

\bigskip

The variations along the walks in the parameters introduced so far depend
only on the transition instance
$\tau = (\mu, H_L, v_L, H_R, v_R, \bar\mu)$,
so we can define and calculate~$\Delta_\tau$ for each of them.
As it turns out, the variations actually do not depend on the specific
transition instance~$\tau$,
but simply on the transition, one of \boxref{T1} to~\boxref{T8}.
The results are given in the table in Figure~\ref{fig:variations-gen-p}.

\begin{figure}

\centerline{\small
\medskip
\begin{tabular}{c|ccccccc}
$\tau$ inst.\ of          & \boxref{T1} & \boxref{T2} & \boxref{T3} & \boxref{T4} & \boxref{T5} & \boxref{T6} & \boxref{T7} \\
\hline
$\Delta_\tau \bar x$           & 0    & 0           & $-m-1$          & $-m-1$       & 1           & 1               & 1     \\
$\Delta_\tau \bar y$           & $p$  & $p$         & $p-m-1$         & $p-m-1$      & $-1$        & $-1$            & $-1$  \\
$\Delta_\tau z$           & $p$  & $q$         & $-1$            & $-1$         & $q$         & $-1$            & $-1$  \\
$\Delta_\tau v$           & $p$  & $q$         & $\ell$          & $\ell$       & $-1$        & $-1$            & $-1$  \\
$\Delta_\tau \lambda$         & 0    & 0           & $-1$            & $-1$         & 1           & 0               & 0     \\
$\Delta_\tau \zeta$ & 0    & 0           & $-\ell-1$       & $-\ell-1$    & $q+1$       & 0               & 0     \\
$\Delta_\tau \xi$    & 0    & 0           & $-m-1$          & $-m-1$       & 1           & 1               & 0     \\
$\Delta_\tau k$           & $p$  & $p$         & $-2$            & $-2$         & $p$         & $-2$            & $-2$  \\
$\Delta_\tau \bar k$          & $p$  & $p$         & $p$             & $p$          & $-2$        & $-2$            & $-2$  \\
$\Delta_\tau \bar r$           & 0    & 0           & 0               & 0            & 0           & 0               & 1     \\
$\Delta_\tau \bar s$           & 0    & $p-q$       & $p-m-1-\ell$    & 0            & 0           & 0               & 0     \\
\hline
\multirow{2}{*}{hypothesis}
                     &      & \myclap{$p-1\geq q$} & & \myclap{$p-m-1=\ell$} & & \myclap{$p-m-2\geq\ell$} & \\
                     &      & & \myclap{$p-m-1\geq\ell$} & & \myclap{$p-1\geq q$} & &
\end{tabular}}

\caption{\label{fig:variations-gen-p}
  Variations in the parameters across the transitions
  defined in Figure~\ref{fig:phi-p}
  for the bijections between $p$-Łukasiewicz and $p$-tandem walks.}
\end{figure}

\begin{proposition}\label{prop inequalities}
For each~$\tau$ and $\Delta = \Delta_\tau$, the following inequalities and equalities hold:
\begin{gather}
\label{eq:Delta-xy}
\Delta \bar x \geq \Delta \xi , \quad
\Delta \bar y \geq \Delta v , \\
\label{eq:Delta-zkbark}
\Delta z = \Delta \zeta + \Delta v , \quad
\Delta k - \Delta\bar k = (p+2) \Delta \lambda , \\
\label{eq:Delta-barxbary}
\Delta \bar x = \Delta \bar r + \Delta \xi , \quad
\Delta \bar y = \Delta \bar s + \Delta v .
\end{gather}

Additionally, the inequalities~\eqref{eq:Delta-xy} are strict in precisely the following circumstances:
\begin{itemize}
\item
  $\Delta_\tau \bar x > \Delta_\tau \xi$ only if $\tau$~is an instance of~\boxref{T7},
  implying $H_L = H_R = \epsilon$;
\item
  $\Delta_\tau \bar y > \Delta_\tau v$ only if $\tau$~is an instance of \boxref{T2} or~\boxref{T3},
  implying $v_L = 0$.
\end{itemize}
\end{proposition}

\begin{proof}
By inspection of the table.  (The equalities of~\eqref{eq:Delta-barxbary} are really just
restatements of the definitions of $\bar r$ and~$\bar s$.)
\end{proof}

Since these quantities depend only on the transitions, we also have that
for given $w \in \Sigma_p^\ast$, $\bar w \in S_p^\ast$, and
$0 \leq i \leq j \leq n$,
the variations $\DLR{w}{i}{j}$ and $\DRL{\bar w}{i}{j}$
are sums of $\Delta_\tau$'s operating on the letters of $w$ and $\bar w$.
It follows by induction that the two equalities and two inequalities
in \eqref{eq:Delta-xy} and~\eqref{eq:Delta-zkbark}
still hold for $\Delta = \DLR{w}{i}{j}$ and $\Delta = \DRL{\bar w}{i}{j}$.

Observe that $\Delta \bar r$ and $\Delta \bar s$ are both nonnegative
for all choices of~$\Delta$
as a~$\Delta_\tau$, a~$\DLR{w}{i}{j}$, or a~$\DRL{\bar w}{i}{j}$,
provided~$i<j$.

\subsection{Properties of the constructed walks}
\label{sec:prop-of-contrd-walks}

\begin{theorem}\label{thm:bij-Phi-p-Psi-p}
The map $\Phi_p$ is a bijection
between $p$-Łukasiewicz walks and $p$-tandem walks
that preserves the total number of steps,
maps $k$ to~$\bar k$,
and maps the number of\/~$-1$ steps to the number of\/~\ldir{3bb} steps.
The map~$\Psi_p$ is the inverse bijection.
\end{theorem}

\begin{proof}
We first prove that $\Phi_p$ maps
prefixes of $p$-Łukasie\hyph wicz walks to quarter-plane walks
and that applying $\Phi_p$ to a $p$-Łukasiewicz walk in its entirety
results in $(v,H) = (0,\epsilon)$ at the end of the algorithm.

To this end, again take $w \in \Sigma_p^n$ and $\bar w = \Phi_p(w)$,
both of length~$n$.

Note that we have the initial conditions $\LR{\bar x}{w}{0} = \LR{\bar y}{w}{0} = \LR{z}{w}{0} = \LR{v}{w}{0} = \LR{k}{w}{0} = \LR{\bar k}{w}{0} = \LR{\lambda}{w}{0} = 0$.
By~\eqref{eq:Delta-xy} we have for all~$i$ that
$\DLR{w}{i-1}{i}\bar x \geq \DLR{w}{i-1}{i}\xi$.  Summing this with the initial conditions we get that for all~$i$
\[
\LR{\bar x}{w}{i} \geq \LR{\xi}{w}{i} .
\]
Arguing likewise from~\eqref{eq:Delta-xy} and~\eqref{eq:Delta-zkbark} we have
\begin{equation*}
  \LR{\bar y}{w}{i} \geq \LR{v}{w}{i} , \qquad
  \LR{z}{w}{i} = \LR{\zeta}{w}{i} + \LR{v}{w}{i} , \qquad
  \LR{k}{w}{i} - \LR{\bar k}{w}{i} = (p+2)\LR{\lambda}{w}{i} .
\end{equation*}

A first consequence is that $ \LR{\bar y}{w}{i}$ and~$\LR{\bar x}{w}{i}$ are nonnegative for all $i$, as they are both bounded below by lengths, and so
the $p$-tandem walk associated to~$\bar w$ remains in the quarter plane.
Second, $\LR{z}{w}{i}$~is the sum of two nonnegative integers,
and more precisely $\LR{z}{w}{i} = 0$ if and only if $(\LR{\zeta}{w}{i}, \LR{v}{w}{i}) = (0, 0)$.
As a consequence, by Proposition~\ref{prop ways to be empty}, $w$~is a $p$-Łukasiewicz word if and only if
the construction of~$\bar w = \Phi_p(w)$ ends with $H=\epsilon$ and~$v=0$.
Additionally, if $w$~is a $p$-Łukasiewicz word,
then $w = \Psi_p(\Phi_p(w))$
by the remark at the end of Section~\ref{sec:cons-maps}.
Finally, also $\LR{\lambda}{w}{n} =0$,
because $\LR{H}{w}{n} = \epsilon$,
and so
\begin{equation}\label{eq:phi-preserves-k-bark}
\LR{k}{w}{n} = \LR{\bar k}{w}{n} .
\end{equation}

\medskip

Next we will prove that $\Psi_p$~maps
quarter-plane walks in~$S_p^\ast$ to suffixes of $p$-Łukasiewicz walks
and that we have $(H,v) = (\epsilon,0)$ at the end of the algorithm.
To this end, now take $\bar w \in S_p^\ast$ and $w = \Psi_p(\bar w)$, both of length~$n$.

By~\eqref{eq:Delta-zkbark} we have for all $i$ that
\[
\DRL{\bar w}{i}{i+1}{z} = \DRL{\bar w}{i}{i+1}{\zeta} + \DRL{\bar w}{i}{i+1}{v}.
\]
Next, summing we get that $\RL{z}{\bar w}{i}$ is a sum of nonnegative terms and so is nonnegative for all~$i$.
Thus every suffix of~$w$ remains in the upper half plane.

Next consider $\RL{\bar r}{\bar w}{i}$ and $\RL{\bar s}{\bar w}{i}$.
All the entries for $\bar r$ and $\bar s$ in the table are nonnegative
and so $\RL{\bar r}{\bar w}{i} \geq \RL{\bar r}{\bar w}{j}$ for $i\leq j$,
and likewise for $\RL{\bar s}{\bar w}{i}$.
In other words, $\RL{\bar r}{\bar w}{i}$ and $\RL{\bar s}{\bar w}{i}$
are weakly decreasing as functions of~$i$.
Consequently $\RL{\bar r}{\bar w}{0}$ is the maximum value of
$\{ \RL{\bar r}{\bar w}{i} : 0 \leq i \leq n \}$.
Let $j$ be maximal such that $\RL{\bar r}{\bar w}{j} = \RL{\bar r}{\bar w}{0}$.
As observed in the previous section,
$\RL{\bar r}{\bar w}{i}$~can only change under transition~\boxref{T7}
and the stack must be empty for this transition.
Thus the choice of~$j$ implies by Proposition~\ref{prop ways to be empty}
that~$\RL{\xi}{\bar w}{j}=0$.
By definition of~$\bar r$,
for the remainder of the algorithm, from index~$j$ down to index~$0$,
$\RL{\bar x}{\bar w}{j'}$ and $\RL{\xi}{\bar w}{j'}$ move in parallel
with a fixed shift of $\RL{\bar r}{\bar w}{0}$.

At the end of the algorithm,
$\Psi_p$ has consumed all of~$\bar w$ from its end to its beginning,
and so we have $\RL{\bar x}{\bar w}{0} = 0$.
Since $\bar w$~is a quarter-plane walk, $\RL{\bar x}{\bar w}{j} \geq 0$, and by the previous paragraph $\RL{\bar x}{\bar w}{j} - \RL{\xi}{\bar w}{j} = \RL{\bar r}{\bar w}{0} = \RL{\bar x}{\bar w}{0} - \RL{\xi}{\bar w}{0}$.
We get that
\begin{align*}
  \RL{\xi}{\bar w}{0} & = \RL{\bar x}{\bar w}{0} - \RL{\bar x}{\bar w}{j} + \RL{\xi}{\bar w}{j} = - \RL{\bar x}{\bar w}{j} \leq 0.
\end{align*}
However $m$-weight is always nonnegative and hence $\RL{\xi}{\bar w}{0}=0$ and so, by Proposition~\ref{prop ways to be empty}, $H$~is empty at the end of the process.

We can make an analogous argument
with $\bar y$ in place of~$\bar x$ and $v$ in place of~$\wt_m(H)$.
Specifically, now let $j$ be maximal such that
$\RL{\bar s}{\bar w}{j} = \RL{\bar s}{\bar w}{0}$.
Since $\bar s$ can only change under transitions \boxref{T2} and~\boxref{T3}, $\RL{v}{\bar w}{j}=0$.
For the remainder of the algorithm, from index~$j$ down to index~$0$,
$\RL{\bar y}{\bar w}{j'}$ and~$\RL{v}{\bar w}{j'}$
move in parallel with a fixed shift of~$\RL{\bar s}{\bar w}{0}$.
We have $\RL{\bar y}{\bar w}{0}=0$, and,
since $\bar w$ is a quarter-plane walk, $\RL{\bar y}{\bar w}{j} \geq 0$.
Additionally, by the definition of~$j$,
$\RL{\bar y}{\bar w}{j} - \RL{v}{\bar w}{j} = \RL{\bar y}{\bar w}{0} - \RL{v}{\bar w}{0}$
and so we get $\RL{v}{\bar w}{0} = - \RL{\bar y}{\bar w}{j} \leq 0$.
However, $v$~is always nonnegative so $\RL{v}{\bar w}{0} = 0$.

From all this we conclude that $(H,v) = (\epsilon,0)$
after applying~$\Psi_p$ to~$\bar w$
and that $\RL{z}{\bar w}{0}=0$.
Thus $w$~is a $p$-Łukasiewicz word and $\bar w=\Phi_p(\Psi_p(\bar w))$.

Therefore for each $n$, $\Phi_p$ and $\Psi_p$ are mutually inverse bijections between $p$-Łukasiewicz walks of length $n$ and quarter-plane walks in $S_p^n$.

Finally, by~\eqref{eq:Delta-barxbary}, $\RL{k}{\bar w}{i} - \RL{\bar k}{\bar w}{i} = (p+2)\RL{\lambda}{\bar w}{i}$, and $\RL{\lambda}{\bar w}{0} =0$ and so
$\RL{k}{\bar w}{0} = \RL{\bar k}{\bar w}{0}$.
Remembering~\eqref{eq:phi-preserves-k-bark},
this tells us that $\Phi_p$ and $\Psi_p$, convert $k$ to $\bar k$ and vice versa.
Additionally, since $-1$ (respectively \ldir{3bb}) is the only step that decreases~$k$ (respectively $\bar k$)
and all the other steps increase~$k$ (respectively $\bar k$) by the same amount~$p+2$,
we also obtain that $\Phi_p$ and~$\Psi_p$ take the number of $-1$ steps to the number \ldir{3bb} steps
in the walks and vice versa.

\end{proof}

Many interesting consequences can be read off of the observations we have made so far.  We collect a few in the following proposition.
\begin{proposition}\label{prop:observations} \  
  \begin{itemize}
  \item A lower bound on the minimum $y$-value of the suffixes of the quarter-plane walk is obtained as follows: each time transition~(T2) or transition~(T3) with $p-m-1>\ell$ is used, the lower bound on the minimum for the suffix from after this step to the end is increased at least by $p-q>0$ (for~(T2)) or $p-m-1-\ell>0$ (for~(T3)).
  \item A lower bound on the minimum $x$-value of the suffixes of the quarter-plane walk is obtained as follows: each time transition~(T7) is used, the lower bound on the minimum for the suffix from after this step to the end is increased at least by~$1$.
  \item If a $p$-Łukasiewicz walk~$w$ returns to the $x$-axis after $i$~letters, thus factoring into two $p$-Łukasiewicz walks, write $w=w_1w_2$ for a left factor of length~$i$. Then $\Phi_p(w)$~is the concatenation of $\Phi_p(w_1)$ and $\Phi_p(w_2)$.
  \end{itemize}
\end{proposition}

Note that in the case $p=1$, transition~(T3) does not exist and transition~(T2) becomes~(L1).  Consequenly, the first item of the proposition explains the observation at the end of Example~\ref{ex:bij-ex-motzkin-yamanouchi}.

Yeats \cite[Proposition~2.4]{Yeats-2014-BBC} has given the result analogous to the second point in the six-step model.  The converse to the third point is also true, but this is easier to see from the raising description and so the proof will be postponed to Section~\ref{sec:raising}.

\begin{proof}
  {}From Proposition~\ref{prop inequalities}, we have $\Delta \bar y = \Delta \bar s + \Delta v$.  Intuitively, this means that $\bar y$ and $v$ move in paralell except when $\bar s$ is nonzero.  From the table in Figure~\ref{fig:variations-gen-p} we see that $\bar s$ is nonzero precisely for transition~(T2) and transition~(T3) with $p-m-1>\ell$, and when $\bar s$ is nonzero it is positive.  Finally, by the Łukasiewicz property, $v\geq 0$, therefore $\bar s$ gives a lower bound as described in the first item of the statement.

  Similarly, from Proposition~\ref{prop inequalities}, we have $\Delta \bar x = \Delta \bar r + \Delta \xi$, so intuitively $\bar x$ and $\xi$ move in parallel except with $\bar r$ is nonzero.  From the table in Figure~\ref{fig:variations-gen-p} we see that $\bar r$ is nonzero precisely for transition~(T7) where it has value $1$.  Finally, $\xi$ is a weight on $H$ and so is always nonnegative, therefore $\bar r$ gives a lower bound as described in the second item of the statement.

  For the third point, since $w_1$~is a $p$-Łukasiewicz walk, then as in the proof of Theorem~\ref{thm:bij-Phi-p-Psi-p}, we have $(v,H) = (0,\epsilon)$ after $i$~steps of the LR process for $\Phi_p$.  Consequently, when the process moves beyond length~$i$ it cannot remember anything from the prefix, as the only internal state is $v$ and~$H$, both of which have been returned to their initial situation.  Therefore the same steps are outputted when $\Phi_p$~acts on the steps of~$w_2$ as part of~$w$ as when $\Phi_p$~acts on~$w_2$ as a $p$-Łukasiewicz walk in its own right.  The second point follows.
\end{proof}

\section{Six-step model}\label{sec:six-step}

Let us consider now the 6-step model, that is we return to the case $p=1$ and consider augmenting the step set with the steps' flips through the line $y=x$.

Figure~\ref{fig:phi-psi-sym} provides the forward transitions with their flips for the 6-step model.
They include a first copy
of the rules (U1) to~(DE) of the model for~$p=1$ in the top row and and a second copy of the same rules with and output letters flipped around the main diagonal in the bottom row.  The input letters now come in two colours.
This doubled set of forward transitions is given by the boxes below, where, now:
$\mu$~is from
$\{ {-\boldsymbol{1}}, {\boldsymbol0}, {+\boldsymbol{1}},
     {-\mathbb{1}}, {\mathbb0}, {+\mathbb1} \}$;
\ldir{{$m$}sb}~stands for a letter from the alphabet
$\Sigma_{1,\text{bicol}} := \{\, \ldir{1Lb}, \ldir{2Lb}, \ldir{3Lb},
\ldir{1Rb}, \ldir{2Rb}, \ldir{3Rb} \,\}$;
$(h_L, v_L)$ and~$(h_R, v_R)$ are counters, again from~$\bN^2$;
\ldir{{$\bar s$}sb}~stands for a letter from
$S_{1,\text{sym}} = \{\, \ldir{0bb}, \ldir{6bb}, \ldir{3bb}, \ldir{4bb},
                      \ldir{2bb}, \ldir{7bb} \,\}$.
We will call the steps \ldir{1Lb}, \ldir{2Lb}, \ldir{3Lb} \emph{solid} and the steps \ldir{1Rb}, \ldir{2Rb}, \ldir{3Rb} \emph{striped}, and the walks made of these steps we will call \emph{bicoloured}.

                      Define $\Phi_{1,\text{sym}}$ and $\Psi_{1,\text{sym}}$ by the following transitions as in Section~\ref{sec:3-step-model} or Section~\ref{sec:cons-maps}.

\begin{figure}
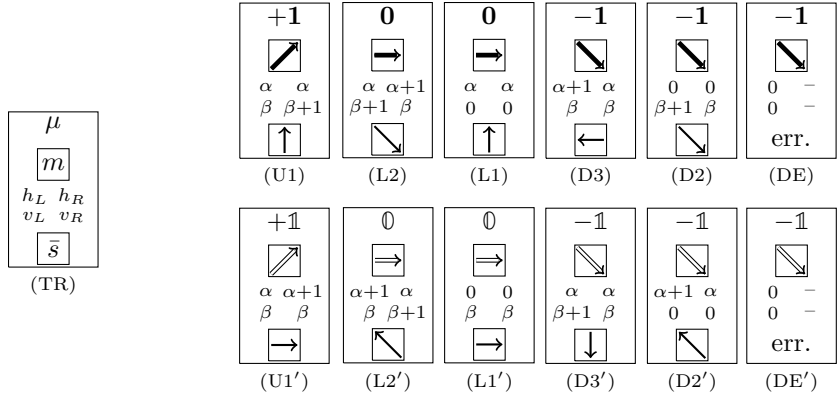

\noindent\ \hfill
\trsVIstep in:$\mu$/{$m$}sb, h:$h_L$->$h_R$, v:$v_L$->$v_R$, out:{$\bar s$}sb, label:(TR).
\ \hfill\
\parbox{.75\textwidth}{%
\noindent\ \hfill
\trsVIstep in:$+\mc1$/1Lb, h:$\alpha$->$\alpha$,     v:$\beta$->$\beta{+}1$, out:0bb,     label:(U1).
\trsVIstep in:$\mc0$/2Lb,  h:$\alpha$->$\alpha{+}1$, v:$\beta{+}1$->$\beta$, out:3bb,     label:(L2).
\trsVIstep in:$\mc0$/2Lb,  h:$\alpha$->$\alpha$,     v:0->0,                 out:0bb,     label:(L1).
\trsVIstep in:$-\mc1$/3Lb, h:$\alpha{+}1$->$\alpha$, v:$\beta$->$\beta$,     out:6bb,     label:(D3).
\trsVIstep in:$-\mc1$/3Lb, h:0->0,                   v:$\beta{+}1$->$\beta$, out:3bb,     label:(D2).
\trsVIstep in:$-\mc1$/3Lb, h:0->--,                  v:0->--,                out:{err.}sn,     label:(DE).
\hfill\

\smallskip

\noindent\ \hfill
\trsVIstep in:$+\mc1$/1Rb, h:$\alpha$->$\alpha{+}1$, v:$\beta$->$\beta$,     out:2bb, label:(U1$'$).
\trsVIstep in:$\mc0$/2Rb,  h:$\alpha{+}1$->$\alpha$, v:$\beta$->$\beta{+}1$, out:7bb, label:(L2$'$).
\trsVIstep in:$\mc0$/2Rb,  h:0->0,                   v:$\beta$->$\beta$,     out:2bb, label:(L1$'$).
\trsVIstep in:$-\mc1$/3Rb, h:$\alpha$->$\alpha$,     v:$\beta{+}1$->$\beta$, out:4bb, label:(D3$'$).
\trsVIstep in:$-\mc1$/3Rb, h:$\alpha{+}1$->$\alpha$, v:0->0,                 out:7bb, label:(D2$'$).
\trsVIstep in:$-\mc1$/3Rb, h:0->--,                  v:0->--,                out:{err.}sn, label:(DE$'$).
\hfill\
}
\caption{\label{fig:phi-psi-sym}
  Transitions used for the 6-step model.
  The first row is a copy of Figure~\ref{fig:phi} with ``solid'' input letters.
  The second row is a copy of Figure~\ref{fig:phi} with ``striped'' input letters, output letters flipped around the main diagonal, and roles of counters exchanged.
}
\end{figure}

Inspecting the transitions we see that the forwards and backwards input sets of the transitions partition the forwards and backwards input spaces and the outputs and parameters lie in the correct sets.  We can define the same parameters as in Section~\ref{sec:invariants} and capture how these parameters change in the table in Figure~\ref{fig:variations-6-step}.

\begin{figure}

\centerline{%
\renewcommand{\arraystretch}{1.1}
\begin{tabular}{c|cccccccccc}
{}          & (U1)&(U1$'$) & (L2)&(L2$'$) & (L1)&(L1$'$) & (D3)&(D3$'$) & (D2)&(D2$'$) \\
\hline
$\Delta \bar x$ &    0&1     &    1&$-1$  &    0&1     & $-1$&0     &    1&$-1$  \\
$\Delta \bar y$ &    1&0     & $-1$&1     &    1&0     &    0&$-1$  & $-1$&1     \\
$\Delta z$        &    1&1     &    0&0     &    0&0     & $-1$&$-1$  & $-1$&$-1$  \\
$\Delta h$        &    0&1     &    1&$-1$  &    0&0     & $-1$&0     &    0&$-1$  \\
$\Delta v$        &    1&0     & $-1$&1     &    0&0     &    0&$-1$  & $-1$&0     \\
$\Delta k$        &    1&1     &    1&1     &    1&1     & $-2$&$-2$  & $-2$&$-2$  \\
$\Delta \bar k$ &    1&$-1$  & $-2$&2     &    1&$-1$  &    1&$-1$  & $-2$&2 \\
$\Delta \bar r$ &    0&0  & 0&0     &    0&1  &    0&0  & 1&0 \\
$\Delta \bar s$ &    0&0  & 0&0     &    1&0  &    0&0  & 0&1
\end{tabular}
}

\caption{\label{fig:variations-6-step}
  Variations in the parameters across the transitions
  defined in Figure~\ref{fig:phi-psi-sym} for the 6-step model.}
\end{figure}

Inspecting the table we see
\begin{gather*}
\Delta z = \Delta h + \Delta v , \\
\Delta \bar x \geq \Delta h \quad\text{and}\quad \Delta \bar x > \Delta h \implies h = 0 , \\
\Delta \bar y \geq \Delta v \quad\text{and}\quad \Delta \bar y > \Delta v \implies v = 0 , \\
\Delta \bar x = \Delta h + \Delta \bar r , \qquad
\Delta \bar y = \Delta v + \Delta \bar s .
\end{gather*}

Then, using these equalities and inequalities, the proof that the maps give a bijection between bicoloured Motzkin paths and quarter-plane walks in $S_{1,\text{sym}}$ proceeds as in the unsymmetrized case.

This bijection answers a question of Bousquet-Mélou and Mishna (see \cite[Section~7.1, item~2]{BousquetMelouMishna-2010-WSS}); this is encapsulated in the following theorem.  The question was first answered in a more complicated way in \cite{Yeats-2014-BBC}.  We will use the notation $\mathcal{M}(n)$ and $\mathcal{M}^{\text{bicol}}(n)$ for the classes of Motzkin and bicoloured Motzkin paths with $n$ steps, $\mathcal{Q}(n)$ for 1-tandem walks with $n$ steps, $\mathcal{Q}^{\text{sym}}(n)$ for quarter-plane walks with $n$ steps each in $S_{1,\text{sym}}$ and $2^{[n]}$ for the powerset of $\{1,2,\ldots, n\}$.
\begin{theorem}\label{thm:2^n-explained}
The map $\Phi_{1, \text{sym}}$ is a length-preserving bijection between bicoloured Motzkin paths and quarter-plane walks with steps in $S_{1,\text{sym}}$. The map $\Psi_{1,\text{sym}}$ is the inverse bijection.

With the notation as described above, the chain of bijections
\[
\mathcal{Q}^{\text{sym}}(n) \xrightarrow{\Psi_{1,\text{sym}}} \mathcal{M}^{\text{bicol}}(n) \cong  \mathcal{M}(n)\times 2^{[n]} \xrightarrow{\Phi_{1} \times \text{id}} \mathcal{Q}(n)\times 2^{[n]}
\]
explains combinatorially why $|\mathcal{Q}^{\text{sym}}(n)| = 2^n|\mathcal{Q}(n)|$.
\end{theorem}

To say the same thing in a slightly different way, if $f$ is the map forgetting the colouring then $\Phi_{1}\circ f\circ\Psi_{1,\text{sym}}$ gives a combinatorially defined $2^n$-to-$1$ map from quarter-plane walks with steps in $S_{1,\text{sym}}$ to quarter-plane walks with steps in $S_{1}$ explaining $|\mathcal{Q}^{\text{sym}}(n)| = 2^n|\mathcal{Q}(n)|$.

Our answer to the question of Bousquet-Mélou and Mishna, while now fully combinatorial, still goes through Motzkin paths as intermediate objects.  We can nevertheless understand some things about the composition directly without this intermediary.  Following the definition of $\Psi_{1,\text{sym}}$ we see that steps $\ldir{0bb}$, $\ldir{3bb}$, and $\ldir{6bb}$ map to solid letters and $\ldir{2bb}$, $\ldir{7bb}$, and $\ldir{4bb}$ map to striped letters.  Carrying solid and striped as ``colours'' though $\Phi_1$ to obtain coloured quarter plane walks in $S_{1}$, we see that output steps which came from input steps in $S_{1}$ have the solid colour while output steps which came from the reflection of $S_{1}$ have the striped colour.  So the colours in the output remember which half of the step set the input steps came from.

From the definition of the maps one can immediately observe that $\Psi_{1,\text{sym}}$ restricted to those 6-step walks which happen to only use steps of $S_1$ agrees with $\Psi_1$ after colouring the output of $\Psi_1$ solid.  Additionally, $\Psi_{1,\text{sym}}$ restricted to those 6-step walks which only use steps not in $S_1$ agrees with $\Psi_1$ applied to the reflection of the walk after colouring the output of $\Psi_1$ striped.  This was part of the initial motivation: we were looking for a generalization that extended $\Psi_1$ and that also extended $\Psi_1$ applied to the reflection for walks purely not using $S_1$ steps.  Returning to the context of the question of question of Bousquet-Mélou and Mishna this implies that $\Phi_{1}\circ f\circ\Psi_{1,\text{sym}}$ is the identity on those walks with only steps in $S_1$ and is reflection through $y=x$ on those walks with no steps in $S_1$.


\section{Alternate interpretation via automata}\label{sec:automata}

It was said in the introduction
that the bijections we introduce in the present work
take their origin in automata theory.
Yet, no explicit automaton is provided
in Sections~\ref{sec:3-step-model}--\ref{sec:six-step},
where we used a notation by “boxes”.
In this section, we show that boxes are just a rephrasing
of the transitions in an automaton,
the only essential difference being that
a single one of our boxes will usually pack a number of transitions of an automaton.
We refer the reader to~%
\cite{AutebertBerstelBoasson-1997-CFL,Lothaire-1997-CW,Sakarovitch-2009-CUP}
for a detailed introduction to automata theory,
and we follow the first of these references for the terminology.
More accurately, our maps
$\phi$, $\psi$, $\Phi_p$, $\Psi_p$, $\Phi_{1,\text{sym}}$, $\Psi_{1,\text{sym}}$
turn out to be the \emph{transductions realized}
by \emph{pushdown transducers}
(see~\cite[Section~5.4]{AutebertBerstelBoasson-1997-CFL}),
with the slight extension that our automata/transducers
use two stacks (not just one).

We now give an explicit description
of the transducer to transform Motzkin walks to tandem walks (case~$p = 1$).
We fix
an input alphabet $A = \Sigma_1 = \{\, \ldir{1bb}, \ldir{2bb}, \ldir{3bb} \,\}$,
an output alphabet $B = S_1 = \{\, \ldir{0bb}, \ldir{6bb}, \ldir{3bb} \,\}$,
a finite set of states $Q = \{q\}$ (in fact a singleton),
a set of accepting states $F = Q$ (that is, we want all states to be accepting),
a stack alphabet common to both stacks $Z_1 = Z_2 = \{o, \iota\}$,
an internal starting configuration
$i = (q, o, o) \in Q \times Z_1^\ast \times Z_2^\ast$
(of the associated pushdown machine),
and a set of internal accepting configurations
$K = {F \times \{\epsilon\} \times \{\epsilon\}} \subset Q \times Z_1^\ast \times Z_2^\ast$.
We will also shortly fix a finite set of transition rules
$\gamma \subset {(A \cup \{\epsilon\})} \times Q \times Z_1 \times Z_1^\ast \times Z_2 \times Z_2^\ast \times Q \times B^\ast$,
after which collecting those sets into the tuple $\mathbb T = (A, Q, Z_1, Z_2, B, i, K, \gamma)$
will describe a pushdown transducer over~$A$.
As is traditional, an automaton is pictorially described
as an oriented graph whose vertices are the elements of~$Q$ (states),
with transition rules given by a labeling of edges.
This labeling takes the form “$s$; $t_1/T_1$, $t_2/T_2$; $s'$”,
which really denotes a 6-tuple
$(s, t_1, T_1, t_2, T_2, s') \in A \times Z_1 \times Z_1^\ast \times Z_2 \times Z_2^\ast \times B$.
In~our transducers,
all possible transitions are from state~$q$ to itself.
A computation by the underlying pushdown machine thus consists
of a succession of transitions.
Each transition operates in a well-defined manner.
If the machine inputs a letter~$s \in A$
when the internal configuration has~$t_1$ as top letter of the first stack
and $t_2$ as top letter of the second stack,
then it replaces
the top letter~$t_1$ by the word~$T_1$
and the top letter~$t_2$ by the word~$T_2$,
before outputting the letter~$s' \in B$.
We will make sure that no stack will ever become empty
and that their contents encode natural integers in unary:
some $m \in \bN$ will correspond to a stack consisting of
an~$o$ (at the bottom), followed by $m$~copies of~$\iota$ (towards the top).
More precisely, the first, respectively second, stack will represent
the counter~$h$, respectively the counter~$v$.
So the variations of counters will be be encoded in one of three forms:
$\iota/\epsilon$ for a decrement by~1;
$o/o$ and $\iota/\iota$ for no variation;
$o/o\iota$ and $\iota/\iota\iota$ for an increment by~1.
This leads to a set~$\gamma$ containing the following 11~transition rules:
\begin{itemize}
\item
  $\ldir{1bb}$; $o/o\iota$, $o/o$; $\ldir{0bb}$, \
  $\ldir{1bb}$; $o/o\iota$, $\iota/\iota$; $\ldir{0bb}$, \
  $\ldir{1bb}$; $\iota/\iota\iota$, $o/o$; $\ldir{0bb}$, \
  $\ldir{1bb}$; $\iota/\iota\iota$, $\iota/\iota$; $\ldir{0bb}$, \
\item
  $\ldir{2bb}$; $o/o$, $o/o$; $\ldir{0bb}$, \
  $\ldir{2bb}$; $o/o$, $\iota/\iota$; $\ldir{0bb}$, \
  $\ldir{2bb}$; $\iota/\epsilon$, $o/o\iota$; $\ldir{3bb}$, \
  $\ldir{2bb}$; $\iota/\epsilon$, $\iota/\iota\iota$; $\ldir{3bb}$, \
\item
  $\ldir{3bb}$; $\iota/\epsilon$, $o/o$; $\ldir{3bb}$, \
  $\ldir{3bb}$; $o/o$, $\iota/\epsilon$; $\ldir{6bb}$, \
  $\ldir{3bb}$; $\iota/\iota$, $\iota/\epsilon$; $\ldir{6bb}$.
\end{itemize}
Observe the absence of a transition rule of the form
  $\ldir{3bb}$; $o/T_1$, $o/T_2$; $s'$,
which would correspond to the box with an error
in Section~\ref{sec:3-step-model}:
not listing any such transition rule makes
a computation on input~$\ldir{3bb}$ and internal configuration $(q, o, o)$
impossible by the underlying pushdown machine.
The reader will easily identify the transduction realized
by the machine just introduced
as the map~$\phi$ of Section~\ref{sec:3-step-model}.

For general~$p$ as well as for the reflected $p=1$ case, we change
some of the ingredients defining the previous transducer~$\mathbb T$.
In the first situation,
the input alphabet, output alphabet, and the alphabet for the first stack are
\begin{gather*}
A = \Sigma_p = \{-1, 0, 1, \dots, p\}, \qquad
B = S_p = \{\, \pdir{p}, \dots, \pdir{m}, \dots, \pdir{0}, \pdir{-} \,\} , \\
Z_1 = \{o\} \cup A_p = \{o\} \cup \{ a_{\ell,m} : 0 \leq \ell+m \leq p-1 \} .
\end{gather*}
In the second situation,
the input and output alphabets are
\begin{gather*}
A = \Sigma_{1,\text{bicol}} = \{\, \ldir{1Lb}, \ldir{2Lb}, \ldir{3Lb},
                                   \ldir{1Rb}, \ldir{2Rb}, \ldir{3Rb} \,\} , \\
B = S_{1,\text{sym}} = \{\, \ldir{0bb}, \ldir{6bb}, \ldir{3bb}, \ldir{4bb},
                            \ldir{2bb}, \ldir{7bb} \,\} ,
\end{gather*}
and the alphabet for the first stack is unchanged.
We will not go here into the details of the corresponding redefinitions
for the sets~$\gamma$ of transition rules.
We only provide one example, the encoding of the rule~\boxref{T6}:
this single box gets expanded into the $(p-1)p/2$ transition rules
$-1$; $a_{\ell,m}/a_{\ell,m+1}$, $\iota/\epsilon$; $-1$,
one for each pair~$(\ell,m)$ satisfying $\ell+m\leq p-2$.

\section{Alternate interpretation via raising}\label{sec:raising}

So far, we have provided descriptions of algorithms by transducers
that map several classes of quarter-plane walks starting at the origin
to related classes of walks in the upper half plane returning to the $x$-axis.
Those maps are described in terms of the alphabet and notation
given in Section~\ref{sec:notation},
which we repeat here for convenience
and to make the present section more independent from the rest.
Specifically:
\begin{itemize}
\item In Section~\ref{sec:3-step-model}, the map~$\psi$ (also denoted~$\Psi_1$)
  converts an input (classical) tandem walk with steps from $S_1 = \{\, \ldir{0bb}, \ldir{6bb}, \ldir{3bb} \,\}$
  into a Motzkin walk using steps from $\Sigma_1 = \{-1,0,+1\} \simeq \{\, \ldir{3bb}, \ldir{2bb}, \ldir{1bb} \,\}$.
\item In Section~\ref{sec:p-tandem}, the map~$\Psi_p$
  converts an input $p$-tandem walk with steps from $S_p = \{\, \pdir{p}, \dots, \pdir{m}, \dots, \pdir{0}, \pdir{-} \,\}$
  into a $p$-Łukasiewicz walk using steps from $\Sigma_p = \{-1,0,1,\dots,p\}$.
\item In Section~\ref{sec:six-step}, the map~$\Psi_{1,\text{sym}}$
  converts an input walk with steps from the symmetrized step set
  $S_{1,\text{sym}} = \{\, \ldir{0bb}, \ldir{6bb}, \ldir{3bb},
                           \ldir{4bb}, \ldir{2bb}, \ldir{7bb} \,\}$
  into a bicoloured Motzkin walk using steps from
  $\Sigma_{1,\text{bicol}} =
  \{ {-\boldsymbol{1}}, {\boldsymbol{0}}, {+\boldsymbol{1}},
     {-\mathbb{1}}, {\mathbb{0}}, {+\mathbb{1}} \}
  \simeq
  \{\, \ldir{1Lb}, \ldir{2Lb}, \ldir{3Lb},
       \ldir{1Rb}, \ldir{2Rb}, \ldir{3Rb} \,\}$,
  with colours that we name “solid” and “striped”.
\end{itemize}
There is another interpretation of our maps, $\psi, \Psi_p, \Psi_{1,\text{sym}}$,
from quarter-plane walks to Motzkin, Łukasiewicz, and bicoloured Motzkin walks.
This other interpretation builds the output half-plane walk
by, at each stage of the construction, adding a new step at the end
and sometimes raising a past step.
In this interpretation,
stopping the process after dealing with a prefix of a quarter-plane walk
results in exactly the action of the full process on that shorter-length walk.
Thus the smaller words constructed along the way will not be prefixes of the ultimately constructed half-plane walk,
and the approach can no longer be implemented in a single pass.
However, it allows us to define the map~$\psi$
by reading the quarter-plane walk from left to right instead of from right to left.
Additionally it has a prefix property
that is reminiscent of the work by Bousquet-Mélou, Fusy, and Raschel%
~\cite{BousquetMelouFusyRaschel-2017-OBC}:
the process really constructs a sequence of half-plane walks
that all return to the $x$-axis,
one for each length up to the total length of the input quarter-plane walk.

In the $p=1$ case and the reflected $p=1$ case,
the same maps as those we constructed
in Sections \ref{sec:3-step-model} and~\ref{sec:six-step}
have been defined (albeit in a difficult-to-read way)
using this raising perspective in~\cite{Yeats-2014-BBC}.
The extension of this raising idea to the $p>1$ case is new
and will be presented in Section~\ref{sec:raising-p-gt-1}.

The decision of which past step to raise requires knowing whether past steps can be raised.
This cannot be determined simply by looking at past steps, as for instance for $p=1$ a past step $0$ could have been raised from $-1$ in which case it is no longer available to be raised or it may have been originally added to the word as a~$0$ in which case it has not yet been raised and is available for raising.
To keep track of this information, over the course of the construction each Motzkin, Łukasiewicz, or bicoloured Motzkin step carries the additional information of whether it can still be raised, and then this additional data is discarded at the end of the construction.
For the $p=1$ and the flipped $p=1$ cases, it is not necessary to carry this extra information provided the algorithm is allowed to look back at previous steps in the quarter-plane walk.
This is discussed further below.

In the $p>1$ case we additionally need to keep track of a pairing between certain steps.  Again this information is discarded at the end of the algorithm.

\subsection{Raising in the 3-step model}\label{sec:tandem-by-raising}

To make the idea clear first in the simplest case, consider the 3-step model one more time.  It is clearer and easier to generalize to $p>1$ if, for the duration of the algorithm, we augment the output alphabet so as to include the raisability of the steps.  Thus for the duration of the algorithm we will have 5 possible letters in the output word.  These will be $\ldir{2bb}$ for a $\ldir{2bb}$ in the Motzkin path that is not raisable and $\ldir{lbb}$ for one that is raisable.  Similarly there will be $\ldir{3bb}$ for a $\ldir{3bb}$ in the Motzkin path that is not raisable and $\ldir{dbb}$ for one that is raisable.  The fifth letter in the alphabet will be $\ldir{1bb}$ since such steps are never raisable.  At the end of the algorithm we forget the extra information and so project onto the smaller alphabet $\{\, \ldir{1bb}, \ldir{2bb}, \ldir{3bb} \,\}$ in order to get a Motzkin path.

\begin{figure}
\begin{ouralgo}
\Input{a quarter-plane walk written as a word $\bar{a}_1\bar{a}_2\dots \bar{a}_n$ in~$S_1^\ast$}
\Output{a Motzkin walk written as a word in~$\Sigma_1^\ast$}
\midalgo
\item Set $M_0:=\epsilon$.
\item For $i$ from~1 to~$n$:
\begin{enumerate}
  \item \label{step up} If $\bar{a}_i = \,\ldir{0bb}$, then $M_i := M_{i-1} \, \ldir{lbb}$.
  \item \label{step down} If $\bar{a}_i = \,\ldir{3bb}$ then
    write $M_{i-1} = w \, \ldir{lbb} \, w'$ where $w'$~is free of~$\ldir{lbb}$,
    and set $M_i := w \, \ldir{1bb} \, w' \, \ldir{dbb}$.
  \item \label{step back} If $\bar{a}_i = \,\ldir{6bb}$ then
    write $M_{i-1} = w \, \ldir{dbb} \, w'$ where $w'$~is free of~$\ldir{dbb}$
    and set $M_i := w \, \ldir{2bb} \, w' \, \ldir{3bb}$.
\end{enumerate}
\item Remove all $\raisabledec$ from ~$M_n$ and return the result.
\end{ouralgo}
\caption{\label{fig:algo-p=1-by-raisings}
  Algorithm by raisings in the 3-step model \cite{Yeats-2014-BBC}.}
\end{figure}

Given a quarter-plane walk in~$S_1^\ast$
written as a sequence of steps $\bar{a}_1\bar{a}_2\dots \bar{a}_n$,
the procedure in Figure~\ref{fig:algo-p=1-by-raisings}
computes a sequence of words, $M_0$~to~$M_n$,
by successively obtaining each~$M_i$ (of length~$i$) from $\bar{a}_i$ and~$M_{i-1}$.
The parsings of $M_{i-1}$ required in \ref{step down} and \ref{step back} are always possible by the quarter-plane property.
Note that steps added in case \ref{step back} cannot be raised and no step can be raised more than once.

In the language of \cite{Yeats-2014-BBC} a step that is raisable is \emph{marked} and a step that is not raisable is \emph{unmarked}.
Alternately, we can determine whether a step can be raised without enlarging the alphabet with $\raisabledec$ by looking at which $\bar{a}_j$ created the step.  This can be done in the following way: Let $m_{j}^{(i)}$ be the $j$th step in $M_i$ (with no $\raisabledec$), then
\begin{itemize}
  \item if $m_{j}^{(i)} = \ldir{2bb}$ and $\bar{a}_j=\ldir{0bb}$ then $m_{j}^{(i)}$ can still be raised while if $m_{j}^{(i)} = \ldir{2bb}$ and $\bar{a}_j=\ldir{3bb}$ then $m_{j}^{(i)}$ has been raised and so can no longer be raised;
  \item if $m_{j}^{(i)} = \ldir{3bb}$ and $\bar{a}_j=\ldir{3bb}$ then $m_{j}^{(i)}$ can still be raised while if $m_{j}^{(i)} = \ldir{3bb}$ and $\bar{a}_j=\ldir{6bb}$ then $m_{j}^{(i)}$ can never be raised; and
  \item if $m_{j}^{(i)} = \ldir{1bb}$ then $\bar{a}_j=\ldir{0bb}$ and $m_{j}^{(i)}$ has been raised and so can no longer be raised.
\end{itemize}

It turns out that this procedure is an alternate realization of the map $\psi$ from Section~\ref{sec:3-step-model}.
See Figure~\ref{fig raising eg} and the related Figure~\ref{fig:early-comparison} for an example.

\begin{figure}
\includegraphics[width=.9\linewidth]{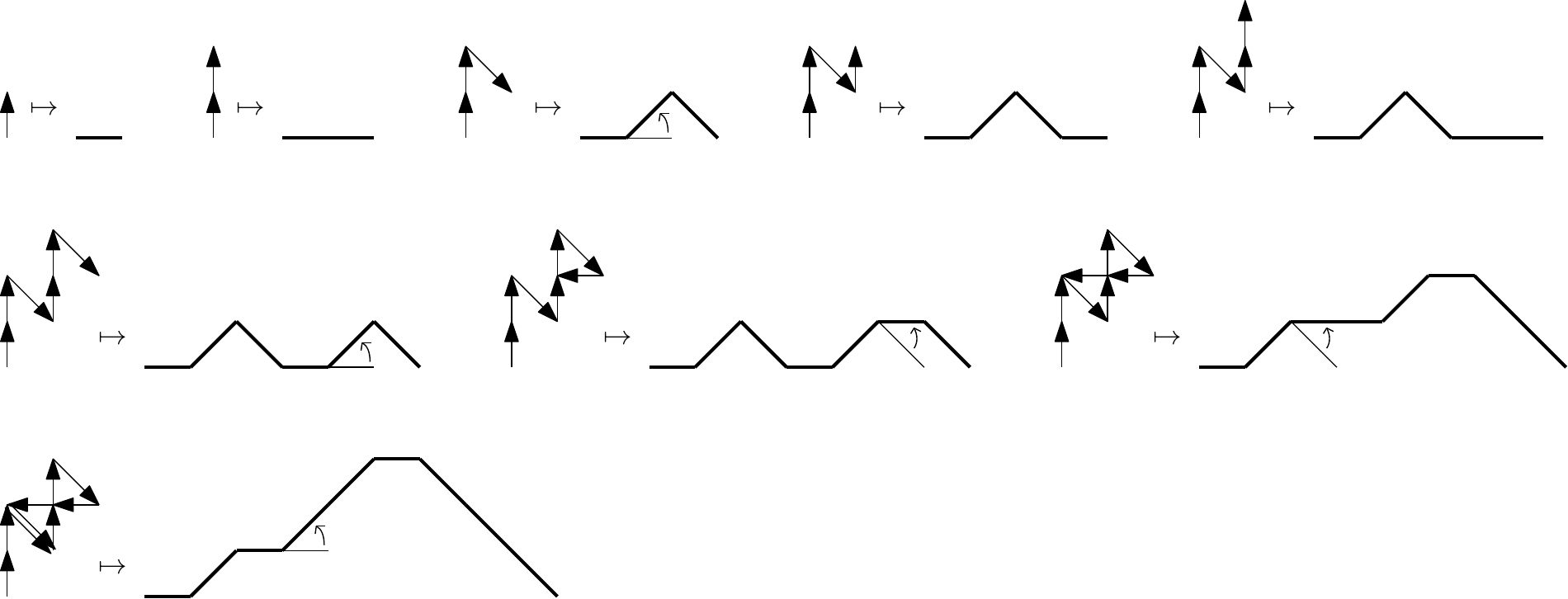}
\caption{\label{fig raising eg}
  An example of the raising process describing the map~$\psi$ on the 3-step model,
  applied to the word $\bar w = {}$\protect\ldir{0bb}\protect\ldir{0bb}\protect\ldir{3bb}\protect\ldir{0bb}\protect\ldir{0bb}\protect\ldir{3bb}\protect\ldir{6bb}\protect\ldir{6bb}\protect\ldir{3bb}.
  Note that we could represent all the Motzkin paths in a manner similar to Figure~\ref{fig:raising-prefixes}:
  in informal words, each Motzkin path is `above' all previous ones.
  Compare with columns (a) and~(b) in Figure~\ref{fig:early-comparison}.}
\end{figure}

It is useful to sketch a proof that the map defined above is $\psi$ because this explains how the idea of raising relates to the transitions that define $\psi$.  To know what step will be in location~$i$ in the final output, it suffices to know what step $\bar{a}_i$ is and to know whether or not the letter first output to this location~$i$ will have to be raised at some later stage.
Whether or not it will be raised is determined by which steps come later in the input, when read forward.
In particular every $\bar{a}_j = \ldir{3bb}$ must raise a step that came from some $\bar{a}_i=\ldir{0bb}$ with $i<j$ and every $\bar{a}_j=\ldir{6bb}$ must raise a step that came from some $\bar{a}_i=\ldir{3bb}$ with $i<j$.  In both cases they raise the most recent such steps.
If we read the quarter-plane walk in the reverse direction,
that is, backwards,
we already have the information about whether $m_i$ will need to be raised or not when we get to it.
Specifically, we only need to count $\ldir{6bb}$ steps and $\ldir{3bb}$ steps (these cause raisings)
and subtract these amounts from the counts for $\ldir{3bb}$ and $\ldir{0bb}$ steps, respectively (these get raised).
This is exactly what $h$ and $v$ are doing in the original definition of $\psi$.  The cases based on $h$ and $v$ give the raised and unraised options for the different input steps.

The proof in \cite{Yeats-2014-BBC} that the map is bijective is both more complicated and more unpleasant than the proofs of this paper, illustrating one of the benefits of our transition-based single pass construction.  None-the-less the raisings can be a helpful reformulation for visualizing the map, and explains how to implement $\psi$ while stepping \emph{forward} through the quarter-plane walk,
rather than \emph{backward} as with the RL transducer algorithm.

\subsection{Freedom of pairing and comparison with Eu, Fu, Hou, and Hsu}\label{sec:EuFuHouHsu}

Observe that knowing whether or not a step will be raised is sufficient to determine the output word, but it is less information than what is given by the raising algorithm.  The raising algorithm additionally pairs each raised step with the step that raised it.  The left-hand side of Figure~\ref{fig different pairings} shows the pairing obtained on the example worked out in Figure~\ref{fig raising eg}.  This leaves open the possibility that other algorithms could be described, which give the same output but a different pairing.
In fact, in~\cite{EuFuHouHsu-2013-SYT}, Eu, Fu, Hou, and Hsu give a description of the bijection by an algorithm giving a different pairing.
See also \cite{Eu-2010-SST}.

Their algorithm (Algorithm~B of \cite[Section~2.2]{EuFuHouHsu-2013-SYT}) from the quarter-plane walk takes three passes each running right to left. This algorithm views the quarter-plane walk as a word in $S_1^\ast$ and modifies it in place to obtain a Motzkin path. Because of the in-place modification, we need disjoint input and output alphabets, so we will write the Motzkin paths as words in $\{-1, 0, 1\}^\ast$ for the purposes of describing this algorithm.
This yields the algorithm in Figure~\ref{fig:algo-Eu-et-al}.

\begin{figure}
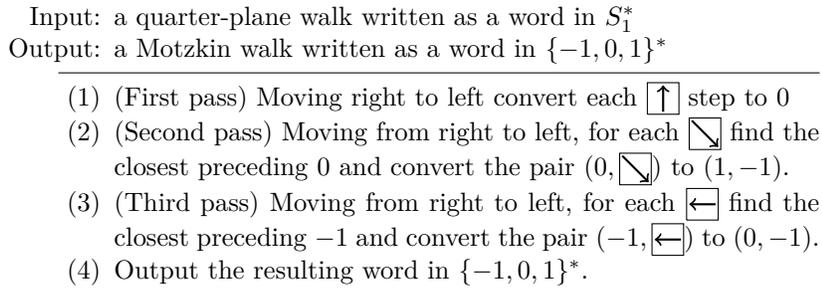

\begin{ouralgo}
\Input{a quarter-plane walk written as a word in~$S_1^\ast$}
\Output{a Motzkin walk written as a word in~$\{-1, 0, 1\}^\ast$}
\midalgo
  \item (First pass) Moving right to left convert each $\ldir{0bb}$ step to $0$
  \item (Second pass) Moving from right to left, for each $\ldir{3bb}$ find the closest preceding $0$ and convert the pair $(0, \ldir{3bb})$ to $(1, -1)$.
  \item (Third pass) Moving from right to left, for each $\ldir{6bb}$ find the closest preceding $-1$ and convert the pair $(-1, \ldir{6bb})$ to $(0, -1)$.
  \item Output the resulting word in $\{-1, 0, 1\}^\ast$.
\end{ouralgo}
\caption{\label{fig:algo-Eu-et-al}
  Algorithm by Eu, Fu, Hou, and Hsu in the 3-step model \cite{EuFuHouHsu-2013-SYT}.}
\end{figure}

As in the raising description, $\ldir{0bb}$ steps become $0$, which may or may not then be raised to $1$ by a later $\ldir{3bb}$ step, $\ldir{3bb}$ steps themselves become $-1$, which may or may not be raised to $0$ by a later $\ldir{6bb}$ step, and $\ldir{6bb}$ steps themselves always become~$-1$.
By making the three separate passes Eu, Fu, Hou, and Hsu's algorithm does not need to distinguish between those $0$ steps that have been raised and those that have not, as first all the unraised $0$s are generated, then in the second pass some are raised, and finally in the last pass the raised $0$s are generated.

Their algorithm also builds the same map $\psi$ because, similarly to the situation in the raising algorithm, to know what step will be in position $i$ in the output it suffices to know what the input step in position $i$ is and whether or not this step will get converted in the second or third pass.  To know whether or not a step will be converted, as we read from right to left through the word, we only need to count $\ldir{3bb}$ steps and $\ldir{6bb}$ steps (these cause conversions in the second and third passes) and subtract from the counts for each $\ldir{0bb}$ and $\ldir{3bb}$ steps, respectively, that we meet (these get converted).  This is again exactly what $h$ and $v$ are doing in the original definition of $\psi$.  However, the particular pairing given by the three passes is different from the one given in the raising description.

To see the different pairing, consider the example of Figure~\ref{fig raising eg}.  As a word the quarter-plane walk is
\[
\ldir{0bb}, \ldir{0bb}, \ldir{3bb}, \ldir{0bb}, \ldir{0bb}, \ldir{3bb}, \ldir{6bb}, \ldir{6bb}, \ldir{3bb}.
\]
After the first pass we have
\[
0, 0, \ldir{3bb}, 0, 0, \ldir{3bb}, \ldir{6bb}, \ldir{6bb}, \ldir{3bb}.
\]
After the second pass we have
\[
0, 1, {-1}, 0, 0, {-1}, \ldir{6bb}, \ldir{6bb}, {-1},
\]
and after the third pass we have
\[
0, 1, 0, 1, 1, 0, {-1}, {-1}, {-1},
\]
where the right-hand side of Figure~\ref{fig different pairings} shows the pairings.  Observe that the pairings are different though the Motzkin path is the same.

\begin{figure}
  \includegraphics[width=.9\linewidth]{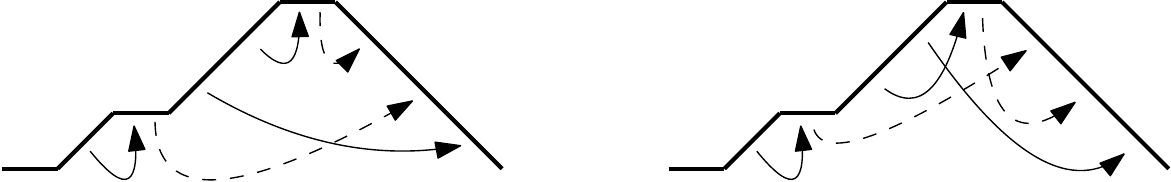}
  \caption{The example of Figure~\ref{fig raising eg} showing the different pairings from the raising algorithm (left) and Eu, Fu, Hou, and Hsu's algorithm (right).  Pairings indicated with solid lines come from \ref{step down} in the raising algorithm or the second pass of Eu, Fu, Hou, and Hsu's algorithm while pairings indicated with dashed lines come from \ref{step back} or the third pass.}\label{fig different pairings}
\end{figure}

The raising algorithm gives the pairing that is noncrossing when restricted to those pairings from step \ref{step down} and when restricted to those from \ref{step back}.  Eu, Fu, Hou, and Hsu's algorithm gives the pairing that takes the closest step from right to left and so is in some sense maximally crossing in each pass.

\subsection{Raising in the six-step model}

Passing to the six-step model from the raising perspective is very much the same, see \cite{Yeats-2014-BBC} for details.  Motzkin steps coming from $\ldir{6bb}$ and $\ldir{4bb}$ can never be raised, but Motzkin steps coming from the other four quarter-plane steps can be raised.  However they fall into two classes (do they go away from the $x$-axis or do they go away from the $y$-axis).  Likewise all the steps except $\ldir{0bb}$ and $\ldir{2bb}$ must cause raisings of earlier steps.  Again they fall into two classes (do they go towards the $x$-axis or do they go towards the $y$-axis).  This corresponds to the fact that in the transition-based description of the bijection for the six-step model we only need to keep track of two parameters rather than the four that one might naively expect from doubling the number of parameters in the three-step case.  Additionally a step raises whichever step in the correct category was most recent with no regard to which kind of step in that category it was.  This means that steps and flipped steps interact.

The algorithmic definition of the map has been given as \cite[Definition~2.1]{Yeats-2014-BBC}.
We reproduce it in Figure~\ref{fig:algo-flipped-by-raisings}.
Again we will, for the duration of the algorithm, augment our output alphabet by $\raisabledec$ to indicate raisability.

Given a quarter-plane walk in~$S_{1,\text{sym}}^\ast$
written as a sequence of steps $\bar{a}_1\bar{a}_2\dots \bar{a}_n$,
the following procedure computes a sequence of words, $M_0$~to~$M_n$,
by successively obtaining each~$M_i$ (of length~$i$) from $\bar{a}_i$ and~$M_{i-1}$:

\begin{figure}
\begin{ouralgo}
\Input{a quarter-plane walk written as a word $\bar{a}_1\bar{a}_2\dots \bar{a}_n$ in~$S_{1,\text{sym}}^\ast$}
\Output{a bicoloured Motzkin walk written as a word in~$\Sigma_{1,\text{bicol}}^\ast$}
\midalgo
\item Set $M_0:=\epsilon$.
\item For $i$ from~1 to~$n$:
\begin{enumerate}
  \item \label{step up2} If $\bar{a}_i = \ldir{0bb}$ then $M_i := M_{i-1} \, \ldir{lLb}$.
  \item \label{step right} If $\bar{a}_i = \ldir{2bb}$ then $M_i := M_{i-1} \, \ldir{lRb}$.
  \item \label{step diagdown} If $\bar{a}_i = \ldir{3bb}$ then either $M_{i-1}=w \, \ldir{lLb} \, w'$ or $M_{i-1} = w \, \ldir{dRb} \, w'$ with $w'$ free of both $\ldir{lLb}$ and $\ldir{dRb}$.  Then set $M_i := w \, \ldir{1Lb} \, w' \, \ldir{dLb}$ if $M_{i-1}=w \, \ldir{lLb} w'$ and $M_i:= w \, \ldir{2Rb} \, w' \, \ldir{dLb}$ otherwise.
  \item \label{step diagup} If $\bar{a}_i = \ldir{7bb}$ then
    either $M_{i-1}=w \, \ldir{lRb} \, w'$ or $M_{i-1} = w \, \ldir{dLb} \, w'$ with $w'$ free of both $\ldir{lRb}$ and $\ldir{dLb}$.  Then set $M_i := w \, \ldir{1Rb} \, w' \ldir{dRb}$ if $M_{i-1}=w \, \ldir{lRb} \, w'$ and $M_i:= w \, \ldir{2Lb} \, w' \, \ldir{dRb}$ otherwise.
  \item \label{step left} If $\bar{a}_i = \ldir{6bb}$ then
    either $M_{i-1}=w \, \ldir{lRb} \, w'$ or $M_{i-1} = w \, \ldir{dLb} \, w'$ with $w'$ free of both $\ldir{lRb}$ and $\ldir{dLb}$.  Then set $M_i := w \, \ldir{1Rb} \, w' \, \ldir{3Lb}$ if $M_{i-1}=w \, \ldir{lRb} \, w'$ and $M_i:= w \, \ldir{2Lb} \, w' \, \ldir{3Lb}$ otherwise.
  \item \label{step down2} If $\bar{a}_i = \ldir{4bb}$ then
    either $M_{i-1}=w \, \ldir{lLb} \, w'$ or $M_{i-1} = w \, \ldir{dRb} \, w'$ with $w'$ free of both $\ldir{lLb}$ and $\ldir{dRb}$.  Then set $M_i := w \, \ldir{1Lb} \, w' \, \ldir{3Rb}$ if $M_{i-1}=w \, \ldir{lLb} \, w'$ and $M_i:= w \, \ldir{2Rb} \, w' \, \ldir{3Rb}$ otherwise.
\end{enumerate}
\item Remove all $\raisabledec$ from~$M_n$ and return the result.
\end{ouralgo}
\caption{\label{fig:algo-flipped-by-raisings}
  Algorithm by raisings in the 6-step model \cite{Yeats-2014-BBC}.}
\end{figure}

In \cite{Yeats-2014-BBC} the bicolouring is given by red and black where solid here is red there and striped here is black there.  The proof that this defines the same map as in Section~\ref{sec:six-step} is analogous to that of the 3-step model.

\subsection{Raising for $p>1$}\label{sec:raising-p-gt-1}

For $p>1$ we can move beyond \cite{Yeats-2014-BBC}
and still give a raising interpretation for $\Psi_p$.
There are additional complications compared to the $p=1$ cases
as steps may be raised more than once,
may be frozen from further raising by later steps,
and may pass back their raising to a previous step
(that is, at a position to its left).
A first, rough description of the process follows,
before we devote the rest of the section to formal details,
culminating with the result that
a raising algorithm (given in Figure~\ref{fig:algo-gen-p-by-raisings})
and the RL transducer algorithm
compute the same object
(Theorem~\ref{thm:raising-RL}).
Reading forwards through the quarter-plane walk,
take a letter~$\bar{a}$ in the latter
to the smallest Łukasiewicz letter~$\mu$
that is possible under some transition for that quarter-plane letter~$\bar{a}$.
Most of the time this will break the Łukasiewicz property.
Then raise enough previous steps to re-obtain the property.
The steps to raise are
recent ones that came from quarter-plane steps
that went away from whichever axis $\bar{a}$~goes towards
and that have not already been maximally raised.

\subsubsection{Raisability and modality of raising}

As for the case~$p=1$, we will keep track of raisability
by augmenting the alphabet used during the run of our algorithm,
and discarding the raisability information at the end of it.
\emph{Raisability} is now an integer parameter~$j\geq0$,
indicating how many times the step can be raised
(each time by~1, as we will see).
However, we need to keep track
of more than just the raisability of a step,
namely we need to keep track of
which kind of future steps it can be raised by
and whether it can be raised itself
or has to pass its raising on to a previous “proxy” step.
We will encode this \emph{modality} of raising
with a second integer parameter $t\geq -1$.
Depending on modality, steps separate into three sorts.
Steps with modality~$t = -1$ can be raised themselves,
but only by one of the $p$~steps of the form $(-m, p-m)$ for~$m\geq1$
(those having some strict west component).
Steps with modality~$t = 0$ can be raised themselves,
but only by a step $(1, -1)$.
Steps with modality~$t \geq 1$ can only be raised by proxy,
by a step $(1, -1)$,
and the value of~$t$ then indicates the distance to the proxy step.
(Observe that, conversely, a step~$(0, p)$ never causes a raising.)

We arrive at an alphabet of letters that we denote~$\raisablelink{\mu}{j}{t}$.
For example, $\raisablelink{0}{4}{-1}$~represents a 0~letter from the Łukasiewicz alphabet that can be raised 4~times, by a step with strict west component and so has no proxy, while $\raisablelink{-1}{1}{3}$~is a letter~$-1$ that can be raised once, by a south-east step and has the letter three positions before it as its proxy.
(Beware that what we will soon define as \emph{a step getting raised}
is marked by the decrease of its raisability~$j$,
and not necessarily by the increase of its nominal value~$\mu$.)

\begin{figure}
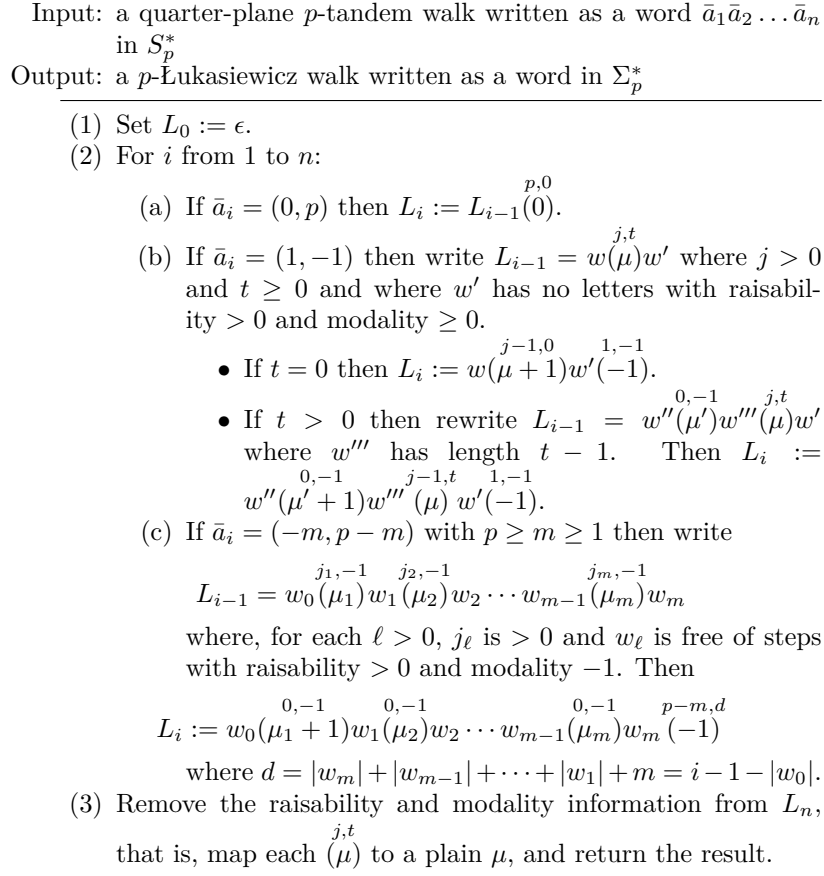

\begin{ouralgo}
\Input{a quarter-plane $p$-tandem walk written as a word $\bar{a}_1\bar{a}_2\dots \bar{a}_n$ in~$S_p^\ast$}
\Output{a $p$-Łukasiewicz walk written as a word in~$\Sigma_p^\ast$}
\midalgo
\item Set $L_0:=\epsilon$.
\item For $i$ from~1 to~$n$:
\begin{enumerate}
  \item \label{step p} If $\bar{a}_i = (0,p)$ then $L_i := L_{i-1}\raisablelink{0}{p}{0}$.
  \item \label{step -1} If $\bar{a}_i = (1,-1)$ then write $L_{i-1} = w \raisablelink{\mu}{j}{t} w'$ where $j>0$ and~$t\geq0$
      and where $w'$~has no letters with raisability~${>0}$ and modality~$\geq0$.
    \begin{itemize}
    \item If $t=0$ then $L_i := w \raisablelink{\mu+1}{j-1}{0} w' \raisablelink{-1}{1}{-1}$.
    \item If $t>0$ then rewrite $L_{i-1} = w'' \raisablelink{\mu'}{0}{-1} w'''  \raisablelink{\mu}{j}{t} w'$ where $w'''$ has length $t-1$.  Then $L_i := w''\raisablelink{\mu'+1}{0}{-1} w'''  \raisablelink{\mu}{j-1}{t} w'\raisablelink{-1}{1}{-1}.$
    \end{itemize}
  \item \label{step m} If $\bar{a}_i = (-m, p-m)$ with $p \geq m \geq 1$ then write
    \[
    L_{i-1} = w_0 \raisablelink{\mu_1}{j_1}{-1} w_1  \raisablelink{\mu_2}{j_2}{-1} w_2\cdots  w_{m-1} \raisablelink{\mu_m}{j_m}{-1} w_m
    \]
    where, for each~$\ell>0$, $j_\ell$~is~$> 0$
    and $w_\ell$~is free of steps with raisability~${>0}$ and modality~$-1$.
    Then
    \[L_{i} :=  w_0 \raisablelink{\mu_1+1}{0}{-1} w_1  \raisablelink{\mu_2}{0}{-1} w_2\cdots  w_{m-1} \raisablelink{\mu_m}{0}{-1} w_m \raisablelink{-1}{p-m}{d}\]
    where $d=|w_m| + |w_{m-1}| + \dots + |w_1| + m=i-1-|w_0|$.
\end{enumerate}
\item Remove the raisability and modality information from~$L_n$,
  that is, map each~$\raisablelink{\mu}{j}{t}$ to a plain~$\mu$,
  and return the result.
\end{ouralgo}
\caption{\label{fig:algo-gen-p-by-raisings}
  Algorithm by raisings for general~$p \geq 1$.}
\end{figure}

Given a quarter-plane walk in $S_p^\ast$
written as a sequence of steps $\bar{a}_1\bar{a}_2\dots \bar{a}_n$,
the procedure in Figure~\ref{fig:algo-gen-p-by-raisings}
computes a sequence of words, $L_0$~to~$L_n$,
by successively obtaining each~$L_i$ (of length~$i$) from $\bar{a}_i$ and~$L_{i-1}$.
Observe that the factorization in the second point of~\ref{step -1}
is always possible because the only case yielding $t>0$, case~\ref{step m},
also places some~$\raisablelink{\mu'}{0}{-1}$ to its left,
which then can only be altered into some~$\raisablelink{\mu''}{0}{-1}$
for~$\mu''\geq\mu'$.
Also, \ref{step m}~can be rephrased as follows: find the $m$ most recent raisable letters with modality~$-1$, raise the $m$th oldest one, and remove the raisability of the $m$~letters for the rest of the process.  Furthermore, the $m-1$ more recent ones are permanently frozen as their raisability is 0 and they cannot be the proxy steps for any future step.
The $m$th most recent of the letters that were found is different.  Its raisability has also been set to $0$ but it is a proxy.
To avoid ambiguities we will use the following language with respect to raising:
\begin{itemize}
\item In cases \ref{step m} and \ref{step -1} we say that the step at position $i$ \emph{causes a raising}.
\item In case \ref{step -1}, in both the $t=0$ and $t>0$ cases we say that the letter~$\raisablelink{\mu}{j}{t}$ \emph{gets raised}, becoming $\raisablelink{\mu+1}{j-1}{t}$ and $\raisablelink{\mu}{j-1}{t}$ respectively, and that the obtained letter \emph{was raised by} the step at position~$i$.  The $t>0$ case is a raising that is passed on to a proxy step, precisely~$\raisablelink{\mu'}{0}{-1}$, but it still counts as a raising for~$\mu$ not for~$\mu'$.
\item In case \ref{step m} we say that the letter $\raisablelink{\mu_1}{j_1}{-1}$ \emph{gets raised}, becoming $\raisablelink{\mu_1+1}{0}{-1}$, and that the obtained letter \emph{was raised by} the step at position~$i$. We say the~$\raisablelink{\mu_\ell}{j_\ell}{-1}$ for $1\leq\ell\leq m$ \emph{get frozen by} the step at position~$i$, and this also does not count as a raising.
\end{itemize}
Note that the raisability of a letter is never increased and only letters with positive raisability can get raised.

\begin{figure}
\begin{tiny}
\renewcommand{\arraystretch}{1.3}
\begin{tabular}{r|c@{\,\,}c@{\,\,}c@{\,\,}c@{\,\,}c@{\,}|@{\,}c@{\,\,}c@{\,\,}c@{\,\,}c@{\,\,}c@{\,}|@{\,}c@{\,\,}c@{\,\,}c@{\,\,}c@{\,\,}c@{\,}|@{\,}c@{\,\,}c@{\,\,}c@{\,\,}c@{\,\,}c@{\,}|@{\,}c}
$\bar w$ & 5  & D  & D  & D  & D  & D  & 3  & D  & D  & 3  & D  & 4  & 5  & D  & D  & D  & D  & 4  & D  & D  & D  \\
         & a  & b1 & b1 & b1 & b1 & b1 & c  & b2 & b2 & c  & b2 & c  & a  & b1 & b1 & b1 & b1 & c  & b2 & b2 & b2 \\
\hline
 1 & \rlk{050} \\
 2 & \rlk{140} & \rlk{D1-} \\
 3 & \rlk{230} & \rlk{D1-} & \rlk{D1-} \\
 4 & \rlk{320} & \rlk{D1-} & \rlk{D1-} & \rlk{D1-} \\
 5 & \rlk{410} & \rlk{D1-} & \rlk{D1-} & \rlk{D1-} & \rlk{D1-} \\
 6 & \rlk{500} & \rlk{D1-} & \rlk{D1-} & \rlk{D1-} & \rlk{D1-} & \rlk{D1-} \\
 7 & \rlk{500} & \rlk{D1-} & \rlk{D1-} & \rlk{D1-} & \rlk{00-} & \rlk{D0-} & \rlk{D32} \\
 8 & \rlk{500} & \rlk{D1-} & \rlk{D1-} & \rlk{D1-} & \rlk{10-} & \rlk{D0-} & \rlk{D22} & \rlk{D1-} \\
 9 & \rlk{500} & \rlk{D1-} & \rlk{D1-} & \rlk{D1-} & \rlk{20-} & \rlk{D0-} & \rlk{D12} & \rlk{D1-} & \rlk{D1-} \\
10 & \rlk{500} & \rlk{D1-} & \rlk{D1-} & \rlk{D1-} & \rlk{20-} & \rlk{D0-} & \rlk{D12} & \rlk{00-} & \rlk{D0-} & \rlk{D32} \\
11 & \rlk{500} & \rlk{D1-} & \rlk{D1-} & \rlk{D1-} & \rlk{20-} & \rlk{D0-} & \rlk{D12} & \rlk{10-} & \rlk{D0-} & \rlk{D22} & \rlk{D1-} \\
12 & \rlk{500} & \rlk{D1-} & \rlk{D1-} & \rlk{D1-} & \rlk{20-} & \rlk{D0-} & \rlk{D12} & \rlk{10-} & \rlk{D0-} & \rlk{D22} & \rlk{00-} & \rlk{D41} \\
13 & \rlk{500} & \rlk{D1-} & \rlk{D1-} & \rlk{D1-} & \rlk{20-} & \rlk{D0-} & \rlk{D12} & \rlk{10-} & \rlk{D0-} & \rlk{D22} & \rlk{00-} & \rlk{D41} & \rlk{050} \\
14 & \rlk{500} & \rlk{D1-} & \rlk{D1-} & \rlk{D1-} & \rlk{20-} & \rlk{D0-} & \rlk{D12} & \rlk{10-} & \rlk{D0-} & \rlk{D22} & \rlk{00-} & \rlk{D41} & \rlk{140} & \rlk{D1-} \\
15 & \rlk{500} & \rlk{D1-} & \rlk{D1-} & \rlk{D1-} & \rlk{20-} & \rlk{D0-} & \rlk{D12} & \rlk{10-} & \rlk{D0-} & \rlk{D22} & \rlk{00-} & \rlk{D41} & \rlk{230} & \rlk{D1-} & \rlk{D1-} \\
16 & \rlk{500} & \rlk{D1-} & \rlk{D1-} & \rlk{D1-} & \rlk{20-} & \rlk{D0-} & \rlk{D12} & \rlk{10-} & \rlk{D0-} & \rlk{D22} & \rlk{00-} & \rlk{D41} & \rlk{320} & \rlk{D1-} & \rlk{D1-} & \rlk{D1-} \\
17 & \rlk{500} & \rlk{D1-} & \rlk{D1-} & \rlk{D1-} & \rlk{20-} & \rlk{D0-} & \rlk{D12} & \rlk{10-} & \rlk{D0-} & \rlk{D22} & \rlk{00-} & \rlk{D41} & \rlk{410} & \rlk{D1-} & \rlk{D1-} & \rlk{D1-} & \rlk{D1-} \\
18 & \rlk{500} & \rlk{D1-} & \rlk{D1-} & \rlk{D1-} & \rlk{20-} & \rlk{D0-} & \rlk{D12} & \rlk{10-} & \rlk{D0-} & \rlk{D22} & \rlk{00-} & \rlk{D41} & \rlk{410} & \rlk{D1-} & \rlk{D1-} & \rlk{D1-} & \rlk{00-} & \rlk{D41} \\
19 & \rlk{500} & \rlk{D1-} & \rlk{D1-} & \rlk{D1-} & \rlk{20-} & \rlk{D0-} & \rlk{D12} & \rlk{10-} & \rlk{D0-} & \rlk{D22} & \rlk{00-} & \rlk{D41} & \rlk{410} & \rlk{D1-} & \rlk{D1-} & \rlk{D1-} & \rlk{10-} & \rlk{D31} & \rlk{D1-} \\
20 & \rlk{500} & \rlk{D1-} & \rlk{D1-} & \rlk{D1-} & \rlk{20-} & \rlk{D0-} & \rlk{D12} & \rlk{10-} & \rlk{D0-} & \rlk{D22} & \rlk{00-} & \rlk{D41} & \rlk{410} & \rlk{D1-} & \rlk{D1-} & \rlk{D1-} & \rlk{20-} & \rlk{D21} & \rlk{D1-} & \rlk{D1-} \\
21 & \rlk{500} & \rlk{D1-} & \rlk{D1-} & \rlk{D1-} & \rlk{20-} & \rlk{D0-} & \rlk{D12} & \rlk{10-} & \rlk{D0-} & \rlk{D22} & \rlk{00-} & \rlk{D41} & \rlk{410} & \rlk{D1-} & \rlk{D1-} & \rlk{D1-} & \rlk{30-} & \rlk{D11} & \rlk{D1-} & \rlk{D1-} & \rlk{D1-} \\
\hline
$w$  & 5  & D  & D  & D  & 2  & D  & D  & 1  & D  & D  & 0  & D  & 4  & D  & D  & D  & 3  & D  & D  & D  & D
\end{tabular}
\end{tiny}
\caption{\label{fig:raising-exec}
  Example of execution of the raising algorithm on the 5-tandem walk~$\bar w$
  of Figure~\ref{fig:bij-ex-p=5} (right).
  The letters~$\bar\mu$ in row~$\bar w$ (input) give the $p$-tandem word~$\bar w$ as the succession of steps~$(\bar\mu-5,\bar\mu)$, or $(1,-1)$~if $\bar\mu={}$D.
  The letters~$\mu$ in row~$w$ (output) give the Łukasiewicz walk~$w$ as the succession of steps~$(1,\mu)$, with D~standing again for~$-1$.
  The second row indicates the case used in the algorithm:
  `a'~for case~\ref{step p}; `b1' and~`b2' for the first and second item of case~\ref{step -1}; `c'~for case~\ref{step m}.
  The rows numbered from~1 to~21 provide the successive Łukasiewicz words~$L_i$ constructed by the algorithm.}
\end{figure}

\begin{figure}
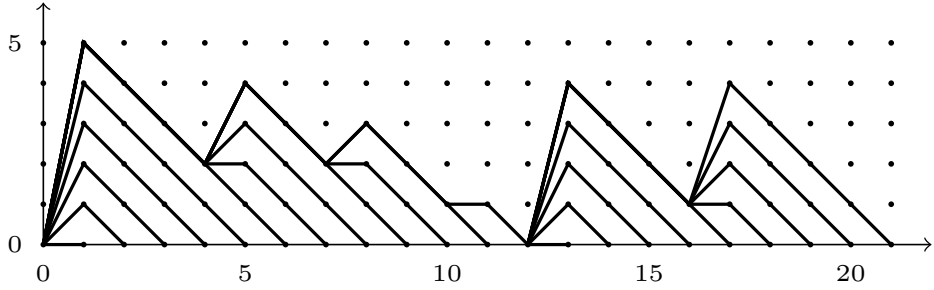

\resizebox{\textwidth}{!}{%
\begin{mypic}
\foreach \x in {0,1,2,3,4,5,6,7,8,9,10,11,12,13,14,15,16,17,18,19,20,21}
  \foreach \y in {0,1,2,3,4,5}
    \fill(\x,\y) circle[radius=2pt];
\draw[->](0,0)--(22,0);
\draw[->](0,0)--(0,6);
%
\draw[thick,-](0,0)--(1,0);
\draw[thick,-](0,0)--(1,1)--(2,0);
\draw[thick,-](0,0)--(1,2)--(2,1)--(3,0);
\draw[thick,-](0,0)--(1,3)--(2,2)--(3,1)--(4,0);
\draw[thick,-](0,0)--(1,4)--(2,3)--(3,2)--(4,1)--(5,0);
\draw[thick,-](0,0)--(1,5)--(2,4)--(3,3)--(4,2)--(5,1)--(6,0);
\draw[thick,-](0,0)--(1,5)--(2,4)--(3,3)--(4,2)--(5,2)--(6,1)--(7,0);
\draw[thick,-](0,0)--(1,5)--(2,4)--(3,3)--(4,2)--(5,3)--(6,2)--(7,1)--(8,0);
\draw[thick,-](0,0)--(1,5)--(2,4)--(3,3)--(4,2)--(5,4)--(6,3)--(7,2)--(8,1)--(9,0);
\draw[thick,-](0,0)--(1,5)--(2,4)--(3,3)--(4,2)--(5,4)--(6,3)--(7,2)--(8,2)--(9,1)--(10,0);
\draw[thick,-](0,0)--(1,5)--(2,4)--(3,3)--(4,2)--(5,4)--(6,3)--(7,2)--(8,3)--(9,2)--(10,1)--(11,0);
\draw[thick,-](0,0)--(1,5)--(2,4)--(3,3)--(4,2)--(5,4)--(6,3)--(7,2)--(8,3)--(9,2)--(10,1)--(11,1)--(12,0);
\draw[thick,-](12,0)--(13,0);
\draw[thick,-](12,0)--(13,1)--(14,0);
\draw[thick,-](12,0)--(13,2)--(14,1)--(15,0);
\draw[thick,-](12,0)--(13,3)--(14,2)--(15,1)--(16,0);
\draw[thick,-](12,0)--(13,4)--(14,3)--(15,2)--(16,1)--(17,0);
\draw[thick,-](12,0)--(13,4)--(14,3)--(15,2)--(16,1)--(17,1)--(18,0);
\draw[thick,-](12,0)--(13,4)--(14,3)--(15,2)--(16,1)--(17,2)--(18,1)--(19,0);
\draw[thick,-](12,0)--(13,4)--(14,3)--(15,2)--(16,1)--(17,3)--(18,2)--(19,1)--(20,0);
\draw[thick,-](12,0)--(13,4)--(14,3)--(15,2)--(16,1)--(17,4)--(18,3)--(19,2)--(20,1)--(21,0);
\foreach \x in {0,5,10,15,20} \draw(\x,-.7) node{\tiny\x};
\foreach \y in {0,5} \draw(-.7,\y) node{\tiny\y};
\end{mypic}
}
\caption{\label{fig:raising-prefixes}
  Joint graphical representation of the successive Łukasiewicz words~$L_i$ constructed by the raising algorithm
  in the run of Figure~\ref{fig:raising-exec}.
  Each walk~$L_i$ is above all~$L_j$ for~$j<i$ and the sequence ends with the Łukasiewicz walk~$w$
  of Figure~\ref{fig:bij-ex-p=5} (left).}
\end{figure}

\begin{example}
We continue Example~\ref{ex:bij-ex-p=5} by providing
in Figures \ref{fig:raising-exec} and~\ref{fig:raising-prefixes}
an example of execution of the raising algorithm
on the same 5-tandem walk~$\bar w$.
The forward process exemplified in Figure~\ref{fig:raising-exec},
with input~$\bar w$ (at the top) and output~$w$ (at the bottom),
should be compared to the RL transducer process
in Figure~\ref{fig:transduction-ex-p=5},
with same input (at the bottom) and output (at the top).
The Łukasiewicz words~$L_i$ constructed
are first shown in there precise form
with raisability and modality in Figure~\ref{fig:raising-exec},
and next as plain walks in Figure~\ref{fig:raising-prefixes}.
\end{example}

To make the connection between the raising algorithm and the transducer algorithm more precise in Theorem~\ref{thm:raising-RL} we need a few lemmas.
We gather simple properties in Figures \ref{fig:x-trichotomy} and~\ref{fig:y-trichotomy},
then give the more elaborate Lemmas \ref{lem upish} and~\ref{lem leftish}.
For the rest of the section, let $\bar{w}$ be a quarter-plane walk in $S_p^*$ with letters $\bar{a}_j$.

Immediate properties of the RL process and of the raising algorithm
can be derived from the two algorithm descriptions.
They are expressed in terms of the nature of steps in the input word~$\bar w$,
by trichotomies according to their $x$-coordinate and $y$-coordinate,
respectively.
These properties are gathered as tables
in Figures \ref{fig:x-trichotomy} and~\ref{fig:y-trichotomy}
(only their last rows are less immediate and will be proved in lemmas).
The trichotomy with respect to~$y$ matches exactly
the case distinction in the raising algorithm
(into \ref{step p}, \ref{step -1}, \ref{step m});
it also matches a partitioning of the transitions.
In contrast, the trichotomy with respect to~$x$ does not directly match
the computational case distinctions in our algorithms,
but further refinements of case~\ref{step m} and of \boxref{T3} and \boxref{T4}
provide a suitable partitioning.
(We could have gathered the two tables by a partitioning into four columns,
at the cost of some repetition.)

\begin{figure}
\centerline{%
\renewcommand{\arraystretch}{1.3}
\begin{tabular}{c|ccc}
\multirow{2}{*}{input step $\bar a_j$} & small step & west long step & other long step \\[-4pt]
 & $(1,-1)$ & $(-p,0)$ & with~$y>0$ \\
\hline
variation in~$v$ & +1 & 0 & $\leq0$, in $\{ -y, \dots, 0 \}$ \\
\multirow{2}{*}{transitions} & \multirow{2}{*}{\boxref{T5}, \boxref{T6}, \boxref{T7}} & \boxref{T3} or \boxref{T4}, & \boxref{T3} or \boxref{T4}, \\[-4pt]
 & & when $m = p-1$ & when $m < p-1$ \\
\multirow{2}{*}{case} & \multirow{2}{*}{\ref{step -1}} & \multirow{2}{*}{\ref{step m} with $m=p$} & \ref{step p} or \\[-4pt]
 & & & \ref{step m} with $m<p$ \\
\multirow{2}{*}{output modality} & \multirow{2}{*}{$-1$} & $>0$ & \multirow{2}{*}{$\geq0$} \\[-4pt]
 & & with raisability~0 & \\
\multirow{2}{*}{role in \ref{step -1}} & \multirow{2}{*}{cause one raising} & \multirow{2}{*}{none} & output may get raised \\[-4pt]
 & & & up to $y$~times
\end{tabular}
}
\caption{\label{fig:x-trichotomy}
  Impact of the $x$-coordinate of input steps~$\bar a_j$ in the flow of the RL transducer process and in the flow of the raising algorithm.
  Transitions have to be understood in reverse form,
  and variation in~$v$ is signed according to the RL flow.}
\end{figure}

\begin{figure}
\centerline{%
\renewcommand{\arraystretch}{1.3}
\begin{tabular}{c|ccc}
\multirow{2}{*}{input step $\bar a_j$} & small step & north long step & other long step \\[-4pt]
 & $(1,-1)$ & $(0,p)$ & with~$x<0$ \\
\hline
\multirow{2}{*}{effect on $H$} &$a_{\ell, m'} \mapsto a_{\ell,m'-1}$,  & \multirow{2}{*}{none} & \multirow{2}{*}{push $a_{\ell, |x|-1}$}\\[-4pt]
&pop $a_{\ell, 0}$, or do nothing & & \\
\multirow{2}{*}{transitions} & \boxref{T6}, \boxref{T5}, or \boxref{T7} & \multirow{2}{*}{\boxref{T1} or \boxref{T2}} & \multirow{2}{*}{\boxref{T3} or \boxref{T4}} \\[-4pt]
 & respectively & &\\
case & \ref{step -1} & \ref{step p} & \ref{step m} \\
output modality & $-1$ & $0$ & $>0$ \\
\multirow{2}{*}{role in \ref{step m}} & output gets frozen, raised & \multirow{2}{*}{none} & cause one raising and \\[-4pt]
  & or is never touched by \ref{step m} & & potentially freezings \\
\end{tabular}
}
\caption{\label{fig:y-trichotomy}
  Impact of the $y$-coordinate of input steps~$\bar a_j$ in the flow of the RL transducer process and in the flow of the raising algorithm.
  Transitions have to be understood in reverse form,
  and alteration of~$H$ is according to the RL flow.}
\end{figure}

\subsubsection{Augmented RL transducer process: a more elaborate stack~$V$}

For the purposes of the first lemma, we will reinterpret the parameter $v$ from the RL transducer algorithm as the length of a new stack~$V$ added to the algorithm.
The letters allowed on~$V$ will be integers representing positions in the input word~$\bar{w}$.
Thus it will make sense to speak of which letter or transition instance put a particular element onto~$V$ and which took it off, and the stack elements will keep track of these indices.
Specifically, for the rest of this section,
we augment the RL process with the stack~$V$
and we modify the reverses of the transitions \boxref{T1} to~\boxref{T7}.
In the augmented RL process,
the reverses of \boxref{T5}, \boxref{T6}, and~\boxref{T7}
push $i$ onto the stack~$V$
whenever those transitions are used at stage~$n+1-i$,
that is, when dealing with the input letter~$\bar{a}_i$.
Moreover,
the reverses of \boxref{T1}, \boxref{T2}, \boxref{T3}, and~\boxref{T4}
pop respectively $p$, $q$, $\ell$, and~$\ell$ elements from the stack~$V$.
(Those numbers are defined in the transition descriptions.)

\begin{example}
We consider again Example~\ref{ex:bij-ex-p=5}
and in particular the evolution of counters along LR and RL transductions
provided by Figure~\ref{fig:transduction-ex-p=5}.
As represented by the row for~$V$, the augmented RL transducer process
successively pushes 21, 20, and~19 to the stack~$V$ by three instances of \boxref{T7},
before an instance of \boxref{T3} pops all three integers, emptying the stack.
Four similar cycles of successive pushes before emptying the stack occur in the rest of the run.
Other examples involving \boxref{T2} would similarly empty the stack;
on the other hand, examples with longer runs of successive pushes
could lead to the use of \boxref{T4} or \boxref{T1},
causing pops of several integers without fully emptying the stack.
\end{example}

\subsubsection{Relation between steps with positive $y$-coordinates, the stack~$V$, and raisings}

\begin{lemma}\label{lem upish}
Consider an index~$j$ in~$\bar{w}$ for which $\bar{a}_j$~has positive $y$-coordinate.
When processing~$\bar{a}_j$, the RL transducer algorithm
either
\begin{itemize}
\item leaves the stack\/~$V$ unchanged, in which case the letter in position~$j$ is never raised in the raising algorithm, or
\item takes off at least one element from it,
in which case the steps (with indices $>j$)
that put on those elements taken off by~$\bar{a}_j$
are precisely the steps that raise the letter in position~$j$
in the raising algorithm.
\end{itemize}
\end{lemma}

\begin{proof}
Consider the trichotomy of input steps given by the $(1,-1)$ steps, the $(-p, 0)$ steps, and the steps with positive $y$-coordinate.
We will now to describe, in relation to this trichotomy,
which transitions in the RL transducer algorithm
and which cases in the raising algorithm deal with the steps.
We will also describe how the steps affect~$V$ in the RL process and are involved in raisings of steps of modality~$\geq 0$ in the raising algorithm.
The table in Figure~\ref{fig:x-trichotomy} captures facts,
all of which can be immediately derived from the two algorithm descriptions,
with the exception of the final row, which requires additional explanation.
This row is addressing how the input steps are involved in raisings of steps with modality $\geq 0$,
which can occur only in case~\ref{step -1} (as a raising requires raisability $>0$).
All $(1,-1)$ steps are dealt with by case~\ref{step -1} and vice versa; each time exactly one raising of a step of modality $\geq 0$ is caused (this raising may or may not involve a proxy, but that is of no concern for this lemma).  The step~$(-p, 0)$ has no involvement with raisings of steps of modality $\geq 0$ as the raisings it causes are on steps of modality $-1$ and its output step is unraisable from its creation.  The steps with positive $y$-coordinate have output steps of modality $\geq 0$ with nonzero raisability: each such output step starts with raisability equal to the $y$ value of the input step and may get raised any number of times between $0$ and this value.  Steps with positive~$y$ do cause raisings but only on steps of modality $-1$ which are not of present concern.

Therefore the input steps that raise steps of nonnegative modality
(that is, those dealt with by case~\ref{step -1})
and the steps that put an element on~$V$
(that is, those that increase~$v$)
are the same, namely the $(1,-1)$ steps.
Furthermore the steps with positive $y$-coordinate are the only ones which may take elements off~$V$ and are also the only ones whose corresponding outputs may get raised by step~\ref{step -1}.

It remains to show that the these classes of steps are associated to each other as described in the statement of the lemma.
This is a parenthesis matching problem.
Let each step with positive $y$-coordinate correspond to as many open parentheses as its $y$-coordinate and each $(1,-1)$ correspond to a close parenthesis.
In this way the input word is associated to a string of parentheses.
Note that there may be more open parentheses in this string than close parentheses.
The final thing we need to prove is that the RL process and the raising process
both induce a natural matching of these parentheses,
both being in fact the same matching.
By remembering what input letter has introduced each parenthesis,
this matching induces an association between input letters,
each step with positive $y$-coordinate
being potentially linked to several $(1,-1)$ more rightwards.
Specifically, this association will be
the one described in the statement of the lemma.

Consider first the RL transducer process.
Reading the string of parentheses from right to left in parallel to the RL process,
a close parenthesis corresponds to adding an element to~$V$, open parentheses to removing elements of~$V$ provided $V$~is not empty yet.
By inspecting the transition rules, we see that in the RL process,
when $v$ decreases it always decreases by as much as it can
(namely by the minimum of its current value and the $y$-coordinate of the input step).
Consequently the parentheses are matched naturally
from right to left in a last in first out manner,
with the additional rule that unmatchable opens are skipped when they are met.
The result of this is the parenthesis matching obtained
by forcing all extra opens to be as leftwards as possible
while retaining the possibility to match all closes.
Furthermore,
the links between input letters naturally induced from the parenthesis matching
are such that any step $\bar{a}_j$ with positive $y$-coordinate is linked to
exactly those input steps in the RL process
that added elements of~$V$ that $\bar{a}_j$~later removes.  The stack $V$ remains unchanged precisely when the linking induced from the parenthesis matching does not link $\bar{a}_j$ to any other steps.

Now consider the raising process.
The reader can check that the raisability of the output letter at position~$j$ after dealing with the input letter~$\bar{a}_i$ is equal to the number of unmatched open parentheses in the group originating from~$\bar{a}_j$ when we consider the left-to-right last in first out parentheses-matching algorithm after all groups originating from $\bar{a}_1$, \dots, $\bar{a}_i$ have been dealt with.  This can be verified by a step-by-step comparison or a straightforward induction.
This gives the same parenthesis matching as the RL transducer process.
Furthermore,
the links between input letters naturally induced from the parenthesis matching
are such that any step $\bar{a}_j$ with positive $y$-coordinate is linked to
exactly those input steps which raised the step at position~$j$.  The step at position~$j$ is never raised precisely when the linking induced from the parenthesis matching does not link $\bar{a}_j$ to any other steps.

The concluding sentences of the previous two paragraphs together prove the lemma.
\end{proof}

\subsubsection{Relation between steps with negative $x$-coordinates, the stack~$H$, and freezings of raisings}

Letters with negative $x$-coordinates cause a push onto $H$ in the RL transducer algorithm and no other letters cause a push onto $H$.  Let $a_{\ell, m-1}$ be a letter that is pushed onto $H$.  Observe that the only way for $a_{\ell, m-1}$ to be taken off the stack is by $m-1$ \boxref{T6} transition instances and then a \boxref{T5} transition instance (not necessarily consecutively).  Note also that in the raising algorithm the case that deals with letters with negative $x$-coordinate is case~\ref{step m}.
These simple observations are refined in the next lemma.

\begin{lemma}\label{lem leftish}
Let $j$ be an index in~$\bar{w}$
for which $\bar{a}_j$~has negative $x$-coordinate,
so that some letter $a_{\ell, m-1}$ is pushed on~$H$
at step~$j$ in the RL process.
Then,
the indices where the $m-1$ \boxref{T6} transition instances occur
are precisely the indices of the letters
that are permanently frozen by case~\ref{step m} for $\bar{a}_j$,
and the index where the \boxref{T5} instance occurs is
the index of the letter raised by~$\bar{a}_j$.

Additionally, indices where \boxref{T7} transition instances occur correspond to indices where the output step in the raising algorithm has modality $-1$ but is never raised or frozen.
\end{lemma}

\begin{proof}
The relevant trichotomy of input letters for this lemma
is based on $x$-coordinate rather than $y$-coordinate.
A table again helps to collect the key information
that can be read off the definitions of the two algorithms,
and is given in Figure~\ref{fig:y-trichotomy}.
The final row addresses how the input steps and their outputs are involved in raisings and freezings of steps of modality $-1$.  These raisings and freezings occur only in case~\ref{step m}.  As we will prove in what follows, the three possibilities in the bottom left entry of the table correspond respectively with the three possibilities in each of the first and second left entries of the table.

  Again we have a parenthesization problem.  Associate every $(1,-1)$ step in the input to an open parenthesis.  Associate every input step with $x$-coordinate ${x<0}$ with~$|x|$ close parentheses.   The input word is again associated to a string of parentheses and again there may be more open parentheses in this string than close parentheses.
The RL process and the raising process give natural matchings of these parentheses,
and the final claim is that they both give the same matching,
and that specifically, the matching is
as described in the statement of the lemma.

  Consider first the RL transducer process.  Reading the string of parentheses from right to left in parallel to the RL process, every time an element $a_{\ell, |x|-1}$ is added to the stack $H$, $|x|$ close parentheses appear in the parenthesis string, and close parentheses can appear in no other way.
Every time the second coordinate of the top of~$H$ is
decreased by~$1$
as well as every time the top of~$H$ is popped,
there is an open parenthesis that matches
the most recent unmatched close parenthesis;
additionally, each instance of transition~\boxref{T7},
which can only occur when $H$~is empty,
introduces one open parenthesis,
which remains unmatched.

We proceed to justify that such a matching is always possible.
Recall that the $m$-weight of~$a_{\ell, m'-1}$ is~$m'$ (see Section~\ref{sec:invariants}).
Thus,
because the RL transducer process makes no other kind of alteration to~$H$,
at any given point in the RL process,
$\wt_m(H)$ is
the number of unmatched close parentheses at the corresponding point
in the parenthesis string.
In particular the matching is always possible.  Also, the open parentheses from \boxref{T7} steps only occur when $H$ is empty, that is, when there are no unmatched close parentheses.

Therefore, the RL transducer algorithm matches the parentheses from right to left in a last in first out manner with the additional rule that unmatchable opens are skipped when they are met.  Furthermore,
the links between input letters naturally induced from the parenthesis matching
are such that any step $\bar{a}_j$ with negative $x$-coordinate is linked to exactly those steps that modify (via \boxref{T6}) or pop (via \boxref{T5}) the element of $H$ that $\bar{a}_j$ pushed.

Now consider the raising process.
The reader can verify by a straightforward induction that for a $(1,-1)$ input step at some position $j$,
the output letter from this step,
after dealing with the input letter~$\bar{a}_i$ for $i\geq j$,
is raisable if and only if the open parenthesis from the $(1,-1)$ step is unmatched when we consider the left-to-right last in first out parentheses-matching algorithm after all parentheses originating from $\bar{a}_1$, \dots, $\bar{a}_i$ have been dealt with.
Thus the raising algorithm matches the parentheses in a left to right, last in first out manner.  This gives the same parenthesis matching as the RL transducer process.  Observe that the unmatched open parentheses correspond to output letters from $(1,-1)$ input steps which are still raisable at the end of the algorithm.  Since such output steps are raisable at most once and freezing also sets the raisability to $0$, these are steps which were never frozen or raised.
Furthermore,
the links between input letters naturally induced from the parenthesis matching
are such that any step $\bar{a}_j$ with negative $x$-coordinate is linked to exactly those steps that were either frozen or raised by~$\bar{a}_j$ by step~\ref{step m}.

\subsubsection{Equivalence between the raising and transducer processes}

The concluding sentences of the previous two paragraphs prove the first statement of the lemma.  The characterization of the unmatched open parentheses in the two algorithms implies the second statement of the lemma.
\end{proof}

\begin{theorem}\label{thm:raising-RL}
  The raising algorithm in Figure~\ref{fig:algo-gen-p-by-raisings}
  and the RL transducer algorithm produce the same output.
\end{theorem}

\begin{proof}
We prove the result for the augmented RL transducer algorithm,
which has the same output as the unmodified RL transducer algorithm.
The proof considers each position~$j$ independently
and goes by a case analysis on the input letter~$\bar{a}_j$:

\begin{itemize}

\item
  Suppose $\bar{a}_j$~is~$(0,p)$.
Case~\ref{step p} stipulates that the first letter output by the raising algorithm at posititon $j$ is $0$ and the second point of case~\ref{step -1} stipulates that when this step is raised its output letter increases by 1 for each raising.
Thus the final output by the raising algorithm at position~$j$ is the number of times the step in position $j$ is raised.
On the other hand, in the RL transducer algorithm, transitions \boxref{T1} and \boxref{T2} together make the output at~$j$ be the number of elements of~$V$ that are removed at step~$j$.
Therefore Lemma~\ref{lem upish} implies that the final output of both algorithms at position~$j$ is the same.

\item
  Suppose $\bar{a}_j$~is~$(1,-1)$.
  We refine the case analysis according the the transition
  used at that position~$j$ by the RL process:

  \begin{itemize}
  \item
If $\bar{a}_j$ triggers a \boxref{T6} instance, then the output letter in this position by the RL process is $-1$ , and by Lemma~\ref{lem leftish} the output letter in this position in the raising process was frozen and so is also $-1$ (by steps~\ref{step -1} and step~\ref{step m}, along with the fact that Lemma~\ref{lem leftish} accounts for all transitions that affect $H$).

  \item
  If $\bar{a}_j$ triggers a \boxref{T5}
then Lemma~\ref{lem leftish} also implies
that the output by the raising algorithm at that position~$j$
is turned into a proxy step after dealing with some position~$i \geq j$.
The step at position~$i$ is the step that originally put on
the stack letter that the \boxref{T5} instance removed.
The output of this \boxref{T5} instance is the value of the first index of this stack letter, but this is the number of elements taken off of~$V$ by~$\bar{a}_i$.
Lemma~\ref{lem upish} implies that this number is the same as the number of steps that raised the output letter at position~$i$.
Hence by the nature of proxy raising,
this number is the final output value at position~$j$.
Therefore both algorithms also agree in this case.

  \item If $\bar{a}_j$ triggers a~\boxref{T7} then the RL transducer algorithm outputs a $-1$; in the raising algorithm, the output step at position $j$ first appears as a $-1$ by case~\ref{step -1}, while by Lemma~\ref{lem leftish} it is never raised, and so it is $-1$ in the final output.
  \end{itemize}

\item
Suppose $\bar{a}_j$~has a negative $x$-coordinate.
All these cases yield~$-1$ as output steps in both algorithms.

\end{itemize}

Therefore, the description of~$\Psi_p$ by raisings agrees
with the transducer description.
\end{proof}

Overall, the raising description is more complicated than the transition-based description of $\Psi_p$.

This is a good moment to return to Proposition~\ref{prop:observations}.
The (T2) and~(T3) steps with $p-m-1>\ell$ described in the first point of the proposition are the steps that could be raised but do not get raised by the end of the raising algorithm.  This gives an intuitive explanation for why they correspond to an increased bound on the abscissa for the quarter-plane walk: if the quarter-plane walk were extended to return to lower $x$-values, then sufficiently many steps must be present in the $p$-Łukasiewicz walk that can still be raised to correspond to these extension steps.  The extra capacity for raising corresponds to the bound on the abscissa, in the precise manner given by Proposition~\ref{prop:observations}.

Additionally, the converse of the third item of Proposition~\ref{prop:observations} is also true.  To see this, suppose the quarter-plane walk~$\bar{w}$ is a concatenation of $\bar{w}_1$ and~$\bar{w}_2$.  By the raising description of~$\Psi_p$, once all the steps of~$\bar{w}_1$ have been processed we have $\Psi_p(\bar{w}_1)$.  Since $\bar{w}_2$~is also a quarter-plane walk, $\Psi_p$~applied to~$\bar{w}_2$ is well-defined.  This implies that whenever raising occurs while working through~$\bar{w}_2$, the steps to be raised must remain within~$\bar{w}_2$.  Therefore the same holds with $\Psi_p$ applied to~$\bar{w}_2$ as it occurs within~$\bar{w}$.  Therefore $\Psi_p(\bar{w})$~is the concatenation of $\Psi_p(\bar{w}_1)$ and~$\Psi_p(\bar{w}_2)$, each of which are $p$-Łukasiewicz walks.  Hence there is a return to the $x$-axis at the concatenation point, giving the converse to the second point of Proposition~\ref{prop:observations}.

\section{Discussion}

\subsection{Comparison to the approach by oriented planar maps}
\label{sec:conclusions}

The work by Bousquet-Mélou, Fusy, and Raschel%
~\cite{BousquetMelouFusyRaschel-2017-OBC}
studies a larger class of generalized tandem walks,
with (infinite) step set
\[ \{\, \ldir{3bb} \,\} \cup \bigcup_{p\geq1} S_p , \]
and it provides a length-preserving involution~$\sigma$ of this class
that also preserves the number of \ldir{3bb} steps
and, for each~$p$, the total number of steps in~$S_p$.
More importantly, this involution exchanges the differences
$a = x_{\text{start}} - x_{\text{min}}$
and $d = y_{\text{end}} - y_{\text{min}}$
and preserves the differences
$b = y_{\text{start}} - y_{\text{min}}$
and $c = x_{\text{end}} - x_{\text{min}}$
(with obvious notation).
By a suitable restriction, namely,
forcing the number of steps from the~$S_p$ to be~0 except for a single~$p$
and considering
walks with $a = d = 0$ on the one hand (half plane),
and walks with $a = b = 0$ on the other hand (quarter plane),
this induces a bijection between $\mathcal L_p$ and~$\mathcal T_p$.
In particular, the involution~$\sigma$ maps
each prefix of a given~$\bar w \in \mathcal T_p$
to a word of~$\mathcal L_p$.
This makes it impossible for the sequence of images in~$\mathcal L_p$
to be the prefixes of the image of the full word~$\bar w$.
Thus, $\sigma$~breaks the principal requirement of our construction,
and our bijections $\Phi_p$ and~$\Psi_p$
are no restrictions of the involution~$\sigma$.
Further, our bijection preserves $y_{\text{end}} - x_{\text{end}}$,
which has no clear behaviour under~$\sigma$.

It remains of interest to determine
if $\sigma$~could be implemented
by an automaton similar to those we introduced.
In this regard,
it is interesting to note that
the quantities exchanged by the involution~$\sigma$
involve the records $x_{\text{min}}$ and~$y_{\text{min}}$,
which cannot be modeled without extending the notion of parameters
that we introduced in Section~\ref{sec:parameters}:
one needs to augment the “arithmetic” of parameters to allow comparisons.

\subsection{Symmetrization beyond $p=1$}

The generalization from the $p=1$ case to the symmetrized $p=1$ case is surprisingly nice; we keep track of the very same counters and simply slot in the extra transitions.  One might have expected a need for four counters, doubling counters as well as steps, but two suffices.  Having made this generalization the obvious next step would be to generalize $p>1$ to some sort of symmetrized situation.  Unfortunately, we do not know how to do this.

Take $S_{p, \text{sym}}$ to be the union of $S_p$ and the reflection of $S_p$ in the line $y=x$.  The first question is what sets of walks ought be in bijection.  We would expect the generalized $\Psi_{p, \text{sym}}$ to have quarter-plane walks in $S_{p, \text{sym}}^\ast$ as its domain.  The most obvious answers for the image of $\Psi_{p, \text{sym}}$ would be either bicoloured Łukasiewicz walks or half-plane walks ending on the $x$-axis in $S_{p, \text{sym}}^\ast$.  The latter suggestion is in view of the bijection between Łukasiewicz walks and half-plane walks ending on the $x$-axis in $S_p^\ast$.  However, direct counting of small walks in these sets show that neither of these are equinumerous with quarter-plane walks in $S_{p, \text{sym}}^\ast$.

Leaving this problem aside for the moment, we could try to generalize the transitions in the hopes that the image itself will become clear in due time.  Analogously to the relationship between the $p=1$ transitions and the symmetrized $p=1$ transitions, we would expect that restricting to only the flipped steps as input should give us something equivalent to the original map but with differently coloured output steps and with the parameters swapped.  Thus we expect that transitions for a $\Psi_{p,\text{sym}}$ map would have two stacks as parameters and would specialize to $\Psi_p$ in two ways.  Unfortunately, neither considering the symmetries of the transitions nor working through many small examples have shown how this could be done.

Consequently, we are left without a clear idea of how to generalize our maps to a symmetrized $p>1$ situation.

\subsection{Standard Young tableaux of bounded height}

In their work~\cite{EuFuHouHsu-2013-SYT}, Eu, Fu, Hou, and Hsu gave
an algorithm for the bijection when~$p=1$ that is similar in spirit to the algorithm by raising.
See our comments in Section~\ref{sec:EuFuHouHsu}.
This bijection,
between (usual) Motzkin walks and Yamanouchi words as we described it in Section~\ref{sec:3-step-model},
really is a bijection between Motzkin walks and standard Young tableaux of height at most~$3$.
But the true contribution of~\cite{EuFuHouHsu-2013-SYT} was to generalize it to a bijection
between standard Young tableaux of height at most~$d$
and a class of coloured Motzkin walks that they introduced.
Their “$d$-Motzkin walks” are Motzkin walks whose up and down steps
may appear in $d$-colours,
with an additional constraint between the occurrences of the colours
that they describe by inequalities.
The bijection works by treating and removing the colours one after the other.
This makes it very tempting to try and apply the transducer approach to $d$-Motzkin walks.
Our attempts have failed so far.
If one conceives the transducer approach applied to (usual) Motzkin walks
as a means to reschedule the raising transformations
so that they can be applied as soon as a raised step is considered,
instead of postponing to when the process needs it to restore a Motzkin property,
then a transducer algorithm for bounded height
would develop such a rescheduling for each colour,
and the $d$~transducers thus obtained would then have to be merged into a single one,
leading to further rescheduling.
It would be great to see such an algorithm.

\printbibliography


\end{document}